\documentclass[english]{amsart}

\usepackage{esint}
\usepackage[svgnames]{xcolor} 
\usepackage[colorlinks,citecolor=red,pagebackref,hypertexnames=false,breaklinks]{hyperref}
\usepackage{pgf,tikz}
\usepackage{pdfsync}

\usepackage{dsfont}
\usepackage{url}
\usepackage[utf8]{inputenc}
\usepackage[T1]{fontenc}
\usepackage{lmodern}
\usepackage{babel}
\usepackage{mathtools}  
\usepackage{amssymb}
\usepackage{lipsum}
\usepackage{mathrsfs}
\usepackage{color}
\usepackage{stmaryrd}
\usepackage{soul}

\newtheorem{theorem}{Theorem}[section]
\newtheorem{proposition}{Proposition}[section]
\newtheorem{lemma}{Lemma}[section]

\newtheorem{corollary}{Corollary}[section]

\numberwithin{equation}{section}

\title[Quantitative uniqueness of continuation]{Quantitative uniqueness of continuation for the Schr\"odinger equation : explicit dependence on the potential}

\author[Mourad Choulli]{Mourad Choulli}
\address{Universit\'e de Lorraine}
\email{mourad.choulli@univ-lorraine.fr}

\author{Hiroshi Takase}
\address{Institute of Mathematics for Industry, Kyushu University, 744 Motooka, Nishi-ku, Fukuoka 819-0395, Japan}
\email{htakase@imi.kyushu-u.ac.jp}

\begin{document}

\begin{abstract}
We demonstrate a quantitative version of the usual properties related to unique continuation from an interior datum for the Schr\"odinger equation with bounded or unbounded potential. The inequalities we establish have constants that explicitly depend on the potential. We also indicate how the above-mentioned inequalities can be extended to elliptic equations with bounded or unbounded first-order derivatives. The case of unique continuation from Cauchy data is also considered.
\end{abstract}



\maketitle

\tableofcontents

\section{Introduction}\label{S1}

We establish a quantitative version of some properties related to the unique continuation for the Schrödinger operator $-\Delta+V$ in the case where $V$ is unbounded. More precisely, we obtain inequalities whose constants depend explicitly on $V$. We find only a few references in the literature dealing with this type of results \cite{CET} and \cite{Da,DZ1,DZ2}. These latter references contain the quantitative vanishing order of zeroes of solutions of elliptic equations. We will comment in more detail on \cite{Da,DZ1,DZ2} in Section \ref{S8}.

The results we prove are based on a new Carleman inequality established in \cite{CET,DET} for a potential in $L^s$ with $s\in (n/2,\infty]$. While the case of a potential in $L^{n/2}$ uses a well-known Carleman inequality proven in \cite{JK}. The main ingredients of our analysis rely on three-ball inequalities that we obtain by using the above-mentioned Carleman inequalities and Cacciopoli-type inequalities that we establish in Section \ref{S2} and Section \ref{S3}, respectively. Using these inequalities, we adapt existing methods to prove quantitative unique continuation, global unique continuation from an interior datum. This is done in Sections \ref{S4}, \ref{S5} and \ref{S6}. The proofs we present here improve and clarify some previous ones. The so-called doubling inequalities are established in Section \ref{S7}  then used in Section \ref{S8} to obtain a quantitative vanishing order of zeroes of solutions of the Schr\"ondiger equation $(-\Delta +V)u=0$. In Section \ref{S9} we give an alternative proof to the one used in Section \ref{S8} for the case of continuous potentials. In Section \ref{S10} we give the   modifications necessary to extend the results for the equation  $(-\Delta +V)u=0$ with $V$ in $L^s$ with $s\in [n/2,\infty]$ to the equation $(-\Delta +W\cdot \nabla +V)u=0$ with $V$ in $L^s$ with $s\in [n/2,\infty]$ and $W\in L^m$ with $m\in [n,\infty]$. Finally, in Section \ref{S11} we treat the case of unique continuation from Cauchy data for the equation $(-\Delta +W\cdot \nabla +V)u=0$ with $V$ in $L^s$ with $s\in [n,\infty]$ and $W\in L^\infty$.

Throughout this text, $n\ge 3$ is an integer and $\Omega$ is a bounded Lipschitz domain of $\mathbb{R}^n$ with boundary $\Gamma$. For further use, we recall some properties of the Lipschitz bounded domain $\Omega$. To this end, let $p:=2n/(n+2)$ and $p':=2n/(n-2)$. Note that $p'$ is the conjugate of $p$,  which means that $1/p+1/p'=1$. In this case, the embeddings  $W^{2,p}(\Omega) \hookrightarrow H^1(\Omega)$ and $H^1(\Omega)\hookrightarrow L^{p'}(\Omega)$ are continuous (e.g. \cite[Theorem 5.4]{Ad}). Furthermore, there exists $0<\mathfrak{r}=\mathfrak{r}(\Omega)\le 1$ such that for all $0<r<\mathfrak{r}$
\[
\Omega^r:=\{x\in \Omega;\; \mathrm{dist}(x,\mathbb{R}^n\setminus \Omega)>r\}
\]
is nonempty and connected (see \cite[Corollary 4.3]{ER}). It is worth observing that if $\Omega$ is convex that so is $\Omega^r$ provided $r>0$ is sufficiently small (e.g. \cite[Lemma 13.38]{Le}).

In the following, for $s\in [n/2,\infty]$ and  $V\in L^s(\Omega)$, we use the notation
\[
\kappa_V:=\|V\|_{L^s(\Omega)}.
\]

Also, in the rest of this text, $B_r:=B(0,r)$ $r>0$, the ball of radius $r$ centered at the origin, and
\[
[\![a,b]\!]:=\{x\in \mathbb{R}^n;\; a<|x|<b\},\quad 0\le a<b.
\]

Let
\begin{equation}\label{sigma}
\sigma:=\sup \{ \|w\|_{L^{p'}(\Omega)};\; w\in H_0^1(\Omega),\;  \|\nabla w\|_{L^2(\Omega)}=1\}.
\end{equation}
We verify that we have
\[
 \|w\|_{L^{p'}(\Omega)}\le \sigma \|\nabla w\|_{L^2(\Omega)},\quad w\in H_0^1(\Omega).
\] 

In the following, we will often use the fact that if $V\in L^s(\Omega)$ with $s\in [n/2,\infty]$  and $u\in L^{p'}(\Omega)$ then by H\"older's inequality $Vu\in L^p(\Omega)$ and
\[
\|Vu\|_{L^p(\Omega)}\le \|V\|_{L^{n/2}(\Omega)}\|u\|_{L^{p'}(\Omega)}.
\]
In particular, if $V\in L^s(\Omega)$ with $s\in [n/2,\infty]$  and $u\in W^{2,p}(\Omega)$ then $(-\Delta+V)u\in L^p(\Omega)$.

\section{Caccioppoli-type inequalities}\label{S2}

The following technical lemma will be useful below.

\begin{lemma}\label{lemqu1}
Let $s\in (n/2,\infty)$ and $V\in L^s(\Omega)$. For all $t>0$, $V$ can be decomposed into two terms $V=V_t^1+V_t^2$ so that $V_t^1\in L^{n/2}(\Omega)$, $V_t^2\in L^\infty (\Omega)$ and
\begin{equation}\label{qu4}
\|V_t^1\|_{L^{n/2}(\Omega)}\le t^{-(2s/n-1)}\kappa_V^{2s/n} ,\quad \|V_t^2\|_{L^\infty(\Omega)}\le t.
\end{equation}
\end{lemma}

\begin{proof}
Define for $t>0$
\[
V_t^1:=V\chi_{\{|V|>t\}},\quad V_t^2:=V\chi_{\{|V|\le t\}}.
\]
We have
\[
\int_\Omega |V_t^1|^{n/2}dx=\int_\Omega |V|^s |V\chi_{\{|V|>t\}}|^{n/2-s}dx\le t^{n/2-s}\int_\Omega |V|^sdx.
\]
Whence $V_t^1\in L^{n/2}(\Omega)$ and
\begin{equation}\label{qu5}
\|V_t^1\|_{L^{n/2}(\Omega)}\le t^{-(2s/n-1)}\|V\|_{L^s(\Omega)}^{2s/n}=t^{-(2s/n-1)}\kappa_V^{2s/n}.
\end{equation}
On the other hand, we have clearly $V_2^t\in L^\infty (\Omega)$ and
\begin{equation}\label{qu6}
\|V_t^2\|_{L^\infty(\Omega)}\le t.
\end{equation}
Putting together \eqref{qu5} and \eqref{qu6}, we obtain \eqref{qu4}. The proof is then complete.
\end{proof}

\begin{proposition}\label{proqu1}
Let $s\in (n/2,\infty)$, $\omega_0\Subset \omega_1\Subset \Omega$ and $d=\mathrm{dist}(\omega_0,\partial \omega_1)$. There exists a constant $\mathfrak{c}=\mathfrak{c}(n,\Omega,s)>0$ such that for all $u\in W^{2,p}(\Omega)$ and $V\in L^s(\Omega)$ we have
\begin{equation}\label{ca1}
\mathfrak{c}\|\nabla u\|_{L^2(\omega_0)}\le \|(-\Delta+V) u\|_{L^p(\Omega)} 
+\left(\kappa_V^{s/(2s-n)}+d^{-1}\right)\|u\|_{L^2(\omega_1)}.\nonumber
\end{equation}
\end{proposition}

\begin{proof}
Pick $\chi \in C_0^\infty (\omega_1)$ satisfying  $0\le \chi \le 1$, $\chi =1$ in a neighborhood of $\omega_0$ and $|\partial^\alpha\chi |\le \mathbf{k}d^{-1}$ for each $|\alpha|=1$, where $\mathbf{k}>0$ is a universal constant.

Let $V\in L^s(\Omega)$, $u\in W^{2,p}(\Omega)$ and $\epsilon>0$. For all $t>0$, we write $V=V_t^1+V_t^2$, where $V_t^1$ and $V_t^2$ are given by Lemma \ref{lemqu1}. In light of H\"older's inequality, we obtain
\[
\left| \int_\Omega \chi^2 u(\Delta-V_t^1) udx\right| \le \|\chi(\Delta-V_t^1) u\|_{L^p(\Omega)}\|\chi u\|_{L^{p'}(\Omega)}
\]
and hence
\begin{align*}
\left| \int_\Omega \chi^2 u(\Delta-V_t^1) udx\right| &\le (2\epsilon)^{-1}\|\chi(\Delta-V_t^1) u\|_{L^p(\Omega)}^2+2^{-1}\epsilon\|\chi u\|_{L^{p'}(\Omega)}^2
\\
&\le (2\epsilon)^{-1}\|\chi(\Delta-V_t^1) u\|_{L^p(\Omega)}^2+2^{-1}\epsilon\sigma^2\|\nabla (\chi u)\|_{L^2(\Omega)}^2.
\end{align*}
That is we have
\begin{align}
\left| \int_\Omega\chi^2 u(\Delta-V_t^1) udx\right| \le (2\epsilon)^{-1}&\|\chi(\Delta-V_t^1) u\|_{L^p(\Omega)}^2\label{ca2}
\\
&+\epsilon\sigma^2\left(\|u\nabla \chi\|_{L^2(\Omega)}^2+\|\chi \nabla u\|_{L^2(\Omega)}^2\right).\nonumber
\end{align}

On the other hand,  an integration by parts gives
\[
\int_\Omega\chi^2 u\Delta udx=-\int_\Omega\chi^2|\nabla u|^2dx-\int_\Omega u\nabla \chi^2 \cdot \nabla udx.
\]
Thus, 
\[
\|\chi \nabla u\|_{L^2(\Omega)}^2\le \left|\int_\Omega\chi^2 u\Delta udx\right|+2\left|\int_\Omega u\chi \nabla \chi \cdot \nabla udx\right|.
\]
Applying Cauchy-Schwarz's inequality and then a convexity inequality to the second term of the right hand side of inequality above, we obtain
\[
\|\chi \nabla u\|_{L^2(\Omega)}^2\le \left|\int_\Omega \chi^2 u\Delta udx\right|+\epsilon^{-1}\|u\nabla \chi\|_{L^2(\Omega)}^2+\epsilon\|\chi \nabla u\|_{L^2(\Omega)}^2.
\]
Whence, 
\[
(1-\epsilon)\|\chi \nabla u\|_{L^2(\Omega)}^2\le \left|\int_\Omega\chi^2 u\Delta udx\right|+\epsilon^{-1}\|u\nabla \chi\|_{L^2(\Omega)}^2,
\]
from which we obtain
\begin{align}
&(1-\epsilon)\|\chi \nabla u\|_{L^2(\Omega)}^2\le \left|\int_\Omega \chi^2 u(\Delta-V_t^1) udx\right| \label{ca3}
\\
&\hskip 4cm +\left|\int_\Omega V_t^1\chi^2 u^2dx\right|+\epsilon^{-1}\|u\nabla \chi\|_{L^2(\Omega)}^2.\nonumber
\end{align}
Next, applying twice H\"older's inequality, we obtain
\[
\left|\int_\Omega V_t^1\chi^2 u^2dx\right|\le \|V_t^1\chi u\|_{L^p(\Omega)}\|\chi u\|_{L^{p'}(\Omega)}\le \|V_t^1\|_{L^{n/2}(\Omega)}\|\chi u\|_{L^{p'}(\Omega)}^2.
\]
Using \eqref{qu4}, we obtain
\begin{align*}
\left|\int_\Omega V_t^1\chi^2 u^2dx\right|&\le  \|V_t^1\|_{L^{n/2}(\Omega)}\sigma^2\|\nabla (\chi u)\|_{L^2(\Omega)}^2
\\
& \le  2t^{-(2s/n-1)}\kappa_V^{2s/n}\sigma^2\left(\|u\nabla \chi \|_{L^2(\Omega)}^2+\|\chi \nabla u\|_{L^2(\Omega)}^2\right).
\end{align*}
This inequality in \eqref{ca3} gives
\begin{align}
&(1-\epsilon-2t^{-(2s/n-1)}\kappa_V^{2s/n}\sigma^2)\|\chi \nabla u\|_{L^2(\Omega)}^2\le \left|\int_\Omega\chi^2 u(\Delta-V_t^1) udx\right| \label{ca4}
\\
&\hskip 4.5cm +(\epsilon^{-1}+2t^{-(2s/n-1)}\kappa_V^{2s/n}\sigma^2)\|u\nabla \chi\|_{L^2(\Omega)}^2.\nonumber
\end{align}
By taking $t=[2\kappa_V^{s/n}\sigma]^{2n/(2s-n)}$ in \eqref{ca4}, we find 
\[
(1/2-\epsilon)\|\chi \nabla u\|_{L^2(\Omega)}^2\le \left|\int_\Omega\chi^2 u(\Delta-V_t^1) udx\right|+(\epsilon^{-1}+1/2)\|u\nabla \chi\|_{L^2(\Omega)}^2,
\]
which, combined with \eqref{ca2}, yields
\begin{align}
&(1/2-\epsilon(1+\sigma^2))\|\chi \nabla u\|_{L^2(\Omega)}^2\le (2\epsilon)^{-1}\|\chi(\Delta-V_t^1) u\|_{L^p(\Omega)}^2 \label{ca5}
\\
&\hskip 5.5cm +(\epsilon^{-1}+1/2+\epsilon \sigma^2)\|u\nabla \chi\|_{L^2(\Omega)}^2.\nonumber
\end{align}
By choosing $\epsilon=[4(1+\sigma^2)]^{-1}$ in \eqref{ca5}, we obtain
\begin{equation}\label{ca6}
c_\ast\|\nabla u\|_{L^2(\omega_0)}^2\le \|\chi(\Delta-V_t^1) u\|_{L^p(\Omega)}^2+d^{-2}\|u\|_{L^2(\omega_1)}^2.
\end{equation}
Here and henceforth, $c_\ast=c_\ast(n,\Omega)>0$ is a generic constant.

Now, since
\begin{align*}
\|\chi(\Delta-V_t^1) u\|_{L^p(\Omega)}^2&\le 2\|\chi(\Delta-V) u\|_{L^p(\Omega)}^2+2\|\chi V_t^2 u\|_{L^p(\Omega)}^2
\\
&\le 2\|\chi(\Delta-V) u\|_{L^p(\Omega)}^2+c_\ast\|V_t^2\|_{L^\infty(\Omega)}\|\chi u\|_{L^2(\Omega)}^2
\end{align*}
and $\|V_t^2\|_{L^\infty(\Omega)}\le t$, we get
\[
\|\chi(\Delta-V_t^1) u\|_{L^p(\Omega)}^2\le 2\|\chi(\Delta-V) u\|_{L^p(\Omega)}^2
 +c_\ast[2\kappa_V^{s/n} \sigma]^{2n/(2s-n)}\|u\|_{L^2(\omega_1)}^2.
\]
This last inequality in \eqref{ca6} gives the expected one.
\end{proof}

Let
\[
\mathscr{V}_0:=\{V\in L^{n/2}(\Omega);\; 2\sigma^2\kappa_V<1\}
\]
and define for $V\in \mathscr{V}_0$
\begin{align*}
&\mathbf{I}_V:=(1-2\sigma^2\kappa_V)^{-1/2},
\\
&\kappa_V^0:=\sqrt{2}(1+\sigma)\mathbf{I}_V^2,
\\
&\kappa_V^1:=2(1+\sigma)\mathbf{I}_V^2+2\sigma\sqrt{\kappa_V}\, \mathbf{I}_V+\sigma/\sqrt{1+\sigma^2}.
\end{align*}

The following Caccioppoli-type inequalities is obtained by modifying slightly the proof of Proposition \ref{proqu1}. 

\begin{proposition}\label{prosc1}
Let $\omega_0\Subset \omega_1\Subset \Omega$ and $d=\mathrm{dist}(\omega_0,\partial \omega_1)$. For all $u\in W^{2,p}(\Omega)$ and $V\in \mathscr{V}_0$ we have
\begin{equation}\label{sc0}
\|\nabla u\|_{L^2(\omega_0)}\le \kappa_V^0\|(-\Delta+V) u\|_{L^p(\Omega)} 
+d^{-1}\kappa_V^1\|u\|_{L^2(\omega_1)}.
\end{equation}
\end{proposition}

Henceforth, $\mathbf{k}>0$ will denote a generic universal constant.

\begin{proposition}\label{proinfty1}
Let $\omega_0\Subset \omega_1\Subset \Omega$ and $d=\mathrm{dist}(\omega_0,\partial \omega_1)$. For all $u\in H^2(\Omega)$ and $V\in L^\infty(\Omega)$ we have
\begin{equation}\label{sc0}
\mathbf{k}\|\nabla u\|_{L^2(\omega_0)}\le \|(-\Delta+V) u\|_{L^2(\Omega)} 
+(d^{-1}+\kappa_V)\|u\|_{L^2(\omega_1)}.
\end{equation}
\end{proposition}

\section{Three-ball inequalities}\label{S3}

In the remainder of this text, unless otherwise stated, $\mathbf{c}=\mathbf{c}(n,s)\ge 1$, $\mathbf{c}_1=\mathbf{c}_1(n,s,r_0)>0$ in the case $s\in (n/2,\infty)$ and $\mathbf{c}=\mathbf{c}(n)\ge 1$, $\mathbf{c}_1=\mathbf{c}_1(n,r_0)>0$ in the case $s=\infty$ will denote generic constants. Also, $0<\alpha <1$ will denote a generic universal constant.

\subsection{Potential in $L^s$, $s\in (n/2,\infty)$}

In the following, if $V\in L^s(\Omega)$, then $\varphi_s(V)$ will denote a generic constant of the form
 \[
\varphi_s(V):= e^{\mathbf{c}_1\kappa_V^\gamma},
 \]
 where
\[
\gamma=\gamma (n,s):=\frac{8ns}{(3n+2)(2s-n)}. 
\]

We establish in this subsection the following three-ball inequality. 

\begin{theorem}\label{mthm1}
Let $r_0>0$ chosen so that $\Omega^{r_0}$ is nonempty and $0<r<r_0/3$. For all $x_0\in \Omega^{3r}$, $V\in L^s(\Omega)$ and $u\in W^{2,p}(\Omega)$ satisfying $(-\Delta +V)u=0$ we have
\begin{equation}\label{qu3}
\|u\|_{L^2(B(x_0,2r))}\le 
\mathbf{c}\varphi_s(V)\|u\|_{L^2(B(x_0,r))}^\alpha \|u\|_{L^2(B(x_0,3r))}^{1-\alpha}.
\end{equation}
\end{theorem}

Before we prove Theorem \ref{mthm1}, we need a consequence of the Carleman inequality proved in \cite{DET}. To this end, let $D$ be a bounded domain of class $C^3$ and $\omega_0\Subset \omega \Subset D$. The unit normal vector field on $\partial D$ will be denoted by $\nu$, and $\partial_\nu$ stands for the derivative along $\nu$. Pick $\varphi \in C^3(\bar{D})$ admitting the following properties: $\varphi_{|\partial D}=0$, $\partial_\nu \varphi<0$,
\[
\varrho:=\min_{\bar{D}\setminus \omega_0}|\nabla \varphi|>0,
\]
and there exists $\beta >0$ such that for all $x\in \bar{D}\setminus \omega_0$ and $\xi \in \mathbb{R}^n$ satisfying $|\nabla \varphi(x)|=|\xi|$ and $\nabla \varphi(x)\bot \xi$ we have
\begin{equation}\label{varphi}
\nabla^2\varphi (x)\nabla \varphi (x)\cdot \nabla \varphi(x)+ \nabla^2\varphi (x) \xi\cdot \xi \ge \beta |\nabla \varphi (x)|^2.
\end{equation}

For further use, we give an example of $\varphi$ when $D=B(0,R)$, $R>0$. Pick $0<\rho_0<\rho<R$. Define $\omega_0=B(0,\rho_0)$ and $\omega=B(0,\rho)$. Then consider $\varphi(x)=R^2-|x|^2$. In this case $\nabla \varphi (x)=-2x$ and as $\nu=x/|x|$ on $\partial D$, we get $\partial_\nu \varphi=-2R$. Also, $|\nabla \varphi|\ge 2\rho_0$ in $\bar{\Omega}\setminus \omega_0$. On the other hand, as $\nabla^2\varphi (x)=2\mathbf{I}$, where $\mathbf{I}$ is the identity matrix, \eqref{varphi} holds. 

Define
\[
W^{2,p}_{\omega}(D):=\{ u\in W^{2,p}(\Omega);\; \mathrm{supp}(u)\subset D \setminus \bar{\omega}\}.
\]

The following Carleman inequality is a special case of \cite[Theorem 1.1]{DET}. 

\begin{proposition}\label{prodd1}
Set $\zeta:=\left(D, \omega_0,\omega, \varrho,\beta, \|\varphi\|_{C^3(\bar{D})}\right)$. There exist $\bar{c}=\bar{c}(\zeta)>0$ and $\tau_0=\tau_0(\zeta)\ge 1$ such that for all $\tau \ge \tau_0$ and $u\in W^{2,p}_{\omega}(D)$ we have
\begin{equation}\label{dd1}
\tau^{3/4+1/(2n)}\|u\|_{L^2_\tau (D)}\le \bar{c}\|\Delta u\|_{L_\tau ^p(D)},
\end{equation}
and
\begin{equation}\label{dd1.1}
\|u\|_{L^{p'}_\tau (D)}\le \bar{c}\|\Delta u\|_{L_\tau ^p(D)}.
\end{equation}
Here $L^r_\tau (\Omega):=L^r (\Omega,e^{r\tau \varphi}dx)$, $r=p,2,p'$.
\end{proposition}

We use this proposition to prove the following lemma that will be used in the proof of Theorem \ref{mthm1}.

\begin{lemma}\label{lemci}
Let $\zeta$ be as in Proposition \ref{prodd1}. There exists $\bar{c}=\bar{c}(\zeta,n,s)>0$  such that for all  $\tau >0$, $V\in L^s(D)$ and $u\in W^{2,p}_{\omega}(D)$ we have
\begin{equation}\label{dd6}
\|u\|_{L^2_\tau (D)}\le \bar{c}e^{\bar{c}\varkappa_V^\gamma}\|(-\Delta +V)u\|_{L_\tau ^p(D)}.
\end{equation}
Here $\varkappa_V:=\|V\|_{L^s(D)}$.
\end{lemma}

\begin{proof}
Let $V\in L^s (D)$ and $u\in W^{2,p}_{\omega}(D)$. From Lemma \ref{lemqu1}, for all $t>0$, $V$ can be decomposed into two terms $V=V_t^1+V_t^2$ so that $V_t^1\in L^{n/2}(D)$, $V_t^2\in L^\infty (D)$ and
\begin{equation}\label{dd2}
\|V_t^1\|_{L^{n/2}(D)}\le t^{-(2s/n-1)}\varkappa_V^{2s/n} ,\quad \|V_t^2\|_{L^\infty(D)}\le t.
\end{equation}

Let $\tau_0$ as in Proposition \ref{prodd1} and $\tau \ge \tau_0$. Then, we have
\begin{align*}
\|Vu\|_{L^p_\tau (D)}&\le \|V_t^1u\|_{L^p_\tau (D)}+\|V_t^2u\|_{L^p_\tau (D)}
\\
&\le \|V_t^1\|_{L^{n/2}(D)}\|u\|_{L^{p'}_\tau (D)}+\bar{c}_0\|V_t^2\|_{L^\infty(\Omega)}\|u\|_{L^2_\tau (D)},
\end{align*}
where $\bar{c}_0=\bar{c}_0(n,D)>0$ is a constant. This and \eqref{dd1.1} imply
\[
\|Vu\|_{L^p_\tau (D)}\le \bar{c}\|V_t^1\|_{L^{n/2}(D)}\|\Delta u\|_{L^p_\tau (D)}+\bar{c}_0\|V_t^2\|_{L^\infty(D)}\|u\|_{L^2_\tau (D)}.
\]
Here and in the sequel, $\bar{c}=\bar{c}(\zeta,n,s)>0$ is a generic constant. Hence 
\begin{align*}
&\|Vu\|_{L^p_\tau (D)}\le \bar{c}\|V_t^1\|_{L^{n/2}(D)}\|Vu\|_{L^p_\tau (D)}+ \bar{c}\|V_t^1\|_{L^{n/2}(D)}\|(-\Delta+V) u\|_{L^p_\tau (D)}
\\
&\hskip 8cm +\bar{c}_0\|V_t^2\|_{L^\infty(D)}\|u\|_{L^2_\tau (D)}.
\end{align*}
Combined with \eqref{dd2}, this inequality yields
\begin{align}
&\|Vu\|_{L^p_\tau (D)}\le \bar{c}t^{-(2s/n-1)}\varkappa_V^{2s/n}\|Vu\|_{L^p_\tau (D)}\label{dd3}
\\
&\hskip 3cm + \bar{c}t^{-(2s/n-1)}\varkappa_V^{2s/n}\|(-\Delta+V) u\|_{L^p_\tau (D)}
+\bar{c}_0t\|u\|_{L^2_\tau (D)}.\nonumber
\end{align}

Assume for the moment that $t>0$ can be chosen in such a way that
\begin{equation}\label{t1}
\bar{c}t^{-(2s/n-1)}\varkappa_V^{2s/n}\le 1/2.
\end{equation}
In this case \eqref{dd3} gives
\begin{equation}\label{dd4}
\|Vu\|_{L^p_\tau (D)}\le \|(-\Delta+V) u\|_{L^p_\tau (D)}+2\bar{c}_0t\|u\|_{L^2_\tau (D)}.
\end{equation}
Combining \eqref{dd1} and \eqref{dd4}, we obtain
\begin{equation}\label{dd5}
\tau^{3/4+1/(2n)}\|u\|_{L^2_\tau (D)}\le \bar{c}\|(-\Delta +V)u\|_{L_\tau ^p(D)}+\bar{c}t\|u\|_{L^2_\tau (D)}.
\end{equation}
Let us now choose $t$ so that $\bar{c}t=\tau^{3/4+1/(2n)}/2$. In light of \eqref{t1}, this choice is possible provided that
\begin{equation}\label{tau1}
\tau \ge \bar{c}\varkappa_V^\gamma.
\end{equation}
Under this condition, we have
\begin{equation}\label{dd5.0}
\tau^{3/4+1/(2n)}\|u\|_{L^2_\tau (D)}\le \bar{c}\|(-\Delta +V)u\|_{L_\tau ^p(D)}.
\end{equation}
Then replacing in \eqref{dd5.0} $\tau$ by $\tau+\bar{c}\varkappa_V^\gamma+\tau_0$ with $\tau >0$, we obtain
\[
\|u\|_{L^2_\tau (D)}\le\bar{c}e^{\bar{c}\varkappa_V^\gamma}\|(-\Delta +V)u\|_{L_\tau ^p(D)}.
\]
This is the expected inequality.
\end{proof}

Recall that $B_r:=B(0,r)$, $r>0$, and
\[
[\![a,b]\!]:=\{x\in \mathbb{R}^n;\; a<|x|<b\},\quad 0\le a<b.
\]

\begin{proof}[Proof of Theorem \ref{mthm1}]
Let $\tau >0$. First, we  use \eqref{dd6} with $D=B_1=:B$, $\omega_0=B_{1/8}$, $\omega=B_{1/7}$ and $\varphi(x)=1-|x|^2$. In the present case \eqref{dd6} holds with constant $\bar{c}=\bar{c}(n,s)>0$. That is, for all $\mathcal{V}\in L^s(B)$ and $w\in W^{2,p}_{B_{1/7}} (B)$ we have 
\begin{equation}\label{dd7}
\|w\|_{L^2_\tau (B)}\le \mathbf{c}e^{\mathbf{c}\varkappa_{\mathcal{V}}^\gamma}\|(-\Delta +\mathcal{V})w\|_{L_\tau ^p(B)}.
\end{equation}

Let $(r_j)_{1\le j\le 8}$ be an increasing sequence of $(1/7,1)$ and $\chi\in C_0^\infty (B)$ satisfying $0\le \chi \le 1$,
\[
\chi=\left\{
\begin{array}{lll}
0\quad &\mbox{in}\; [\![0,r_2]\!]\cup [\![r_7,1]\!],
\\
1 &\mbox{in}\; [\![r_3,r_6]\!],
\end{array}
\right.
\]
and
\begin{align*}
&|\Delta \chi|+|\nabla \chi|^2\le \mathbf{k} d_1^{-2}\quad \mbox{in}\; [\![r_2,r_3]\!],
\\
&|\Delta \chi|+|\nabla \chi|^2\le \mathbf{k} d_2^{-2}\quad \mbox{in}\; [\![r_6,r_7]\!],
\end{align*}
where $\mathbf{k}>0$ is a generic universal constant, and  $d_1=r_3-r_2$ and $d_2=r_7-r_6$.

Let $\mathcal{V}\in L^s(B)$ and $w\in W^{2,p}(B)$ satisfying $(-\Delta +\mathcal{V})w=0$. As $\chi w\in W^{2,p}_{B_{1/7}} (B)$, applying \eqref{dd7}, we obtain
\begin{equation}\label{dd8}
\|\chi w\|_{L^2_\tau (B)}\le \mathbf{c}e^{\mathbf{c}\varkappa_{\mathcal{V}}^\gamma}\|(-\Delta +\mathcal{V})(\chi w)\|_{L_\tau ^p(B)}.
\end{equation}
This and the inequality $\|h\|_{L^p(B)}\le c_0\|h\|_{L^2(B)}$ for all $h\in L^2(B)$, imply
\begin{align*}
&c_0\|e^{\tau \varphi}(-\Delta +\mathcal{V})(\chi w)\|_{L^p(B)}
\\
&\hskip 2cm\le  d_1^{-2}\|e^{\tau \varphi}w\|_{L^2([\![r_2,r_3]\!])}+d_1^{-1}\|e^{\tau \varphi}\nabla w\|_{L^2([\![r_2,r_3]\!])}
\\
&\hskip 3cm +d_2^{-2}\|e^{\tau \varphi}w\|_{L^2([\![r_6,r_7]\!])}+d_2^{-1}\|e^{\tau \varphi}\nabla w\|_{L^2([\![r_6,r_7]\!])}.
\end{align*}
Here and henceforth $c_0=c_0(n)>0$ is a generic constant. Whence
\begin{align*}
&c_0\|e^{\tau \varphi}(-\Delta +\mathcal{V})(\chi w)\|_{L^p(B)}
\\
&\hskip 2cm\le  d_1^{-2}e^{\tau(1-r_2^2)}\|w\|_{L^2([\![r_2,r_3]\!])}+d_1^{-1}e^{\tau(1-r_2^2)}\|\nabla w\|_{L^2([\![r_2,r_3]\!])}
\\
&\hskip 3cm +d_2^{-2}e^{\tau(1-r_6^2)}\|e^{\tau \varphi}w\|_{L^2([\![r_6,r_7]\!])}+d_2^{-1}e^{\tau(1-r_6^2)}\|\nabla w\|_{L^2([\![r_6,r_7]\!])},
\end{align*}
from which we derive
\begin{align*}
&c_0\|e^{-\tau |x|^2}(-\Delta +\mathcal{V})(\chi w)\|_{L^p(B)}
\\
&\hskip 2cm\le  d_1^{-2}e^{-\tau r_2^2}\|w\|_{L^2([\![r_2,r_3]\!])}+d_1^{-1}e^{-\tau r_2^2}\|\nabla w\|_{L^2([\![r_2,r_3]\!])}
\\
&\hskip 3cm +d_2^{-2}e^{-\tau r_6^2}\|e^{\tau \varphi}w\|_{L^2([\![r_6,r_7]\!])}+d_2^{-1}e^{-\tau r_6^2}\|\nabla w\|_{L^2([\![r_6,r_7]\!])}.
\end{align*}
Assume that $d_1=d_2=d$, $r_2-r_1=r_4-r_3=d$ and $r_7-r_6=r_8-r_7=d$. Applying Caccioppoli's inequality \eqref{ca1}, we get
\begin{align*}
&c_0\|e^{-\tau |x|^2}(-\Delta +\mathcal{V})(\chi w)\|_{L^p(B)}
\\
&\hskip 2cm\le  d^{-2}(\varkappa_{\mathcal{V}}^{s/(2s-n)}+1)\left[ e^{-\tau r_2^2}\|w\|_{L^2([\![r_1,r_4]\!])} +e^{-\tau r_6^2}\|w\|_{L^2([\![r_5,r_8]\!])}\right].
\end{align*}

In light of \eqref{dd8}, this inequality yields
\[
c_0e^{-\tau r_5^2}\|w\|_{L^2([\![r_3,r_5]\!])}\le d^{-2}e^{\mathbf{c}\varkappa_{\mathcal{V}}^\gamma}\left[ e^{-\tau r_2^2}\|w\|_{L^2([\![r_1,r_4]\!])} +e^{-\tau r_6^2}\|w\|_{L^2([\![r_5,r_8]\!])}\right].
\]
That is we have 
\[
c_0\|w\|_{L^2([\![r_3,r_5]\!])}\le d^{-2}e^{\mathbf{c}\varkappa_{\mathcal{V}}^\gamma}\left[ e^{\tau a }\|w\|_{L^2([\![r_1,r_4]\!])} +e^{-\tau b}\|w\|_{L^2([\![r_5,r_8]\!])}\right],
\]
where $a:=r_5^2-r_2^2$ and $b:=r_6^2-r_5^2$.

Next, taking $r_1=1/6$, $r_2=2/9$, $r_3=5/18$, $r_4=1/3$, $r_5=2/3$, $r_6=13/18$, $r_7=7/9$ and $r_8=5/6$, we find $d=1/18$, $a=32/81$, $b=25/324$ and 
\[
c_0\|w\|_{L^2([\![5/18,2/3]\!])}\le e^{\mathbf{c}\varkappa_{\mathcal{V}}^\gamma}\left[ e^{\tau a  }\|w\|_{L^2([\![1/6,1/3]\!])} +e^{-\tau b}\|w\|_{L^2([\![2/3,5/6]\!])}\right]
\]
and thus
\[
c_0\|w\|_{L^2(B_{2/3})}\le e^{\mathbf{c}\varkappa_{\mathcal{V}}^\gamma}\left[ e^{\tau a }\|w\|_{L^2(B_{1/3})} +e^{-\tau b}\|w\|_{L^2(B)}\right].
\]
Choosing $\tau= 1/(a+b)\ln \left(\|w\|_{L^2(B)}/\|w\|_{L^2(B_{1/3})}\right)$ in this inequality, we obtain
\begin{equation}\label{dd9}
c_0\|w\|_{L^2(B_{2/3})}\le e^{\mathbf{c}\varkappa_{\mathcal{V}}^\gamma}\|w\|_{L^2(B_{1/3})}^\alpha\|w\|_{L^2(B)}^{1-\alpha},
\end{equation}
where, $\alpha:=b/(a+b)$.

Next, let $x_0\in \Omega^{3r}$ with $0<r<r_0/3$,  $V\in L^s(\Omega)$ and $u\in W^{2,p}(\Omega)$ satisfying $(-\Delta +V)u=0$, and set
\[
w(y):=u(x_0+3ry),\quad  \mathcal{V}(y):=(3r)^2V(x_0+3ry),\quad y\in B.
\]
Then $(-\Delta +\mathcal{V})w=0$ and
\[
\varkappa_\mathcal{V}=(3r)^{2-n/s}\|V\|_{L^s(B(x_0,3r))}\le (3r)^{2-n/s}\kappa_V\le r_0^{2-n/s}\kappa_V.
\]
Applying \eqref{dd9}, we obtain
\[
\|u\|_{L^2(B(x_0,2r))}\le \mathbf{c} e^{\mathbf{c}_1\kappa_V^\gamma}\|u\|_{L^2(B(x_0,r))}^\alpha\|u\|_{L^2(B(x_0,3r))}^{1-\alpha}.
\]
This is the expected inequality.
\end{proof}

\subsection{Potential in $L^\infty$}

The proof of the three-ball inequality in the previous subsection can be adapted to cover the case $s=\infty$. The advantage of the proof we provide in the current subsection is that it remains valid whenever $\Delta$ is substituted by the Laplace-Beltrami operator $\Delta_g$, where $g$ is a metric with coefficients belonging to $C^{0,1}(\overline{\Omega})$.

Pick $D$ a bounded domain of $\mathbb{R}^n$. Let $\psi \in C^2(\bar{D})$ chosen so that there exists $\varrho>0$ such that
\[
\psi \ge \varrho, \quad |\nabla \psi |\ge \varrho .
\]
For each $\lambda >0$, set $\phi_\lambda:=e^{\lambda \psi}$ and recall that the closure of $C_0^\infty (D)$ in $H^2(D)$ is usually denoted by $H_0^2(D)$. If $\zeta:=\left(\varrho ,\|\psi\|_{C^2(\bar{D})}\right)$, then \cite[Theorem 2.8]{ChLN} shows that there exist $\lambda=\lambda(\zeta)>0$, $\tau_0=\tau_0(\zeta)>0$ and $\bar{c}=\bar{c}(\zeta)>0$ so that for all $\tau \ge \tau_0$ and $u\in H_0^2(D)$ we have
\begin{equation}\label{C1}
\tau^{3/2}\|u\|_{L_\tau^2(D)}\le \bar{c}\|\Delta u\|_{L_\tau^2(D)}.
\end{equation}
Here $L_\tau^2(D):=L^2(D, \Phi_\tau (x)dx)$ with $\Phi_\tau=e^{2\tau \phi_\lambda }$.

Let $V\in L^\infty(D)$, $\varkappa_V:=\|V\|_{L^\infty(D)}$ and $u\in H_0^2(D)$. Using \eqref{C1} and 
\[
\|\Delta u\|_{L_\tau^2(D)}\le \|(-\Delta+V) u\|_{L_\tau^2(D)}+\varkappa_V\|u\|_{L_\tau^2(D)},
\]
we obtain
\begin{equation}\label{C2}
\tau^{3/2}\|u\|_{L_\tau^2(D)}\le \bar{c}\|(-\Delta +V)u\|_{L_\tau^2(D)}+\bar{c}\varkappa_V\|u\|_{L_\tau^2(D)}.
\end{equation}

In light of \eqref{C2}, we can proceed as in the proof of Lemma \ref{lemci} to obtain the following result.
\begin{lemma}\label{lemC1}
Let $\zeta$ be as above. There exists $\bar{c}=\bar{c}(\zeta)>0$  such that for all  $\tau >0$, $V\in L^\infty(D)$ and $u\in H_0^2(D)$ we have
\begin{equation}\label{C3}
\|u\|_{L^2_\tau (D)}\le \bar{c}e^{\bar{c}\varkappa_V}\|(-\Delta +V)u\|_{L_\tau ^2(D)}.
\end{equation}
\end{lemma}

In view of this lemma, we can modify the proof of Theorem \ref{mthm1} to establish the following result where, for $V\in L^\infty(\Omega)$, $\varphi_\infty(V)$ is a generic constant of the form

\[
\varphi_\infty (V):=e^{\mathbf{c}_1\kappa_V}.
\]

\begin{theorem}\label{thmC1}
Let $r_0>0$ chosen so that $\Omega^{r_0}$ is nonempty and let $0<r<r_0/3$. For all $x_0\in \Omega^{r}$, $V\in L^\infty(\Omega)$ and $u\in H^2(\Omega)$ satisfying $(-\Delta +V)u=0$ we have
\begin{equation}\label{C4}
\|u\|_{L^2(B(x_0,2r))}\le 
\mathbf{c} \varphi_\infty(V)\|u\|_{L^2(B(x_0,r))}^{\alpha} \|u\|_{L^2(B(x_0,3r))}^{1-\alpha}.
\end{equation}
\end{theorem}

\subsection{Potential in $L^{n/2}$}

Unlike the case $s\in (n/2,\infty]$, we will use in the current case, $s=n/2$, a singular weight Carleman inequality whose parameter is subject to a constraint. Let
\[
\Lambda:=\{ \lambda >0;\; \mathrm{dist}(\lambda ,\mathbb{N}+(n-2)/2)= 1/2\}.
\]
Let $r>0$ so that $\Omega^{4r}$ is nonempty and set $B:=B_{4r}$ and $\dot{B}:=B\setminus\{0\}$. We recall the following Carleman inequality appearing in \cite[(3.4)]{JK} : there exists $\vartheta=\vartheta(n)>0$ such that for any $\lambda\in \Lambda$ and $u\in W^{2,p}(B)$ with $\mathrm{supp}(u)\subset \dot{B}$ we have
\begin{equation}\label{car}
\||x|^{-\lambda}u\|_{L^{p'}(B)}\le \vartheta\||x|^{-\lambda}\Delta u\|_{L^p(B)}.
\end{equation}

Next, define
\[
\mathscr{V}:=\{V\in L^{n/2}(\Omega);\; 2\sigma^2\kappa_V<1\; \mathrm{and}\; \vartheta \kappa_V<1\}\; (\subset \mathscr{V}_0).
\]

Let $V\in \mathscr{V}$, $x_0\in \Omega^{4r}$, $\lambda\in \Lambda$ and $u\in W^{2,p}(\Omega)$ with $\mathrm{supp}(u(x_0+\cdot))\subset \dot{B}$. Applying H\"older's inequality, we obtain
\begin{align*}
\|V(x_0+\cdot )|x|^{-\lambda}u(x_0+\cdot)\|_{L^p(B)}&\le \|V(x_0+\cdot)\|_{L^{n/2}(B)}\||x|^{-\lambda}u(x_0+\cdot)\|_{L^{p'}(B)}
\\
&\le \kappa_V\||x|^{-\lambda }u(x_0+\cdot)\|_{L^{p'}(B)},
\end{align*}
which, in combination with \eqref{car}, yields
\begin{align*}
\||x|^{-\lambda}u(x_0+\cdot)\|_{L^{p'}(B)}&\le \vartheta\||x|^{-\lambda}(-\Delta+V(x_0+\cdot)) u(x_0+\cdot)\|_{L^p(B)}
\\
&\hskip 2cm + \vartheta \|V(x_0+\cdot )|x|^{-\lambda}u(x_0+\cdot)\|_{L^p(B)}
\\
&\le \vartheta\||x|^{-\lambda}(-\Delta+V(x_0+\cdot)) u(x_0+\cdot)\|_{L^p(B)}
\\
&\hskip 3cm + \vartheta \kappa_V\||x|^{-\lambda}u(x_0+\cdot)\|_{L^{p'}(B)}.
\end{align*}
Thus
\begin{equation}\label{sc1}
\||x|^{-\lambda}u(x_0+\cdot)\|_{L^{p'}(B)}\le \vartheta(1-\vartheta\kappa_V)^{-1}\||x|^{-\lambda}(-\Delta+V(x_0+\cdot)) u(x_0+\cdot)\|_{L^p(B)}.
\end{equation}
We will use hereinafter the notation
\[
\vartheta_V:=\vartheta(1-\vartheta\kappa_V)^{-1},\quad V\in \mathscr{V}.
\]
In this case, \eqref{sc1} can be rewritten in the form
\begin{equation}\label{sc2}
\||x|^{-\lambda}u(x_0+\cdot)\|_{L^{p'}(B)}\le \vartheta_V\||x|^{-\lambda}(-\Delta+V(x_0+\cdot) u(x_0+\cdot)\|_{L^p(B)}.
\end{equation}

From now on, for $V\in \mathscr{V}$, $q_V$ denotes a generic constant of the form
\[
q_V:=\mathbf{k}\vartheta_V(1+\kappa_V^1),
\]
that is,
\begin{equation}\label{qV}
q_V=\mathbf{k}\vartheta_V\left[2(1+\sigma)\mathbf{I}_V^2+2\sigma\sqrt{\kappa_V}\, \mathbf{I}_V+\sigma/\sqrt{1+\sigma^2}+1\right],
\end{equation}
where we recall that $\mathbf{I}_V:=(1-2\sigma^2\kappa_V)^{-1/2}$ and $\mathbf{k}>0$ denotes a generic universal constant.

Replacing $\mathbf{k}$ by $\max(\mathbf{k},1)$ and $\vartheta$ by $\max(\vartheta,1)$, we assume that $q_V\ge 1$.

\begin{theorem}\label{thmsc2}
Fix $0<r_0\le 1$ so that $\Omega^{r_0}$ is nonempty and let $0<r<r_0/4$. For all $x_0\in \Omega^{4r}$, $V\in \mathscr{V}$ and $u\in W^{2,p}(\Omega)$ satisfying $(-\Delta +V)u=0$ we have
\begin{equation}\label{a1.1}
\|u\|_{L^2(B(x_0,2r))}\le q_Vr^{-1}\|u\|_{L^2(B(x_0,3r))}^{1-\alpha}\|u\|_{L^2(B(x_0,r))}^{\alpha}.
\end{equation}
\end{theorem}

\begin{proof}
Let $x_0\in \Omega^{4r}$, $0<r<r_0/4$, $V\in \mathscr{V}$ and $u\in W^{2,p}(\Omega)\setminus\{0\}$ (note that \eqref{a1.1} is trivially satisfied when $u=0$) satisfying $(-\Delta +V)u=0$, and define $w$ and $\mathcal{V}$ on $B:=B_{4r}$ by $w:=u(x_0+\cdot)$ and $\mathcal{V}:=V(x_0+\cdot)$. Recall the notation
\[
[\![a,b]\!]:=\{x\in \mathbb{R}^n;\, a<|x|<b\},\quad 0\le a<b.
\]
Let $(r_j)_{1\le j\le 8}$ be an increasing sequence of $(0,4r)$ and $\chi\in C_0^\infty (B)$ satisfying $0\le \chi \le 1$,
\[
\chi=\left\{
\begin{array}{lll}
0\quad &\mbox{in}\; [\![0,r_2]\!]\cup [\![r_7,4r]\!],
\\
1 &\mbox{in}\; [\![r_3,r_6]\!],
\end{array}
\right.
\]
and
\begin{align*}
&|\Delta \chi|+|\nabla \chi|^2\le \mathbf{k} d_1^{-2}\quad \mbox{in}\; [\![r_2,r_3]\!],
\\
&|\Delta \chi|+|\nabla \chi|^2\le \mathbf{k} d_2^{-2}\quad \mbox{in}\; [\![r_6,r_7]\!],
\end{align*}
where $d_1:=r_3-r_2$ and $d_2:=r_7-r_6$.

Using the inequality $\|h\|_{L^p(B)}\le \mathbf{k} r\|h\|_{L^2(B)}$ for $h\in L^2(B)$, we obtain
\begin{align*}
\||x|^{-\lambda}(-\Delta +\mathcal{V})(\chi w)\|_{L^p(B)} &= \||x|^{-\lambda}(\Delta \chi w+2\nabla \chi\cdot \nabla w)\|_{L^p(B)}
\\
&\le \mathbf{k} r \||x|^{-\lambda}(\Delta \chi w+2\nabla \chi\cdot \nabla w)\|_{L^2(B)}
\end{align*}
and therefore 
\begin{align*}
&\mathbf{k} r^{-1}\||x|^{-\lambda}(-\Delta +\mathcal{V})(\chi w)\|_{L^2(B)}\le  d_1^{-2}\||x|^{-\lambda}w\|_{L^2([\![r_2,r_3]\!])}+d_1^{-1}\||x|^{-\lambda}\nabla w\|_{L^2([\![r_2,r_3]\!])}
\\
&\hskip 4.7cm +d_2^{-2}\||x|^{-\lambda}w\|_{L^2([\![r_6,r_7]\!])}+d_2^{-1}\||x|^{-\lambda}\nabla w\|_{L^2([\![r_6,r_7]\!])}.
\end{align*}
In consequence, we obtain
\begin{align}
&\mathbf{k} r^{-1}\||x|^{-\lambda}(-\Delta +\mathcal{V})(\chi w)\|_{L^p(B)}\label{a2}
\\
&\hskip 2cm\le  d_1^{-2}r_2^{-\lambda}\|w\|_{L^2([\![r_2,r_3]\!])}+d_1^{-1}r_2^{-\lambda}\|\nabla w\|_{L^2([\![r_2,r_3]\!])}\nonumber
\\
&\hskip 3.5cm +d_2^{-2}r_6^{-\lambda}\|w\|_{L^2([\![r_6,r_7]\!])}+d_2^{-1}r_6^{-\lambda}\|\nabla w\|_{L^2([\![r_6,r_7]\!])}.\nonumber
\end{align}
Assume that $r_2-r_1=r_4-r_3=d_1$ and $r_6-r_5=r_8-r_7=d_2$. In light of \eqref{sc0}, we obtain from \eqref{a2}
\begin{align}
&\||x|^{-\lambda}(-\Delta +\mathcal{V})(\chi w)\|_{L^p(B)}\label{a3}
\\
&\hskip 1cm \le \mathbf{k} (1+\kappa_V^1)r\left(  d_1^{-2}r_2^{-\lambda}\|w\|_{L^2([\![r_1,r_4]\!])}+ d_2^{-2}r_6^{-\lambda}\|w\|_{L^2([\![r_5,r_8]\!])}\right).\nonumber
\end{align}

On the other hand, we have from \eqref{sc2}
\[
\||x|^{-\lambda}\chi w\|_{L^{p'}(B)}\le \vartheta_V \||x|^{-\lambda}(-\Delta +\mathcal{V})(\chi w)\|_{L^p(B)}.
\]
This and \eqref{a3} imply
\[
r_5^{-\lambda}\|w\|_{L^2([\![r_3,r_5]\!])} \le q_Vr\left(  d_1^{-2}r_2^{-\lambda}\|w\|_{L^2([\![r_1,r_4]\!])}+ d_2^{-2}r_6^{-\lambda}\|w\|_{L^2([\![r_5,r_8]\!])}\right).
\]
Thus,
\begin{equation}\label{a4}
\|w\|_{L^2([\![r_3,r_5]\!])}
\le q_Vrr_5^\lambda\left(  d_1^{-2}r_2^{-\lambda}\|w\|_{L^2([\![r_1,r_4]\!])}+ d_2^{-2}r_6^{-\lambda}\|w\|_{L^2([\![r_5,r_8]\!])}\right).
\end{equation}

Next, let us specify the sequence $(r_j)$. We choose $r_1=3r/8$, $r_2=5r/8$, $r_3=3r/4$, $r_4=r$, $r_5=2r$, $r_6=9r/4$,  $r_7=11r/4$ and $r_8=3r$. In this case $d_1=d_2=r/4$. This choice in \eqref{a4} gives
\[
q_V^{-1}r\|w\|_{L^2([\![3r/4,2r]\!])}\le  (5/16)^{-\lambda}\|w\|_{L^2([\![3r/8,r]\!])}+ (8/9)^\lambda\|w\|_{L^2([\![2r,3r]\!])}
\]
and therefore
\[
q_V^{-1}r\|w\|_{L^2(B_{2r})}\le  (5/16)^{-\lambda}\|w\|_{L^2(B_r)}+q_V^{-1}r \|w\|_{L^2(B_{3r/4})}+ (8/9)^\lambda\|w\|_{L^2(B_{3r})}.
\]
Let
\[
\lambda_V:=\ln(q_V^{-1}r_0)/\ln(16/5).
\]
If $\lambda\ge \lambda_V$, the preceding inequality gives
\begin{equation}\label{a5}
q_V^{-1}r\|w\|_{L^2(B_{2r})}\le  (5/16)^{-\lambda}\|w\|_{L^2(B_r)}+ (8/9)^\lambda\|w\|_{L^2(B_{3r})}.
\end{equation}

We are going to prove that \eqref{a5} yields the following inequality
\begin{equation}\label{a6}
q_V^{-1}r\|w\|_{L^2(B_{2r})}\le \|w\|_{L^2(B_{3r})}^{1-\alpha}\|w\|_{L^2(B_r)}^{\alpha},
\end{equation}
where $\alpha:= (\ln 18-\ln 16)/(\ln 18 -\ln 5)$.

Clearly, \eqref{a6} holds trivially if $\|w\|_{L^2(B_{3r})}=0$ or $\|w\|_{L^2(B_r)}=\|w\|_{L^2(B_{3r})}$. In the case $\|w\|_{L^2(B_r)}=0$, \eqref{a5} implies
\begin{equation}\label{a5.bis}
q_V^{-1}r\|w\|_{L^2(B_{2r})}\le  (8/9)^\lambda\|w\|_{L^2(B_{3r})}.
\end{equation}
Let $(\lambda_j)$ be a sequence in $\Lambda$ converging to $\infty$. Taking in \eqref{a5.bis} $\lambda=\lambda_j$ with $j$ sufficiently large and making the limit as $j$ tends to $\infty$, we get  $\|w\|_{L^2(B_{2r})}=0$ and hence \eqref{a6} holds again in the present case. It remains to verify the case
\[\|w\|_{L^2(B_{3r})}/\|w\|_{L^2(B_r)}>1.\]  
Set $\tilde{\lambda}=\ln(18/5)^{-1}\ln (\|w\|_{L^2(B_{3r})}/\|w\|_{L^2(B_r)})$. When $\tilde{\lambda}\le \lambda_V+1$, we have
\begin{align*}
\|w\|_{L^2(B_{3r})}&\le e^{\ln(18/5)(\lambda_V+1)}\|w\|_{L^2(B_r)}
\\
&\le e^{\ln(18/5)} (q_V^{-1}r_0)^{\ln(18/5)/\ln(16/5)}\|w\|_{L^2(B_r)}
\\
&\le (18/5) \|w\|_{L^2(B_r)}.
\end{align*}

Next, assume that $\tilde{\lambda}>\lambda_V+1$. If $\tilde{\lambda}\in \Lambda$, then $\lambda=\tilde{\lambda}$ in \eqref{a5} yields \eqref{a6}.
In the case $\tilde{\lambda}\not\in \Lambda$, we find a positive integer $j$ so that $\tilde{\lambda}\in [j+(n-2)/2, j+1/2+(n-2)/2)$ or $\tilde{\lambda}\in (j-1/2+(n-2)/2,j+(n-2)/2)]$. If $\tilde{\lambda}\in j+(n-2)/2, j+1/2+(n-2)/2)$ we take $\lambda=j+1/2+(n-2)/2$ in \eqref{a5}. As $\tilde{\lambda}<\lambda<\tilde{\lambda}+1$, we verify that \eqref{a6} holds with the same $\alpha$. When $\tilde{\lambda}\in (j-1/2+(n-2)/2), j+(n-2)/2]$, we proceed similarly as in the preceding case. By taking $\lambda= j-1/2+(n-2)/2$, we show that \eqref{a6} holds again with the same $\alpha$.
\end{proof}

\section{Quantitative uniqueness of continuation}\label{S4}

In this section we establish the quantitative uniqueness of the continuation results that we obtain as a consequence of the three-ball inequalities of the previous section.  Recall that $\Omega^r$ is nonempty and connected for all $0<r<\mathfrak{r}$, where $\mathfrak{r}=\mathfrak{r}(\Omega)\le 1$ is a constant.

Before stating the first result precisely, we introduce new notations. Let $Q$ be the smallest closed cube containing $\bar{\Omega}$ and set $\mathbf{D}:=|Q|^{1/n}$. Henceforth, $0<\eta<1$ is a generic function of the form
\begin{equation}\label{eta}
\eta (r):=e^{-\mathfrak{h} r^{-n}},\quad 0<r <\mathfrak{r}/4,
\end{equation}
where
\begin{equation}\label{h}
\mathfrak{h}=\mathfrak{h}(n,\mathbf{D},\mathfrak{r}):=2^{n-1}|\ln \alpha| \left((\mathfrak{r}/4)^n+(\mathbf{D}\sqrt{n})^n\right) .
\end{equation}

Henceforth, $\mathbf{c}_1=\mathbf{c}_1(n,s,\mathfrak{r})>0$ will denote a generic constant and, for $V\in L^s(\Omega)$, $s\in (n/2,\infty]$, the generic constant $\varphi_s(V)$ is the same as  in the preceding section. That is
\begin{equation}\label{varphis}
\varphi_s(V):=e^{\mathbf{c}_1\kappa_V^\gamma},\quad \varphi_\infty (V):=e^{\mathbf{c}_1\kappa_V},
\end{equation}
where $\kappa_V:=\|V\|_{L^s(\Omega)}$ and $\gamma:= 8ns/[(3n+2)(2s-n)]$.

It is worth noting that, for fixed $n\ge 3$, the mapping $s\mapsto \gamma(s)=\gamma$ is decreasing and bijective from $(n/2,\infty)$ onto $(4n/(3n+2),\infty)$. Therefore, $\gamma >1$ for any $s\in (n/2,\infty)$. In consequence,  for all $s\in (n/2,\infty)$, $\varphi_s (V)>\varphi_\infty (V)$ whenever $\kappa_V$ for sufficiently large.

For convenience, we set $X_s:=W^{2,p}(\Omega)$ if $s\in (n/2,\infty)$ and $X_\infty:=H^2(\Omega)$.

\begin{theorem}\label{mthm2}
Let $s\in (n/2,\infty]$, $0<r<\mathfrak{r}/4$ and  $x,y\in \Omega^{4r}$. For all $V\in L^s(\Omega)$ and $u\in X_s$ satisfying $(-\Delta +V)u=0$ and $\|u\|_{L^2(\Omega)}=1$ we have
\begin{equation}\label{qu10}
\|u\|_{L^2(B(y,r))}\le \varphi_s(V)\|u\|_{L^2(B(x,r))}^{\eta(r)}.
\end{equation}
\end{theorem}

Note that without the condition $\|u\|_{L^2(\Omega)}=1$, \eqref{qu10} must be replaced by
\[
\|u\|_{L^2(B(y,r))}\le \varphi_s(V)\|u\|_{L^2(B(x,r))}^{\eta(r)}\|u\|_{L^2(\Omega)}^{1-\eta(r)}.
\]

Let us note that Theorem \ref{mthm2} quantifies the following uniqueness of continuation result.

\begin{corollary}\label{cors1}
Let $s\in (n/2,\infty]$ and $\omega \Subset \Omega$ and $V\in L^s(\Omega)$. If $u\in X_s$ satisfies $(-\Delta +V)u=0$ and $u=0$ in $\omega$, then $u=0$.
\end{corollary}

\begin{proof}[Proof of Theorem \ref{mthm2}]
Let $0<r<\mathfrak{r}/4$. We first consider the case $s\in (n/2,\infty)$. Recall that $\mathbf{D}:=|Q|^{1/n}$, where $Q$ is the smallest cube containing $\bar{\Omega}$, and set
\[
m_r:=\left(\left\lfloor\mathbf{D}\sqrt{n}/r\right\rfloor+1\right)^n,
\]
where $\left\lfloor\mathbf{D}\sqrt{n}/r\right\rfloor$ is the largest integer less than or equal to $\mathbf{D}\sqrt{n}/r$.

We divide $Q$ into $m_r$ closed sub-cubes. Let $(Q_j)_{1\le j\le m_r}$ denotes the family of these cubes. Note that, for each $j$, $|Q_j|< (r/\sqrt{n})^n$ and therefore $Q_j$ is contained in a ball $B_j$ of radius $r/2$. Define
\[
I_r:=\{j \in \{1,\ldots,m_r\};\; Q_j\cap \bar{\Omega}^{4r}\ne \emptyset\}\quad \mathrm{and}\quad Q^r:=\bigcup_{j\in I_r}Q_j.
\]
In particular, $Q^r\subset \Omega^{3r}$ and since $\bar{\Omega}^{4r}$ is connected then so is $Q^r$.

It follows from Lemma \ref{lemGeo} in Appendix \ref{appB} that there exists  a continuous path $\psi:[0,1]\rightarrow Q^r$ joining $x$ to $y$ whose length, denoted hereinafter by $\ell(\psi)$, does not exceed $r m_r$.

Let $t_0=0$ and define the sequence $(t_k)$ as follows
\[
t_{k+1}:=\inf\{t\in [t_k,1];\; \psi(t)\not\in B(\psi(t_k),r)\},\quad k\ge 0.
\]
Then $|\psi(t_{k+1})-\psi(t_k)|=r$. Thus, there exists a positive integer $p_r$ so that $\psi(1)\in B(\psi(t_{p_r}),r)$. As $r (p_r+1) \le \ell(\psi)\le r m_r$, we have $p_r+1\le m_r$.

Set $x_j=\psi(t_j)$, $j=0, \ldots,p_r$ and $x_{p_r+1}=y$. We verify that $B(x_j,3r)\subset \Omega$, $j=0, \ldots,p_r+1$, and  $B(x_{j+1},r)\subset B(x_j,2r)$, $j=0, \ldots,p_r$.

 If $V\in L^s(\Omega)$  and $u\in W^{2,p}(\Omega)$ satisfying $(-\Delta +V)u=0$ and $\|u\|_{L^2(\Omega)}=1$, then we have from \eqref{qu3}
\begin{equation}\label{qu7}
\|u\|_{L^2(B(z,2r))}\le \mathcal{C}\|u\|_{L^2(B(z,r))}^\alpha ,\quad z\in \Omega^{3r}.
\end{equation}
Here and henceforth, $\mathcal{C}\ge 1$ is a generic constant of the form $\mathbf{c}\varphi_s(V)$.

By applying \eqref{qu7} with $z=x_j$, $j=0,\ldots ,p_r$, we obtain
\begin{equation}\label{qu8}
\|u\|_{L^2(B(x_{j+1},r))}\le \mathcal{C}\|u\|_{L^2(B(x_j,r))}^\alpha ,\quad j=0,\ldots ,p_r.
\end{equation}
Using an induction in $j$, we derive from \eqref{qu8} 
\begin{equation}\label{qu9}
\|u\|_{L^2(B(y,r))}\le \mathcal{C}^{1+\alpha+\ldots +\alpha^{p_r}}\|u\|_{L^2(B(x,r))}^{\alpha^{p_r+1}}.
\end{equation}
Since $\mathcal{C}\ge 1$ and $1+\alpha+\ldots +\alpha^{p_r}\le (1-\alpha)^{-1}$, \eqref{qu9} implies
\[
\|u\|_{L^2(B(y,r))}\le \mathcal{C}\|u\|_{L^2(B(x,r))}^{\alpha^{p_r+1}}.
\]
To complete the proof, since $\|u\|_{L^2(B(x,r))}\le 1$, we need only to verify that $\alpha^{p_r+1}\ge \eta(r)$. For this, we have
\[
\alpha^{p_r+1}= e^{-(p_r+1)|\ln \alpha|}\ge  e^{-|\ln \alpha|([\mathbf{D}\sqrt{n}/r]+1)^n}.
\]
Since 
\begin{align*}
([\mathbf{D}\sqrt{n}/r]+1)^n &\le 2^{n-1}([\mathbf{D}\sqrt{n}/r]^n+1)=2^{n-1}+2^{n-1}(\mathbf{D}\sqrt{n})^n r^{-n}
\\
&\le 2^{n-1}((\mathfrak{r}/4)^n+(\mathbf{D}\sqrt{n})^n) r^{-n},
\end{align*}
we deduce that $\alpha^{p_r+1}\ge \eta(r)$.

The proof of the case $s=\infty$ is exactly the same. We have only to apply Theorem \ref{thmC1} instead of Theorem \ref{mthm1}.
\end{proof}

Proceeding similarly as for  Theorem \ref{mthm2}, we prove

\begin{theorem}\label{mthm2.0}
Let  $0<r<\mathfrak{r}/4$ and $x,y\in \Omega^{4r}$. For all $V\in \mathscr{V}$ and $u\in W^{2,p}(\Omega)$ satisfying $(-\Delta +V)u=0$ and $\|u\|_{L^2(\Omega)}=1$ we have
\begin{equation}\label{a1.3}
\|u\|_{L^2(B(y,r))}\le [q_Vr^{-1}]^{1/(1-\alpha)}\|u\|_{L^2(B(x,r))}^{\eta(r)}.
\end{equation}
\end{theorem}

\section{Global quantitative uniqueness of continuation}\label{S5}

We reuse the notation $X_s:=W^{2,p}(\Omega)$ if $s\in (n/2,\infty)$ and $X_\infty:=H^2(\Omega)$. Set
\[
\Omega_r:=\{x\in \Omega;\; \mathrm{dist}(x,\mathbb{R}^n\setminus \Omega)<r\},\quad 0<r<\mathfrak{r},
\]
where $\mathfrak{r}=\mathfrak{r}(\Omega)$  is as in Section \ref{S1}.

For $0<\mathfrak{t} <1/2$ , $\bar{\mathbf{c}}=\bar{\mathbf{c}}(n,\Omega,\mathfrak{t})>0$ will  denote a generic constant. 
Let $\mathfrak{h}$ given by \eqref{h} and
\[
\varrho(r):=e^{\bar{\mathfrak{h}}r^{-n}},\quad r>0,
\]
where
\begin{equation}\label{tildeh}
\bar{\mathfrak{h}}:=\mathfrak{h}+1.
\end{equation}

\begin{theorem}\label{thma2}
Let $s\in (n/2,\infty]$, $0<\mathfrak{t}<1/2$ and $\omega\Subset \Omega$. There exist $r_\ast=r_\ast(\Omega,\omega,\mathfrak{r})\le \mathfrak{r}/4$ such that for all $V\in L^s(\Omega)$, $0<r<r_\ast$ and $u\in X_s(\Omega)$ satisfying $(-\Delta +V)u=0$ we have
\begin{equation}\label{ap5}
\|u\|_{L^2(\Omega)}\le \mathbf{c}\varphi_s(V)\left( e^{\bar{\mathbf{c}}\varrho(r)}\|u\|_{L^2(\omega)}+r^\mathfrak{t} \|u\|_{H^1(\Omega)}\right).
\end{equation}
\end{theorem}

\begin{proof}
Assume that $s\in (n/2,\infty)$. Let $x\in \omega$ and $d_0>0$ so that $B(x,d_0)\subset \omega$ and $B(x,4d_0)\subset \Omega$. Let $0<r<r_\ast:=\min(d_0,\mathfrak{r}/4)$, $V\in L^s(\Omega)$ and $u\in W^{2,p}(\Omega)$ satisfying $(-\Delta +V)u=0$. Applying \eqref{qu10} in Theorem \ref{mthm2}, we obtain
\begin{equation}\label{ap1}
\|u\|_{L^2(B(y,r))}\le \mathcal{C}\|u\|_{L^2(\omega)}^{\eta(r)}\|u\|_{L^2(\Omega)}^{1-\eta(r)},\quad y\in \Omega^{4r}.
\end{equation}
Here and in the sequel, $\mathcal{C}\ge 1$ is a generic constant of the form $\mathbf{c}\varphi_s(V)$. 

As $\Omega^{4r}$ can be covered by $n_r$ balls with center in $\Omega$ with $n_r\le \hat{c}\, r^{-n}$ for some constant $\hat{c}=\hat{c}\, (n,\mathbf{D})>0$ (see Lemma \ref{lemA1}), we deduce from \eqref{ap1}
\[
\|u\|_{L^2(\Omega^{4r})}\le \hat{c}\, \mathcal{C}r^{-n}\|u\|_{L^2(\omega)}^{\eta(r)}\|u\|_{L^2(\Omega)}^{1-\eta(r)}.
\]
For all $\epsilon >0$, Young's inequality then implies
\begin{equation}\label{ap2}
\mathcal{C}^{-1}\|u\|_{L^2(\Omega^{4r})}\le \left[\hat{c}\, r^{-n}\epsilon^{-1}\right]^{1/\eta(r)}\|u\|_{L^2(\omega)}+\epsilon^{1/(1-\eta(r))}\|u\|_{L^2(\Omega)}.
\end{equation}

On the other hand, we have from Hardy's inequality (e.g. \cite[(8)]{Ch20})
\begin{equation}\label{ap3}
\|u\|_{L^2(\Omega_{4r})}\le c_\ast r^\mathfrak{t} \|u\|_{H^1(\Omega)},
\end{equation}
where $c_\ast=c_\ast\, (n,\Omega,\mathfrak{t})>0$ is a constant.

Putting together \eqref{ap2} and \eqref{ap3}, we obtain
\begin{equation}\label{ap4}
\mathcal{C}^{-1}\|u\|_{L^2(\Omega)}\le \left[\hat{c}\, r^{-n}\epsilon^{-1}\right]^{1/\eta(r)}\|u\|_{L^2(\omega)}+(c_\ast r^\mathfrak{t}+\epsilon^{1/(1-\eta(r))})\|u\|_{H^1(\Omega)}.
\end{equation}
Taking in \eqref{ap4} $\epsilon=(c_\ast r^\mathfrak{t})^{1-\eta(r)}$, we get
\begin{align*}
\mathcal{C}^{-1}\|u\|_{L^2(\Omega)}&\le \left[\hat{c}\, r^{-n}(c_\ast r^\mathfrak{t})^{-(1-\eta(r))}\right]^{1/\eta(r)}\|u\|_{L^2(\omega)}+2c_\ast r^\mathfrak{t} \|u\|_{H^1(\Omega)}
\\
&\le c_\ast r_\ast ^\mathfrak{t} \left[\hat{c}\, c_\ast^{-1}r^{-(n+\mathfrak{t})}\right]^{1/\eta(r)}\|u\|_{L^2(\omega)}+2c_\ast r^\mathfrak{t} \|u\|_{H^1(\Omega)}
\\
&\le  2c_\ast\left(\left[\hat{c}\, c_\ast ^{-1}r^{-(n+\mathfrak{t})}\right]^{1/\eta(r)}\|u\|_{L^2(\omega)}+ r^\mathfrak{t}\|u\|_{H^1(\Omega)} \right).
\end{align*}
The expected inequality then follows by using
\begin{align*}
\left[\hat{c}\, c_\ast ^{-1}r^{-(n+\mathfrak{t})}\right]^{1/\eta(r)}&\le \left[\hat{c}\, c_\ast ^{-1}\right]^{e^{\mathfrak{h}r^{-n}}}e^{(1+\mathfrak{t}/n)e^{r^{-n}}e^{\mathfrak{h}r^{-n}}}
\\
&\le e^{\hat{c}\, c_\ast ^{-1}e^{\mathfrak{h}r^{-n}}}e^{(1+\mathfrak{t}/n)e^{(\mathfrak{h}+1)r^{-n}}}
\\
&\le e^{(\hat{c}\, c_\ast ^{-1}+1+\mathfrak{t}/n)e^{(\mathfrak{h}+1)r^{-n}}}.
\end{align*}

The case $s=\infty$ is proved similarly.
\end{proof}

For $V\in L^s(\Omega)$, $s\in (n/2,\infty]$, let 
\[
f(r)= r^{-\mathfrak{t}}e^{\bar{\mathbf{c}}\varrho(r)}=r^{-\mathfrak{t}}e^{\bar{\mathbf{c}}e^{\bar{\mathfrak{h}}r^{-n}}},\quad r>0.
\]
Clearly, $f$ is bijective and decreasing from $(0,r_\ast)$ onto $(f(r_\ast),\infty)$, where $r_\ast$ is as in Theorem \ref{thma2}.

Let $\hat{\mathfrak{h}}:=\bar{\mathfrak{h}}+1$. Since $\bar{\mathbf{c}}=\bar{\mathbf{c}}(n,\Omega,\mathfrak{t})>0$ denotes a generic constant, we verify that
\[
f(r)\le e^{\bar{\mathbf{c}}e^{\hat{\mathfrak{h}} r^{-n}}}=:g(r),\quad r>0.
\]
Set
\[
\mathbf{a}:=g(r_\ast).
\]
Let $s\in (n/2,\infty]$,  $V\in L^s(\Omega)$ and $u\in X_s(\Omega)\setminus\{0\}$ satisfying $(-\Delta +V)u=0$. Assume  first that $\|u\|_{H^1(\Omega)}/\|u\|_{L^2(\omega)}>\mathbf{a}$. In this case there exists $\bar{r}\in (0,r_\ast)$ so that 
\[
\|u\|_{H^1(\Omega)}/\|u\|_{L^2(\omega)}=f(\bar{r})\le g(\bar{r}).
\] 
Hence 
\[
\bar{r}\le  \hat{\mathfrak{h}}^{1/n}\left\{ \ln \ln \left( \|u\|_{H^1(\Omega)}/\|u\|_{L^2(\omega)}\right)^{\bar{\mathbf{c}}} \right\}^{-1/n}.
\]
Here we used  $\ln \left( \|u\|_{H^1(\Omega)}/\|u\|_{L^2(\omega)}\right)^{\bar{\mathbf{c}}}\ge e^{\hat{\mathfrak{h}}r_\ast^{-n}}>1$. Taking $r=\bar{r}$ in \eqref{ap5}, we obtain
\begin{equation}\label{int2}
\|u\|_{L^2(\Omega)}\le \mathbf{c}\varphi_s(V)\|u\|_{H^1(\Omega)} \hat{\mathfrak{h}}^{\mathfrak{t}/n}\left\{  \ln \ln \left( \|u\|_{H^1(\Omega)}/\|u\|_{L^2(\omega)}\right)^{\bar{\mathbf{c}}} \right\}^{-\mathfrak{t}/n}.
\end{equation}

In the case $0<\|u\|_{H^1(\Omega)}/\|u\|_{L^2(\omega)}\le \mathbf{a}$, we have
\begin{align}\label{int3}
&\|u\|_{L^2(\Omega)}\le\|u\|_{H^1(\Omega)}\le \mathbf{a}\|u\|_{L^2(\omega)}
\\
&\hskip 4cm =\mathbf{a}\|u\|_{H^1(\Omega)} \left(\|u\|_{H^1(\Omega)}/\|u\|_{L^2(\omega)}\right)^{-1}.\nonumber
\end{align}

Define the generic function
\begin{equation}\label{psi}
\Psi (x):=\mathbf{a}\chi_{(0,\mathbf{a}]}(x)x^{-1}+  \chi_{(\mathbf{a},\infty)}(x) \hat{\mathfrak{h}}^{\mathfrak{t}/n}\left\{ \ln  \ln  x^{\bar{\mathbf{c}}} \right\}^{-\mathfrak{t}/n}.
\end{equation}
Here $\chi_I$ denotes the characteristic function of the interval $I$.

Putting together \eqref{int2} and \eqref{int3}, we end up getting the following result.

\begin{corollary}\label{cora1}
Let $s\in (n/2,\infty]$, $0<\mathfrak{t}<1/2$, $\omega\Subset \Omega$ and let $r_\ast$ be as in Theorem \ref{thma2}. For all $V\in L^s(\Omega)$, $0<r<r_\ast$ and $u\in X_s\setminus\{0\}$ satisfying $(-\Delta +V)u=0$ we have
\[
\|u\|_{L^2(\Omega)}\le \mathbf{c}\varphi_s(V) \|u\|_{H^1(\Omega)}\Psi(\|u\|_{H^1(\Omega)}/\|u\|_{L^2(\omega)}).
\]
\end{corollary}

We now consider the case of a potential in $L^{n/2}(\Omega)$. Let $0<\mathfrak{t}<1/2$, $V\in \mathscr{V}$ and $u\in W^{2,p}(\Omega)$ satisfying $(-\Delta+V)u=0$. Keeping the same notations as in the proof of Theorem \ref{thma2}, we mimic the proof of this later to prove
\begin{align}\label{a1.6}
&q_V^{-1/(1-\alpha)}\|u\|_{L^2(\Omega)}\le \left[\hat{c}\, r^{-(n+1/(1-\alpha))}(c_\ast r^\mathfrak{t})^{-(1-\eta(r))}\right]^{1/\eta(r)}\|u\|_{L^2(\omega)}
\\
&\hskip 8cm +2c_\ast r^\mathfrak{t} \|u\|_{H^1(\Omega)},\nonumber
\end{align}
where $q_V$ is given by \eqref{qV}.

Let
\[
\mathfrak{q}_V:=q_V^{1/(1-\alpha)}.
\]
In that case \eqref{a1.6} gives
\[
\|u\|_{L^2(\Omega)}\le \mathfrak{q}_V\left(e^{\bar{\mathbf{c}}e^{\bar{\mathfrak{h}}r^{-n}}}\|u\|_{L^2(\omega)}+ r^\mathfrak{t} \|u\|_{H^1(\Omega)}\right),
\]
where $\bar{\mathfrak{h}}$ is given by \eqref{tildeh}.

In other words, we proved the following result.
 
\begin{theorem}\label{thmsc1}
Let $0<\mathfrak{t}<1/2$ and $\omega\Subset \Omega$. Let $r_\ast$ be as in Theorem \ref{thma2}. For all $V\in \mathscr{V}$, $0<r<r_\ast$ and $u\in W^{2,p}(\Omega)$ satisfying $(-\Delta +V)u=0$ we have
\begin{equation}\label{a1.7}
\|u\|_{L^2(\Omega)}\le \mathfrak{q}_V\left(e^{\bar{\mathbf{c}}e^{\bar{\mathfrak{h}}r^{-n}}}\|u\|_{L^2(\omega)}+ r^\mathfrak{t} \|u\|_{H^1(\Omega)}\right).
\end{equation}
\end{theorem}

Similarly to Corollary \ref{cora1}, we prove the following result
\begin{corollary}\label{corsc1}
Let $0<\mathfrak{t}<1/2$ and $\omega\Subset \Omega$. Let $r_\ast$ be as in Theorem \ref{thma2}. For all $V\in \mathscr{V}$, $0<r<r_\ast$ and $u\in W^{2,p}(\Omega)\setminus\{0\}$ satisfying $(-\Delta +V)u=0$ we have
\[
\|u\|_{L^2(\Omega)}\le \mathfrak{q}_V\|u\|_{H^1(\Omega)}\Psi\left(\|u\|_{H^1(\Omega)}/\|u\|_{L^2(\omega)}\right),
\]
where $\Psi$ is as in \eqref{psi}.
\end{corollary}

\section{Global quantitative uniqueness of continuation for $C^{1,1}$ domains}\label{S6}

In this section, $\Omega$ is a bounded domain of class $C^{1,1}$. The unit normal vector field on $\Gamma:=\partial \Omega$ is denoted by $\nu$.

The following two auxiliary results will be useful in the sequel.

\begin{proposition}\label{propositiona2}
{\rm (\cite[Proposition 2.1]{Ch19})} There exists $\dot{\varrho}=\dot{\varrho}(n,\Omega)>0$  such that :
\\
$\mathrm{(i)}$ For any $x\in \Omega_{\dot{\varrho}}$, there exists a unique $\mathfrak{p}(x)\in \Gamma$ such that
\[
|x-\mathfrak{p}(x)|=\mathrm{dist}(x,\Gamma),\quad x= \mathfrak{p}(x)-|x-\mathfrak{p}(x)|\nu(\mathfrak{p}(x)).
\]
$\mathrm{(ii)}$ If $x\in \Omega_{\dot{\varrho}}$ then $x_t=\mathfrak{p}(x)-t|x-\mathfrak{p}(x)|\nu(\mathfrak{p}(x))\in \Omega_{\dot{\varrho}}$, $t\in ]0,1]$, and
\[
\mathfrak{p}(x_t)=\mathfrak{p}(x),\; |x_t-\mathfrak{p}(x)|=t\mathrm{dist}(x,\Gamma).
\]
\end{proposition} 

Furthermore, using that a $C^{1,1}$ domain admits the uniform interior sphere property (e.g. \cite{Ba}),  we obtain from Proposition \ref{propositiona2} the following lemma.

\begin{lemma}\label{lemmaa1}
If $\dot{\varrho}$ is as in Proposition \ref{propositiona2} then there exists $0<\bar{\varrho}\le \dot{\varrho}$ and $\theta \in (0,\pi/2)$ so that for any $x\in \Omega_{\bar{\varrho}}$ we have 
\[
\mathscr{C}(\mathfrak{p}(x)):=\left\{y\in \mathbb{R}^n;\; 0<|y-\mathfrak{p}(x)|<\bar{\varrho},\; (y-\mathfrak{p}(x))\cdot \xi >|y-\mathfrak{p}(x)|\cos \theta \right\}\subset \Omega ,
\]
where $\mathfrak{p}(x)\in \Gamma$ is as in Proposition \ref{propositiona2} and  $\xi=-\nu (\mathfrak{p}(x))$.
\end{lemma}

For convenience, we consider the following convention. If a constant depends on $(a_1,s,a_2)$ with $s\in (n/2,\infty)$ we include the case $s=\infty$ with the sense that the constant depends on $(a_1,a_2)$. Also, we recall that $0<\mathfrak{r}=\mathfrak{r}(\Omega)\le 1$ is as in Section \ref{S1} and for which $\Omega^r:=\{x\in \Omega;\; \mathrm{dist}(x,\Gamma)>r\}$ is nonempty and connected for all $0<r<\mathfrak{r}$.

As in the preceding section, $X_s:=W^{2,p}(\Omega)$ if $s\in (n/2,\infty)$ and $X_\infty:=H^2(\Omega)$.

\begin{theorem}\label{thmaa1}
Let $\omega\Subset \Omega$, $s\in (n/2,\infty]$, $0<\mathfrak{t} <1/2$. Let $\zeta:=(n,\Omega,\omega,\mathfrak{r},\mathfrak{t})$. There exist three constants $\upsilon =\upsilon(\zeta)>0$, $\tilde{\varsigma}=\tilde{\varsigma}(\zeta)>0$ and $\tilde{\mathbf{c}}=\tilde{\mathbf{c}}(\zeta,s)>0$ such that for all $V\in L^s(\Omega)$, $u\in X_s$ satisfying $(-\Delta+V)u=0$ and $0<r<1$ we have
\begin{equation}\label{aa1}
\|u\|_{L^2(\Omega)}\le \tilde{\mathbf{c}}\varphi_s(V)\left(e^{\upsilon r^{-\tilde{\varsigma}}}\|u\|_{L^2(\omega)}+r^\mathfrak{t} \|u\|_{H^1(\Omega)}\right),
\end{equation}
where the constant $\varphi_s(V)$ is as \eqref{varphis}.
\end{theorem}

\begin{proof}
Let $\bar{\varrho}$ be as in Lemma \ref{lemmaa1}. Without loss of generality, we assume that $\bar{\varrho}\le 3$. Let $0<r<\bar{\varrho}/3\; (\le 1)$ and $x\in \Omega^{r}\setminus \Omega^{\bar{\varrho}/3}$. Set $d:=\mathrm{dist}(x,\Gamma)$, $\tilde{x}:=\mathfrak{p}(x)$ and
\[
x_0:=\tilde{x}-(d+\tilde{d})\nu (\tilde{x}).
\]
Choose $\tilde{d}>0$  in such a way that $2\bar{\varrho}/3<d+\tilde{d}<\bar{\varrho}$. In this case, $x_0\in \Omega_{\bar{\varrho}}$ and $B(x_0,\bar{\varrho}/3)\subset \Omega^{\bar{\varrho}/3}$. Since $x_0-\tilde{x}$ is colinear to $\nu(\tilde{x})$ we derive from Proposition \ref{propositiona2} that $\tilde{x}_0=\tilde{x}$, $\nu(\tilde{x}_0)=\nu(\tilde{x})$ and $d_0:=\mathrm{dist}(x_0,\Gamma)=d+\tilde{d}$. 

Let $\mathscr{C}(\tilde{x})$ be the cone given by Lemma \ref{lemmaa1} and $\rho_0:=\min [(d_0/3) \sin \theta,\bar{\varrho}/9]$. This choice of $\rho_0$ guarantees  that $B(x_0,3\rho_0)\subset \mathscr{C}(\tilde{x})\cap B(x_0,\bar{\varrho}/3)$. 

By reducing $\theta$ if necessary, we can assume that $\sin \theta \le 1/3$. Hence $\rho_0=(d_0/3) \sin \theta$. Define then the sequence of balls $(B(x_k, 3\rho _k))$ as follows
\begin{eqnarray*}
\left\{
\begin{array}{ll}
x_{k+1}=x_k-\alpha _k \xi ,
\\
\rho_{k+1}=\mu \rho_k ,
\\
d_{k+1}=\mu d_k,
\end{array}
\right.
\end{eqnarray*}
where
\[
d_k:=|x_k-\tilde{x}|,\quad \rho _k:=\varpi d_k,\quad \alpha _k:=(1-\mu)d_k ,
\]
with
\[
\varpi :=\frac{\sin \theta}{3},\quad \mu :=\frac{3-2\sin \theta}{3-\sin \theta}.
\]
This definition guarantees that 
\[
B(x_k,3\rho _k)\subset \mathscr{C}(\tilde{x}) \quad {\rm and}\quad  B(x_{k+1},\rho _{k+1})\subset B(x_k,2\rho _k),\quad k\ge 1.
\]
Furthermore, by induction in $k$, we have
\begin{equation}\label{u1}
x_k=x_0-(1-\mu^k)d_0\xi=\tilde{x}+d_0\xi-(1-\mu^k)d_0\xi,\quad k\ge 0.
\end{equation}
and since $x=\tilde{x}+d\xi$, we obtain
\[
x-x_k=-\tilde{d}\xi+(1-\mu^k)d_0\xi,\quad k\ge 0 .
\]
Hence
\begin{equation}\label{k1}
|x-x_k|=|-\tilde{d}+(1-\mu^k)d_0|,\quad k\ge 0.
\end{equation}

Let $V\in L^s(\Omega)$, $u\in X_s$ satisfying $(-\Delta+V)u=0$ and  $\epsilon>0$. Proceeding as in the proof of Theorem \ref{mthm2} and using Young's inequality, we obtain for all $k\ge 1$
\begin{equation}\label{a10}
[\mathbf{c}\varphi_s(V)]^{-1}\|u\|_{L^2(B(x_k,\rho _k))}\le \epsilon^{-\frac{1}{\alpha^k}} \|u\|_{L^2(B(x_0,\rho _0))} +\epsilon^{\frac{1}{1-\alpha ^k}}\|u\|_{L^2(\Omega)}.
\end{equation}

By \eqref{u1}, the sequence $(x_k)$ converge to $\tilde{x}$ and since
\begin{equation}\label{u2}
\rho_{k-1}+\rho_{k}-|x_{k}-x_{k-1}|=\sin \theta \mu d_{k-1},\quad k\ge1,
\end{equation}
we deduce that the union of the balls $B(x_k,\rho _k)$ contains the open segment connecting $\tilde{x}$ to $x_0$. We denote by $k_x$ the greatest positive integer so that $x\in B(x_k,\rho_k)$ (note that by \eqref{u2}, $B(x_k,\rho_k)\cap B(x_{k-1},\rho _{k-1})$ is nonempty for each $k\ge 1$). In light of \eqref{k1}, we obtain
\[
|x-x_{k_x}|=|-\tilde{d}+(1-\mu^{k_x})d_0|\le \rho _{k_x}=\mu^{k_x}\varpi d_0.
\]
Whence
\[
(\varpi +1)^{-1}dd_0^{-1} \le \mu^{k_x}=e^{-|\ln \mu|k_x}.
\]
%\\
Furthermore, as $x\not\in B(x_{k_x+1},\rho _{k_x+1})$, we have similarly as above
\[
(\varpi +1)^{-1}dd_0^{-1} \ge e^{-|\ln \mu|(k_x+1)}.
\]
Then, using  
\[
1\ge dd_0^{-1}=\frac{d}{d+\tilde{d}}\ge \frac{r}{\bar{\varrho}},
\]
we get
\begin{align*}
&e^{|\ln \mu|k_x}\le (\varpi +1)\bar{\varrho}/r,
\\
&e^{|\ln \mu| (k_x+1)}\ge \varpi +1.
\end{align*}
That is we have
\begin{align}
&k_x \le |\ln \mu|^{-1}\ln (\bar{b}/r)=:h(r), \label{a10.1}
\\
& k_x \ge  k_+:=\max(1, |\ln \mu|^{-1}\ln b -1),\label{a10.1.1}
\end{align}
where we set $b:=\varpi +1$ and $\bar{b}:=b \bar{\varrho}$.

Now, \eqref{a10} with $k=k_x$ yields
\begin{equation}\label{a13}
[\mathbf{c}\varphi_s(V)]^{-1}\|u\|_{L^2(B(x_{k_x},\rho _{k_x}))}\le \epsilon^{-\frac{1}{\alpha^{k_x}}} \|u\|_{L^2(B(x_0,\bar{\varrho}/3))}
+\epsilon^{\frac{1}{1-\alpha ^{k_x}}}\|u\|_{L^2(\Omega)}.
\end{equation}
Replacing in \eqref{a13} $\epsilon$ by $\epsilon^{\alpha^{k_x}}$, we find
\begin{equation}\label{a13.1}
[\mathbf{c}\varphi_s(V)]^{-1}\|u\|_{L^2(B(x_{k_x},\rho _{k_x}))}\le \epsilon^{-1} \|u\|_{L^2(B(x_0,\bar{\varrho}/3))}
+\epsilon^{\frac{\alpha^{k_x}}{1-\alpha ^{k_x}}}\|u\|_{L^2(\Omega)}.
\end{equation}

Let $r_\ast:=r_\ast (\Omega,\omega,\mathfrak{r})>0$ be as in the proof of Theorem \ref{thma2}. Reducing again $\bar{\varrho}$, we assume that $\bar{\varrho}/3<r_\ast$. Then similarly to \eqref{ap1}, we have
\[
\|u\|_{L^2(B(x_0,\bar{\varrho}/3))}\le \mathbf{c}\varphi_s(V)\|u\|_{L^2(\omega)}^{\mathfrak{s}}\|u\|_{L^2(\Omega)}^{1-\mathfrak{s}},\
\]
where $\mathfrak{s}:=\eta(\bar{\varrho}/3)$. Young's inequality then yields 
\begin{equation}\label{a13.2}
[\mathbf{c}\varphi_s(V)]^{-1}\|u\|_{L^2(B(x_0,\bar{\varrho}/3))}\le \epsilon^{-1}\|u\|_{L^2(\omega)}+\epsilon^{\mathfrak{s}/(1-\mathfrak{s})}\|u\|_{L^2(\Omega)}.
\end{equation}

Putting together \eqref{a13.1} and \eqref{a13.2}, we obtain
\begin{equation}\label{a13.4}
[\mathbf{c}\varphi_s(V)]^{-1}\|u\|_{L^2(B(x_{k_x},\rho _{k_x}))}\le \epsilon^{-1}\|u\|_{L^2(\omega)}
 +(\epsilon^{\mathfrak{s}/(1-\mathfrak{s})}+\epsilon^{\frac{\alpha^{k_x}}{1-\alpha ^{k_x}}})\|u\|_{L^2(\Omega)}.
\end{equation}
Let $\tau >0$ to be specified later and choose $\epsilon$ so that $\epsilon^{\frac{\alpha^{k_x}}{1-\alpha ^{k_x}}}=r^{\tau}$, that is, $\epsilon=r^{\tau(1-\alpha ^{k_x})/\alpha ^{k_x}}=r^{\tau(1/\alpha ^{k_x}-1)}$.  We have
\begin{align*}
e^{h(r)|\ln \alpha |}&=e^{|\ln \mu|^{-1}\ln (\bar{b}/r) |\ln \alpha |}
\\
&=e^{|\ln \mu|^{-1}\ln \bar{b} |\ln \alpha |}e^{-|\ln \mu|^{-1}\ln r |\ln \alpha |}
\\
&=\upsilon r^{-\varsigma},
\end{align*}
where we set $\upsilon:=e^{|\ln \mu|^{-1}\ln \bar{ b} |\ln \alpha |}$ and $\varsigma:=|\ln \alpha|/|\ln \mu|$. Since
\[
1/\alpha ^{k_x}=e^{k_x|\ln \alpha|}\le e^{h(r)|\ln \alpha|},
\]
we obtain
\[
1/\alpha ^{k_x}\le \upsilon r^{-\varsigma}.
\]
Therefore
\begin{equation}\label{eps1}
\epsilon^{-1}=r^{\tau(1-1/\alpha ^{k_x})}=r^{\tau}e^{1/\alpha ^{k_x}\ln r^{-\tau}}\le e^{\upsilon r^{-\varsigma}\ln r^{-\tau}}\le e^{\upsilon r^{-(\varsigma+\tau)}}.
\end{equation}
Also, we have
\begin{equation}\label{eps2}
\epsilon=r^{\tau (1/\alpha ^{k_x}-1)}=r^{-\tau}r^{\tau/\alpha ^{k_x}}\le r^{-\tau}r^{\tau (1/\alpha)^{k_+}}=r^{\tau \mathfrak{a}},
\end{equation}
where $\mathfrak{a}:=(1/\alpha)^{k_+}-1>0$.

Let $\mathfrak{b}:=\min (\mathfrak{a}\mathfrak{s}/(1-\mathfrak{s}),1)$. Then \eqref{eps1} and \eqref{eps2} in \eqref{a13.4} gives
\begin{equation}\label{a13.5}
[\mathbf{c}\varphi_s(V)]^{-1}\|u\|_{L^2(B(x_{k_x},\rho _{k_x}))}\le e^{\upsilon r^{-(\varsigma+\tau)}}\|u\|_{L^2(\omega)}
 +r^{\tau \mathfrak{b}}\|u\|_{L^2(\Omega)}.
\end{equation}

Next, as
\[
\rho_{k_x+1}+\rho_{k_x}-|x_{k_x}-x_{k_x+1}|=\sin \theta \mu d_{k_x}=:r_x,
\]
and $x\in B(x_{k_x},\rho _{k_x})\setminus B(x_{k_x+1},\rho _{k_x+1})$, we get $B(x,r_x)\subset B(x_{k_x},\rho _{k_x})$.
But
\[
r_x=\sin\theta  d_0\mu^{k_x+1}\ge \mu(2\sin\theta \bar{\varrho}/3)e^{-\ln (\bar{\varrho}/r)}=2\mu\varpi r.
\]
Therefore $B(x,2\mu\varpi r)\subset B(x_{k_x},\rho _{k_x})$.
Applying Lemma \ref{lemA1} to derive that $\Omega^r\setminus \Omega^{\bar{\varrho}/3}$ can be covered by $\ell$ balls $B(x,2\mu\varpi r)$ centered at $x\in\Omega^r\setminus \Omega^{\bar{\varrho}/3}$ with $\ell\le \hat{c}r^{-n}$, where $\hat{c}=\hat{c}(n,\Omega,\bar{\varrho})>0$. In consequence, \eqref{a13.5} gives
\begin{equation}\label{a13.6}
\hat{c}^{-1}[\mathbf{c}\varphi_s(V)]^{-1}\|u\|_{L^2(\Omega^r\setminus \Omega^{\bar{\varrho}/3})}\le r^{-n}e^{\upsilon r^{-(\varsigma+\tau)}}\|u\|_{L^2(\omega)}
 +r^{-n+\tau \mathfrak{b}}\|u\|_{L^2(\Omega)}.
\end{equation}
In \eqref{a13.6}, by choosing $\tau$ in such a way that $-n+\tau\mathfrak{b}=\mathfrak{t}$, that is $\tau=(n+\mathfrak{t})\mathfrak{b}^{-1}$, we get
\begin{equation}\label{a13.7}
\hat{c}^{-1}[\mathbf{c}\varphi_s(V)]^{-1}\|u\|_{L^2(\Omega^r\setminus \Omega^{\bar{\varrho}/3})}\le r^{-n}e^{\upsilon r^{-\tilde{\varsigma}}}\|u\|_{L^2(\omega)}
 +r^\mathfrak{t}\|u\|_{L^2(\Omega)},
\end{equation}
where $\tilde{\varsigma}=\tilde{\varsigma}(\zeta):=\varsigma +(n+\mathfrak{t})\mathfrak{b}^{-1}$.

Since $r<\bar{\varrho}/3$, we obtain from \eqref{ap2} with $r:=\bar{\rho}/12$ and $\epsilon:=r^{\mathfrak{t}(1-\eta(\bar{\rho}/12))}$
\begin{equation}\label{ap2.0}
\hat{c}^{-1}[\mathbf{c}\varphi_s(V)]^{-1}\|u\|_{L^2(\Omega^{\bar{\rho}/3})}\le 4^{n/\mathfrak{s}'}r^{-\varsigma'}\|u\|_{L^2(\omega)}+r^{\mathfrak{t}}\|u\|_{L^2(\Omega)},
\end{equation}
where $\mathfrak{s}'=\mathfrak{s}'(n,\Omega,\omega,\mathfrak{r}):=\eta(\bar{\varrho}/12)$ and $\varsigma'=\varsigma'(\zeta):=(n+\mathfrak{t}(1-\mathfrak{s}'))/\mathfrak{s}'$.

Henceforth, $\tilde{\mathbf{c}}=\tilde{\mathbf{c}}(\zeta,s)>0$ and $\upsilon=\upsilon(\zeta)>0$ will denote generic constants. Then, \eqref{a13.7} and \eqref{ap2.0} give
\begin{equation}\label{ap2.1}
[\tilde{\mathbf{c}}\varphi_s(V)]^{-1}\|u\|_{L^2(\Omega^r)}\le e^{\upsilon r^{-\tilde{\varsigma}}}\|u\|_{L^2(\omega)}+r^{\mathfrak{t}}\|u\|_{L^2(\Omega)}.
\end{equation}

On the other hand, \eqref{ap3} in which we replaced $4r$ by $r$ yields
\begin{equation}\label{a13.8}
\|u\|_{L^2(\Omega_{r})}\le c_\ast r^\mathfrak{t} \|u\|_{H^1(\Omega)},
\end{equation}
where we recall that $c_\ast=c_\ast\, (n,\Omega,\mathfrak{t})>0$ is a constant. The expected inequality follows by combining \eqref{ap2.1} and \eqref{a13.8} and replacing $r$ by $3r\bar{\varrho}$.
\end{proof}

The proof of Theorem \ref{thmaa1} is an adaptation of that of \cite[Theorem 1.1]{Ch24_2} with additional details.

Using the same notations as in the preceding proof, we set $\upsilon':=\upsilon+\mathfrak{t}/\tilde{\varsigma}$, $\upsilon'':=(\upsilon')^{\mathfrak{t}/\tilde{\varsigma}}$ and
\[
\mathbf{F}(x):=e^{\upsilon'}x^{-1}\chi_{(0,e^{\upsilon'}]}(x)+\upsilon''(\ln x)^{-\mathfrak{t}/\tilde{\varsigma}}\chi_{(e^{\upsilon'},\infty)}(x).
\]

We proceed similarly to the proof of Corollary \ref{cora1} to obtain the following result.

\begin{corollary}\label{coraa1}
Let $\omega\Subset \Omega$, $s\in (n/2,\infty]$, $0<\mathfrak{t} <1/2$. There exists a constant $\tilde{\mathbf{c}}=\tilde{\mathbf{c}}(n,\Omega,\omega, \mathfrak{r},s,\mathfrak{t})>0$ such that for all $V\in L^s(\Omega)$, $u\in X_s\setminus\{0\}$ satisfying $(-\Delta+V)u=0$ we have
\[
\|u\|_{L^2(\Omega)}\le \tilde{\mathbf{c}}\varphi_s(V)\mathbf{F}\left(\|u\|_{H^1(\Omega)}/\|u\|_{L^2(\omega)}\right)\|u\|_{H^1(\Omega)}.
\]
\end{corollary}

In the remainder of this section, we give the main steps to show how one can obtain a result similar to Theorem \ref{thmaa1} in the case $s=n/2$. We keep hereafter the notations of the proof of Theorem \ref{thmaa1}.

Let $V\in L^{n/2}(\Omega)$ and $u\in W^{2,p}(\Omega)$ satisfying $(-\Delta+V)u=0$. First, in light of three-ball inequality of Theorem \ref{thmsc2},  for all $k\ge 0$, we have instead of \eqref{a10}
\begin{equation}\label{Aa10}
[q_Vr^{-1}]^{-1/(1-\alpha)}\|u\|_{L^2(B(x_k,\rho _k))}\le \epsilon^{-\frac{1}{\alpha^k}} \|u\|_{L^2(B(x_0,\rho _0))} +\epsilon^{\frac{1}{1-\alpha ^k}}\|u\|_{L^2(\Omega)},
\end{equation}
where $q_V$ is given by \eqref{qV}. Using \eqref{a1.3}, we obtain
\begin{equation}\label{Aa13.2}
[q_Vr^{-1}]^{-1/(1-\alpha)}\|u\|_{L^2(B(x_0,\bar{\varrho}/3))}\le \epsilon^{-1}\|u\|_{L^2(\omega)}+\epsilon^{\mathfrak{s}/(1-\mathfrak{s})}\|u\|_{L^2(\Omega)}.
\end{equation}
Here, $\mathfrak{s}:=\eta(\bar{\varrho}/3)$. Putting together \eqref{Aa10} and \eqref{Aa13.2}, we get similarly as for \eqref{a13.5}
\begin{equation}\label{Aa13.5}
[q_Vr^{-1}]^{-2/(1-\alpha)}\|u\|_{L^2(B(x_{k_x},\rho _{k_x}))}\le e^{\upsilon r^{-(\varsigma+\tau)}}\|u\|_{L^2(\omega)}
 +r^{\tau \mathfrak{b}}\|u\|_{L^2(\Omega)}.
\end{equation}
Here we have kept the same notations as in \eqref{a13.5} by simply replacing $\alpha$ of \eqref{a13.5} with that of \eqref{Aa10}.

Pursuing as in the proof of Theorem \ref{thmaa1}, we obtain
\[
\hat{c}^{-1}[q_Vr^{-1}]^{-2/(1-\alpha)}\|u\|_{L^2(\Omega^r\setminus \Omega^{\bar{\varrho}/3})}\le r^{-n}e^{\upsilon r^{-(\varsigma+\tau)}}\|u\|_{L^2(\omega)}
 +r^{-n+\tau \mathfrak{b}}\|u\|_{L^2(\Omega)}.
\]
That is we have
\begin{align}
&\hat{c}^{-1}q_V^{-2/(1-\alpha)}\|u\|_{L^2(\Omega^r\setminus \Omega^{\bar{\varrho}/3})}\le r^{-2/(1-\alpha)-n}e^{\upsilon r^{-(\varsigma+\tau)}}\|u\|_{L^2(\omega)}\label{Aa13.6}
\\
&\hskip 6.5cm +r^{-2/(1-\alpha)-n+\tau \mathfrak{b}}\|u\|_{L^2(\Omega)}.\nonumber
\end{align}

Set $\tilde{q}_V:=q_V^{2/(1-\alpha)}$ and $\tilde{\varsigma}=\tilde{\varsigma}(\zeta):=\varsigma+\mathfrak{b}^{-1}(\mathfrak{t}+n+2(1-\alpha)^{-1})$, where $\zeta:=(n,\Omega,\omega,\mathfrak{r},\mathfrak{t})$.
By choosing $\tau=\mathfrak{b}^{-1}(\mathfrak{t}+n+2(1-\alpha)^{-1})$ in \eqref{Aa13.6}, we obtain
\[
\hat{c}^{-1}\tilde{q}_V^{-1}\|u\|_{L^2(\Omega^r\setminus \Omega^{\bar{\varrho}/3})}\le e^{\upsilon r^{-\tilde{\varsigma}}}\|u\|_{L^2(\omega)}+r^{\mathfrak{t}}\|u\|_{L^2(\Omega)},
\]
where $\upsilon=\upsilon(\zeta)>0$ denotes a generic constant. In the same way as when \eqref{ap2.1} is obtained, we have
\[
[\tilde{\mathbf{c}}\tilde{q}_V]^{-1}\|u\|_{L^2(\Omega^r)}\le e^{\upsilon r^{-\tilde{\varsigma}}}\|u\|_{L^2(\omega)}+r^{\mathfrak{t}}\|u\|_{L^2(\Omega)},
\]
where $\tilde{\mathbf{c}}=\tilde{\mathbf{c}}(\zeta)>0$ is a generic constant.
This and Hardy's inequality \eqref{a13.8} yields
\[
\|u\|_{L^2(\Omega)}\le \tilde{\mathbf{c}}\tilde{q}_V \left(e^{\upsilon r^{-\tilde{\varsigma}}}\|u\|_{L^2(\omega)}+r^\mathfrak{t}\|u\|_{H^1(\Omega)}\right).
\]

Summing up, we obtain the following result.

\begin{theorem}\label{thmaa1.1}
Let $\omega\Subset \Omega$, $0<\mathfrak{t} <1/2$  and $\zeta:=(n,\Omega,\omega, \mathfrak{r},\mathfrak{t})$. There exist three constants $\upsilon =\upsilon(\zeta)>0$, $\tilde{\varsigma}=\tilde{\varsigma}(\zeta)>0$ and $\tilde{\mathbf{c}}=\tilde{\mathbf{c}}(\zeta)>0$ such that for all $V\in \mathscr{V}$, $u\in W^{2,p}(\Omega)\setminus\{0\}$ satisfying $(-\Delta+V)u=0$  and $0<r<1$ we have
\begin{equation}\label{xx}
\|u\|_{L^2(\Omega)}\le \tilde{\mathbf{c}}\tilde{q}_V \left(e^{\upsilon r^{-\tilde{\varsigma}}}\|u\|_{L^2(\omega)}+r^\mathfrak{t}\|u\|_{H^1(\Omega)}\right),
\end{equation}
\end{theorem}

We can then proceed as in the proof of Corollary \ref{coraa1} to derive from \eqref{xx} a logarithmic stability inequality.

\section{Doubling inequality}\label{S7}

We begin with the case $s=n/2$. Before stating precisely the doubling inequality in this case, we give some definitions.
For $V\in \mathscr{V}$, $u\in W^{2,p}(\Omega)\setminus\{0\}$ and $x_0\in \Omega^{4r}$, set
\begin{align*}
&\tilde{\lambda}=\tilde{\lambda}(V,u,x_0,r):= 
\\
&\hskip 1.5cm \frac{1}{\ln(9/4)}\ln \left(1+ 4r^{-1}\mathbf{k}\left[\vartheta_V(1+\kappa_V^1)\right]\frac{\|u(x_0+\cdot)\|_{L^2([\![r,13r/4]\!])}}{\|u(x_0+\cdot)\|_{L^2([\![r/4,r/2]\!])}}\right).
\end{align*}
Here and henceforth $\mathbf{k}>0$ is a generic universal constant, and we recall that for $0\le a<b$, 
\[
[\![a,b]\!]:=\{x\in \mathbb{R}^n;\; a<|x|<b\},
\]
\[
\kappa_V^1:=2(1+\sigma)\mathbf{I}_V^2+2\sigma\sqrt{\kappa_V}\, \mathbf{I}_V+\sigma/\sqrt{1+\sigma^2},
\]
where $\mathbf{I}_V:=(1-2\sigma^2\kappa_V)^{-1/2}$ and $\vartheta_V:=\vartheta(1-\vartheta\kappa_V)^{-1}$.

Let $\bar{\lambda}_{n/2}:=2+\max(\tilde{\lambda}+1,8+1/2+(n-2)/2)$ and 
\[
\mathbf{M}_{n/2}=\mathbf{M}_{n/2}(V,u,x_0,r):=\mathbf{k}\left[\vartheta_V(1+\kappa_V^1)\right]\bar{\lambda}_{n/2}^2. 
\]

\begin{theorem}\label{thmD1}
Fix $0<r_0<3/4$ so that $\Omega^{4r_0}$ is nonempty. Let $0<r<r_0$ and $x_0\in \Omega^{4r}$. For all  $V\in \mathscr{V}$ and $u\in W^{2,p}(\Omega)\setminus\{0\}$ satisfying $(-\Delta +V)u=0$ we have 
\begin{equation}\label{z8}
\|u\|_{L^2(B(x_0,2\rho))}\le \rho^{-\bar{\lambda}_{n/2}}\mathbf{M}_{n/2}\|u\|_{L^2(B(x_0,\rho))},\quad 0<\rho<r/8.
\end{equation}
\end{theorem}

\begin{proof}
Let $V\in \mathscr{V}$ and $u\in W^{2,p}(\Omega)$ satisfying $(-\Delta +V)u=0$, and define $w$ and $\mathcal{V}$ on $B:=B_{4r}$ by $w:=u(x_0+\cdot)$ and $\mathcal{V}:=V(x_0+\cdot)$. The following inequality were already established in the proof of Theorem \ref{thmsc2}
\begin{align}
&\mathbf{k} r^{-1}\||x|^{-\lambda}(-\Delta +\mathcal{V})(\chi w)\|_{L^p(B)}\label{db1}
\\
&\hskip 2cm\le  d_1^{-2}r_2^{-\lambda}\|w\|_{L^2([\![r_2,r_3]\!])}+d_1^{-1}r_2^{-\lambda}\|\nabla w\|_{L^2([\![r_2,r_3]\!])}\nonumber
\\
&\hskip 3cm +d_2^{-2}r_6^{-\lambda}\|w\|_{L^2([\![r_6,r_7]\!])}+d_2^{-1}r_6^{-\lambda}\|\nabla w\|_{L^2([\![r_6,r_7]\!])}.\nonumber
\end{align}

Choose $r_1=\rho(1-3\lambda^{-1})$, $r_2=\rho(1-2\lambda^{-1})$, $r_3=\rho(1-\lambda^{-1})$, $r_4=\rho$, $r_5=r$, $r_6=9r/4$, $r_7=11r/4$, $r_8=13r/4$. Here $d_1:=r_3-r_2=\rho\lambda^{-1}=r_2-r_1=r_4-r_3$ and $d_2:=r_7-r_6=r/2=r_8-r_7$.

Caccioppoli's inequality then yields
\begin{align}
&\mathbf{k} r^{-1}\||x|^{-\lambda}(-\Delta +\mathcal{V})(\chi w)\|_{L^p(B)}\label{db2}
\le  d_1^{-2}r_2^{-\lambda}(1+\kappa_V^1)\|w\|_{L^2([\![r_1,r_4]\!])}
\\
&\hskip 5.5cm +d_2^{-2}r_6^{-\lambda}(1+\kappa_V^1)\|w\|_{L^2([\![r_5,r_8]\!])}.\nonumber
\end{align}
This and
\[
\||x|^{-\lambda}\chi w\|_{L^{p'}(B)}\le \vartheta_V \||x|^{-\lambda}(-\Delta +\mathcal{V})(\chi w)\|_{L^p(B)}
\]
imply
\begin{align*}
&\mathbf{k} r^{-1}\left[\vartheta_V(1+\kappa_V^1)\right]^{-1}r_5^{-\lambda}\|w\|_{L^2([\![r_3,r_5]\!])}
\\
&\hskip 4cm\le d_1^{-2}r_2^{-\lambda}\|w\|_{L^2([\![r_1,r_4]\!])}
+d_2^{-2}r_6^{-\lambda}\|w\|_{L^2([\![r_5,r_8]\!])}.
\end{align*}
If $\lambda \ge 8$, then we verify
\[
d_1^{-2}r_2^{-\lambda}r_5^\lambda\le \rho^{-2}\lambda^2(4\rho^{-1} r/3)^\lambda ,\quad d_2^{-2}r_6^{-\lambda}r_5^\lambda=4r^{-2}(4/9)^\lambda .
\]
These, combined with the preceding inequality, give
\begin{align*}
&\mathbf{k} r^{-1}\left[\vartheta_V(1+\kappa_V^1)\right]^{-1}\|w\|_{L^2([\![\rho(1-\lambda^{-1}),r]\!])}
\\
&\hskip 2cm \le  \rho^{-2}\lambda^2(4\rho^{-1} r/3)^\lambda\|w\|_{L^2([\![\rho(1-3\lambda^{-1}),\rho]\!])}
\\
&\hskip 6cm +4r^{-2}(4/9)^\lambda\|w\|_{L^2([\![r,13r/4]\!])}.
\end{align*}
In the case $4\rho^{-1} r/3\ge 1$, we derive
\begin{align*}
&\mathbf{k} r^{-1}\left[\vartheta_V(1+\kappa_V^1)\right]^{-1} \left[\|w\|_{L^2(B_{2\rho})}+\|w\|_{L^2([\![2\rho,r]\!])}\right]
\\
&\hskip 4cm\le  \rho^{-2}\lambda^2(4\rho^{-1} r/3)^\lambda\|w\|_{L^2(B_\rho)}
\\
&\hskip 6cm+4r^{-2}(4/9)^\lambda\|w\|_{L^2([\![r,13r/4]\!])}.
\end{align*}
In particular, if $\rho<r/8$, then
\begin{align*}
&\mathbf{k} r^{-1}\left[\vartheta_V(1+\kappa_V^1)\right]^{-1} \left[\|w\|_{L^2(B_{2\rho})}+\|w\|_{L^2([\![r/4,r/2]\!])}\right]
\\
&\hskip 4cm\le  \rho^{-2}\lambda^2(4\rho^{-1} r/3)^\lambda\|w\|_{L^2(B_\rho)}
\\
&\hskip 6cm+4r^{-2}(4/9)^\lambda\|w\|_{L^2([\![r,13r/4]\!])}.
\end{align*}
Assume that we can choose $\lambda$ in such a way that
\begin{equation}\label{z1} 
4r^{-1}(4/9)^\lambda\|w\|_{L^2([\![r,13r/4]\!])}\le \mathbf{k}\left[\vartheta_V(1+\kappa_V^1)\right]^{-1} \|w\|_{L^2([\![r/4,r/2]\!])}.
\end{equation}
In this case, we obtain
\begin{equation}\label{z2}
\mathbf{k} r^{-1}\left[\vartheta_V(1+\kappa_V^1)\right]^{-1} \|w\|_{L^2(B_{2\rho})}\le \lambda^2\rho^{-(2+\lambda)}\|w\|_{L^2(B_\rho)},
\end{equation}
where we used that $4r/3<1$

Note that \eqref{z1} holds whenever $\lambda$ is chosen so that
\[
1+4r^{-1}\mathbf{k}\left[\vartheta_V(1+\kappa_V^1)\right]\frac{\|w\|_{L^2([\![r,13r/4]\!])}}{\|w\|_{L^2([\![r/4,r/2]\!])}}\le (9/4)^{\lambda}=e^{\lambda\ln(9/4)}.
\]
That is
\begin{equation}\label{z3}
\lambda \ge \frac{1}{\ln(9/4)}\ln \left(1+ 4r^{-1}\mathbf{k}\left[\vartheta_V(1+\kappa_V^1)\right]\frac{\|w\|_{L^2([\![r,13r/4]\!])}}{\|w\|_{L^2([\![r/4,r/2]\!])}}\right)=\tilde{\lambda}.
\end{equation}

If $\tilde{\lambda}\ge 9$ and $\tilde{\lambda}\in \Lambda$, we get by taking $\lambda=\tilde{\lambda}$ in \eqref{z2}
\begin{equation}\label{z4}
\mathbf{k}\left[\vartheta_V(1+\kappa_V^1)\right]^{-1} \|w\|_{L^2(B_{2\rho})}\le \tilde{\lambda}^2\rho^{-(2+\tilde{\lambda})}\|w\|_{L^2(B_\rho)}.
\end{equation}

When $\tilde{\lambda}\ge 9$ and $\tilde{\lambda}\not\in \Lambda$,  we have already seen that there exists $\hat{\lambda}\in \Lambda$, $\hat{\lambda} \ge 8$, so that $\tilde{\lambda}<\hat{\lambda}<\tilde{\lambda}+1$ or $\hat{\lambda}<\tilde{\lambda}<\hat{\lambda}+1$. Taking $\lambda=\hat{\lambda}$ in \eqref{z2} we get
\begin{equation}\label{z5}
\mathbf{k}\left[\vartheta_V(1+\kappa_V^1)\right]^{-1} \|w\|_{L^2(B_{2\rho})}\le \tilde{\lambda}^2 \rho^{-(3+\tilde{\lambda})}\|w\|_{L^2(B_\rho)}
\end{equation}
if $\tilde{\lambda}<\hat{\lambda}<\tilde{\lambda}+1$ and
\begin{equation}\label{z6}
\mathbf{k}\left[\vartheta_V(1+\kappa_V^1)\right]^{-1} \|w\|_{L^2(B_{2\rho})}\le \tilde{\lambda}^2 \rho^{-(2+\tilde{\lambda})}\|w\|_{L^2(B_\rho)}
\end{equation}
if $\hat{\lambda}<\tilde{\lambda}<\hat{\lambda}+1$.

Finally, we consider the case $\tilde{\lambda}\le 9$. In this case,  $\tilde{\lambda}\le \lambda _\ast:=8+1/2+(n-2)/2\in \Lambda$ then \eqref{z2} with $\lambda=\lambda _\ast$ implies
\begin{equation}\label{z7}
\mathbf{k}\left[\vartheta_V(1+\kappa_V^1)\right]^{-1} \|w\|_{L^2(B_{2\rho})}\le \lambda_\ast^2\rho^{-(2+\lambda_\ast)}\|w\|_{L^2(B_{\rho})}.
\end{equation}

In view of inequalities \eqref{z4} to \eqref{z7}, we obtain 
\[
\|u\|_{L^2(B(x_0,2\rho))}\le \rho^{-\bar{\lambda}_{n/2}}\mathbf{M}_{n/2}\|u\|_{L^2(B(x_0,\rho))}.
\]
The proof is then complete.
\end{proof}

The principle of the proof of Theorem \ref{thmD1} is borrowed from \cite{Li}.

Before stating the doubling inequality for potentials in $L^s(\Omega)$ with $s\in (n/2,\infty]$, we need to introduce new definition. Set 
\[
\mathscr{V}_s:=\{V\in L^s(\Omega);\; \vartheta \tilde{\kappa}_V<1\},\quad s\in (n/2,\infty],
\]
where $\tilde{\kappa}_V:=\|V\|_{L^{n/2}(\Omega)}$. Also, for $V\in \mathscr{V}_s$, let $\tilde{\vartheta}_V:=\vartheta(1-\vartheta \tilde{\kappa}_V)^{-1}$ and $\iota:=s/(2s-n)$ if $s\in (n/2,\infty)$ and $\iota=1$ if $s=\infty$.

Let $r_0>0$ so that $\Omega^{4r_0}$ is nonempty and $0<r<r_0$. In the case of a potential in $L^s(\Omega)$, $s\in (n/2,\infty]$ we have the following doubling inequality where, for $x_0\in \Omega^{4r}$, $V\in L^s(\Omega)$ and $u\in X_s\setminus\{0\}$, $\tilde{\lambda}_s=\tilde{\lambda}_s(V,u,x_0,r)$ and  $\mathbf{M}_s=\mathbf{M}_s(V,u,x_0,r)$ are given as follows
\begin{align*}
&\tilde{\lambda}_s:=\frac{1}{\ln(9/4)}\ln \left(1+ 4r^{-1}\mathbf{k}\left[\tilde{\vartheta}_V(1+\kappa_V^\iota)\right]\frac{\|u(x_0+\cdot)\|_{L^2([\![r,13r/4]\!])}}{\|u(x_0+\cdot)\|_{L^2([\![r/4,r/2]\!])}}\right),
\\
&\mathbf{M}_s:=\mathbf{k}\left[1+\kappa_V^\iota \right]\bar{\lambda}_s^2,
\end{align*}
and let $\bar{\lambda}_s:=2+\max(\tilde{\lambda}_s+1,8+1/2+(n-2)/2)$.

Recall that $X_s:=W^{2,p}(\Omega)$ if $s\in (n/2,\infty]$ and $X_\infty:=H^2(\Omega)$.

\begin{theorem}\label{thmdin}
Fix $0<r_0<3/4$ so that $\Omega^{4r_0}$ is nonempty  and let $0<r<r_0$, $x_0\in \Omega^{4r}$ and $0<\rho <r/8$. For all $V\in \mathscr{V}_s$ and $u\in X_s$ satisfying $(-\Delta +V)u=0$ and $\|u\|_{L^2(\Omega)}=1$ we have
\begin{equation}\label{z8.0}
\|u\|_{L^2(B(x_0,2\rho))}\le \rho^{-\bar{\lambda}_s}\mathbf{M}_s\|u\|_{L^2(B(x_0,\rho))},\quad 0<\rho <r/8.
\end{equation}
\end{theorem}

\begin{proof}
The proof is obtained by slightly modifying that of Theorem \ref{thmD1}. Let $s\in (n/2,\infty]$, $V\in L^s(\Omega)$, $x_0\in \Omega^{4r}$, $\lambda\in \Lambda$ and $u\in W^{2,p}(\Omega)$ with $\mathrm{supp}(u(x_0+\cdot))\subset \dot{B}$. Set $\mathcal{V}:=V(x_0+\cdot)$ and $w:=u(x_0+\cdot)$. 

From the preceding proof, we have
\[
\||x|^{-\lambda}\chi w\|_{L^{p'}(B)}\le \tilde{\vartheta}_V\||x|^{-\lambda}(-\Delta +\mathcal{V})(\chi w)\|_{L^p(B)}.
\]
In light of this inequality, we proceed similarly as in the proof of the case $V\in L^{n/2}(\Omega)$ by replacing $\vartheta_V(1+\kappa_V^1)$ by $\tilde{\vartheta}_V(1+\kappa_V^\iota)$. The expected doubling inequality then follows.
\end{proof}

\section{Quantitative vanishing order}\label{S8}

Here again $0<\mathfrak{r}=\mathfrak{r}(\Omega)\le 1$ is as in Section \ref{S1}.

We first consider the case of a potential in $L^s(\Omega)$ with $s\in (n/2,\infty]$. Let $\eta$ be as in \eqref{eta}, $\varsigma:=|\ln \alpha|/\ln 2$ and
\begin{equation}\label{tau}
\tau(r):=\alpha^{-1}r^{-\varsigma}\eta (r)^{-1}=\alpha^{-1}r^{-\varsigma}e^{\mathfrak{h}r^{-n}} \quad r>0.
\end{equation}

We will use the following growth lemma, where we recall that $X_s:=W^{2,p}(\Omega)$ if $s\in (n/2,\infty)$ and $X_\infty:=H^2(\Omega)$ and $\varphi_s(V)$ is given by \eqref{varphis}.

\begin{lemma}\label{lemmagr}
Let $s\in (n/2,\infty]$, $\Omega_0\Subset \Omega$ and $x_0\in \Omega$ be fixed. Set $\rho_0:=\mathrm{dist}(\Omega_0,\Gamma)/7$, $d_0:=\mathrm{dist}(x_0,\Gamma)/3$ and $r^\ast:=\min(\mathfrak{r}/4,\rho_0,d_0)$. There exists $0<\tilde{c}=\tilde{c}(n,\Omega_0,s)<1$ such that for all $V\in L^s(\Omega)$ and $u\in X_s$ satisfying $(-\Delta+V)u=0$ and $\|u\|_{L^2(\Omega)}=1$ we have for all $0<r<r^\ast$
\begin{equation}\label{qu17}
\mathfrak{M}_s(V,u,x_0,r):=e^{-\tau(r)\left|\ln \left(\tilde{c}\left[\varphi_s(V)\right]^{-1}\|u\|_{L^2(\Omega_0)}\right)^{\rho_0^\varsigma}\right|}\le \|u\|_{L^2(B(x_0,r))} .
\end{equation}
\end{lemma}

\begin{proof}
Let $y\in \Omega$ and $\rho>0$ so that $B(y,6\rho)\subset \Omega$. If $0<r<\rho$, let $k\ge 1$ be the integer so that $2^{k-1}r<\rho \le 2^kr$. As $3\times 2^kr=6\times 2^{k-1}r<6\rho$, we conclude that $B(y,3\times 2^kr)\subset \Omega$.

Let $V\in L^s(\Omega)$ with $s\in (n/2,\infty)$ and, for simplicity,  $\mathcal{C}\ge 1$ will denote a generic constant of the form $\mathcal{C}=\mathbf{c}\varphi_s(V)$. Let $u\in W^{2,p}(\Omega)$ satisfying $(-\Delta +V)u=0$ and $\|u\|_{L^2(\Omega)}=1$. Similarly to the proof of Theorem \ref{mthm2}, it follows from \eqref{qu3} 
\[
\|u\|_{L^2(B(y,\rho))}\le \mathcal{C}^{1+\alpha+\ldots+\alpha^{k-1}}\|u\|_{L^2(B(y,r))}^{\alpha^k}.
\]
Since $k-1<\ln (\rho/r)^{1/\ln 2}\le k$, we obtain 
\[
\alpha (r/\rho)^{|\ln \alpha|/\ln 2}\le \alpha^k=e^{-k|\ln \alpha|}\le (r/\rho)^{|\ln \alpha|/\ln 2}. 
\]
In consequence, we have
\begin{equation}\label{qu12}
\|u\|_{L^2(B(y,\rho))}\le \mathcal{C}\|u\|_{L^2(B(y,r))}^{\alpha (r/\rho)^\varsigma}.
\end{equation}

Applying \eqref{qu12} with $\rho=\rho_0:=\mathrm{dist}(\Omega_0,\partial \Omega)/7$, we obtain
\begin{equation}\label{qu13}
\|u\|_{L^2(B(y,\rho_0))}\le \mathcal{C}\|u\|_{L^2(B(y,r))}^{\alpha (r/\rho_0)^\varsigma},\quad y\in \Omega_0,\; 0<r<\rho_0.
\end{equation}
As $\bar{\Omega}_0$ is compact, there exists a finite sequence $(y_j)_{1\le j\le \ell}$ in $\Omega_0$ so that
\[
\Omega_0\subset \bigcup_{j=1}^\ell B(y_j,\rho_0).
\]
Then \eqref{qu13} yields
\begin{equation}\label{qu14}
\|u\|_{L^2(\Omega_0)}\le \mathcal{C}\sum_{j=1}^\ell \|u\|_{L^2(B(y_j,r))}^{\alpha (r/\rho_0)^\varsigma},\quad 0<r<\rho_0.
\end{equation}
Let $0<r< r^\ast$. Then \eqref{qu10} gives
\[
\|u\|_{L^2(\Omega_0)}\le \ell \mathcal{C}\left(\mathcal{C}\|u\|_{L^2(B(x_0,r))}^{\eta(r)}\right)^{\alpha (r/\rho_0)^\varsigma}
\]
and therefore
\[
\|u\|_{L^2(\Omega_0)}\le \ell \mathcal{C}\|u\|_{L^2(B(x_0,r))}^{\alpha \eta(r)(r/\rho_0)^\varsigma}.
\]
That is we have
\begin{equation}\label{qu16}
\left((\mathbf{c}\ell)^{-1}\left[\varphi_s(V)\right]^{-1}\|u\|_{L^2(\Omega_0)}\right)^{\rho_0^\varsigma/(\alpha r^\varsigma \eta(r))}\le \|u\|_{L^2(B(x_0,r))} .
\end{equation}
As $\alpha r^\varsigma \eta(r)= \tau(r)^{-1}$, the expected inequality then follows. The case $s=\infty$ is proved similarly.
\end{proof}

In the following, we reuse the notations of the previous section. Let $s\in (n/2,\infty]$, $V\in \mathscr{V}_s$ and $u\in X_s$ satisfying $(-\Delta +V)u=0$ and $\|u\|_{L^2(\Omega)}=1$.
Let $r^\ast$ be as in Lemma \ref{lemmagr} and fix $\bar{r}$ so that $\Omega^{4\bar{r}}$ is nonempty. Let 
\begin{equation}\label{rhat}
\hat{r}:=\min(\bar{r}/16,r^\ast),
\end{equation}
$0<\rho <\hat{r}$, $0<r<\rho$ and $k\ge 1$ be the integer so that $2^{k-1}r<\rho\le 2^kr$. Applying \eqref{z8.0} with $\rho$ replaced by $2^jr$, with $j$ varying from $k$ to $1$ and $r$ replaced by $\bar{r}$, we obtain
\begin{equation}\label{z9.1}
\|u\|_{L^2(B(x_0,\rho))}\le \|u\|_{L^2(B(x_0,2^kr))} \le [r^{-\bar{\lambda}_s}\mathbf{M}_s]^k\|u\|_{L^2(B(x_0,r))},
\end{equation}
where we used that $2^kr<\bar{r}/8$. Here $\bar{\lambda}_s:=2+\max(\tilde{\lambda}_s+1,8+1/2+(n-2)/2)$ with 
\begin{align*}
&\tilde{\lambda}_s:=\frac{1}{\ln(9/4)}\ln \left(1+ 4\bar{r}^{-1}\mathbf{k}\left[\tilde{\vartheta}_V(1+\kappa_V^\iota)\right]\frac{\|u(x_0+\cdot)\|_{L^2([\![\bar{r},13\bar{r}/4]\!])}}{\|u(x_0+\cdot)\|_{L^2([\![\bar{r}/4,\bar{r}/2]\!])}}\right),
\\
&\mathbf{M}_s:=\mathbf{k}\left[1+\kappa_V^\iota \right]\bar{\lambda}_s^2,
\end{align*}
 
Since $k<1+\ln (\rho/r)/\ln 2$, \eqref{z9.1} implies
\begin{align*}
\|u\|_{L^2(B(x_0,\rho))}
&\le [r^{-\bar{\lambda}_s}\mathbf{M}_s][r^{-\bar{\lambda}_s}\mathbf{M}_s]^{\ln (\rho/r)/\ln 2}\|u\|_{L^2(B(x_0,r))}
\\
&\le r^{-\hat{\lambda}_s(\rho)}e^{(\bar{\lambda}_s/\ln 2)(\ln r)^2}\mathbf{M}_s^{1+\ln \rho/\ln 2}\|u\|_{L^2(B(x_0,r))},
\end{align*}
where
\[
\hat{\lambda}_s(\rho):=\bar{\lambda}_s(1+\ln \rho /\ln 2)+\ln \mathbf{M}_s/\ln 2.
\]
Fix $\rho=\hat{r}/2$, let $\hat{\lambda}_s:=\hat{\lambda}_s(\hat{r}/2)$, $\dot{\lambda}_s:=\bar{\lambda}_s/\ln 2$, $\hat{\mathbf{M}}_s(V,u,x_0):=\mathbf{M}_s^{-(1+\ln (\hat{r}/2)/\ln 2)}$ and set for $0<r<\hat{r}/2$
\[
\mathfrak{N}_s(V,u,x_0):=\hat{\mathbf{M}}_s(V,u,x_0)\mathfrak{M}_s(V,u,x_0,\hat{r}/2),
\]
where $\mathfrak{M}_s$ is as in \eqref{qu17}. Putting together \eqref{qu17}, with $r$ replaced by $\hat{r}/2$, and the previous inequality, we obtain
\begin{equation}\label{vo1}
\mathfrak{N}_s(V,u,x_0)r^{\hat{\lambda}_s}e^{-\dot{\lambda}_s(\ln r)^2}\le \|u\|_{L^2(B(x_0,r))},\quad 0<r<\hat{r}/2.
\end{equation}

We now consider the case $s=n/2$. Let $V\in L^{n/2}(\Omega)$ and $u\in W^{2,p}(\Omega)$ satisfying $(-\Delta +V)u=0$ and $\|u\|_{L^2(\Omega)}=1$. Set $\mathcal{C}(r):= [q_Vr^{-1}]^{1/(1-\alpha)}$, where $0<\alpha <1$ is a generic universal constant. Under the assumptions  of Lemma \ref{lemmagr} and the notations of its proof, we have
\[
\|u\|_{L^2(B(y,\rho))}\le \|u\|_{L^2(B(y,2^kr))}\le \mathcal{C}(2^{k-1}r)\|u\|_{L^2(B(y,2^{k-1}r))}
\]
and since $r\mapsto \mathcal{C}(r)$ is decreasing, we deduce 
\[
\|u\|_{L^2(B(y,\rho))}\le \|u\|_{L^2(B(y,2^kr))}\le \mathcal{C}(r)\|u\|_{L^2(B(y,2^{k-1}r))}.
\]
By proceeding as in the proof of Lemma \ref{lemmagr}, we obtain
\begin{equation}\label{a1.4}
\left[\ell^{-1}[q_Vr^{-1}]^{-1/(1-\alpha)}\|u\|_{L^2(\Omega_0)}\right]^{\rho_0^\varsigma/(\alpha r^\varsigma \eta(r))}\le \|u\|_{L^2(B(x_0,r))} ,
\end{equation}
where $\varsigma:=|\ln \alpha|/\ln 2$ and $\ell$ is as in the proof of Lemma \ref{lemmagr}.

Let
\[
\hat{\upsilon}:=\rho_0^{\varsigma}/\alpha ,\quad \tilde{\upsilon}:=\hat{\upsilon}/(1-\alpha),\quad \tilde{\mathfrak{h}}:=\mathfrak{h}+\varsigma/n.
\]
Then \eqref{a1.4} implies
\[
\left[\ell^{-1}q_V^{-1/(1-\alpha)}\|u\|_{L^2(\Omega_0)}\right]^{\hat{\upsilon}e^{\tilde{\mathfrak{h}}r^{-n}}}e^{-\tilde{\upsilon}e^{\tilde{\mathfrak{h}}r^{-n}}|\ln r|}\le \|u\|_{L^2(B(x_0,r))} .
\]
That is we have
\begin{equation}\label{a1.5}
\mathfrak{M}_{n/2}(V,u,x_0,r)\le \|u\|_{L^2(B(x_0,r))} ,
\end{equation}
where we set
\[
\mathfrak{M}_{n/2}(V,u,x_0,r):=\left[\ell^{-1}q_V^{-1/(1-\alpha)}\|u\|_{L^2(\Omega_0)}\right]^{\hat{\upsilon}e^{\tilde{\mathfrak{h}}r^{-n}}}e^{-\tilde{\upsilon}e^{\tilde{\mathfrak{h}}r^{-n}}|\ln r|}.
\]

Next, let $r^\ast$ be as in Lemma \ref{lemmagr} and fix $\bar{r}$ so that $\Omega^{4\bar{r}}$ is nonempty. Let $\hat{r}$ be as in \eqref{rhat}, $0<\rho <\hat{r}$, $0<r<\rho$ and $k\ge 1$ be the integer so that $2^{k-1}r<\rho\le 2^kr$. Let
\[
\mathbf{M}_{n/2}:=\mathbf{k}\left[\vartheta_V(1+\kappa_V^1)\right]\bar{\lambda}_{n/2}^2,
\]
where $\bar{\lambda}_{n/2}:=2+\max(\tilde{\lambda}+1,8+1/2+(n-2)/2$ with
\[
\tilde{\lambda} := \frac{1}{\ln(9/4)}\ln \left(1+ 4\bar{r}^{-1}\mathbf{k}\left[\vartheta_V(1+\kappa_V^1)\right]\frac{\|u(x_0+\cdot)\|_{L^2([\![\bar{r},13\bar{r}/4]\!])}}{\|u(x_0+\cdot)\|_{L^2([\![\bar{r}/4,\bar{r}/2]\!])}}\right).
\]
Applying \eqref{z8} with $\rho$ replaced by $2^jr$, with $j$ varying from $k$ to $1$, we obtain
\begin{equation}\label{z9}
\|u\|_{L^2(B(x_0,\rho))}\le \|u\|_{L^2(B(x_0,2^kr))} \|\le [r^{-\bar{\lambda}}\mathbf{M}]^k\|u\|_{L^2(B(x_0,r))}.
\end{equation}
Again as $k<1+\ln (\rho/r)/\ln 2$, we proceed similarly to the preceding case to obtain
\[
\|u\|_{L^2(B(x_0,\rho))}\le r^{-\hat{\lambda}_{n/2}(\rho)}e^{(\bar{\lambda}_{n/2}/\ln 2)(\ln r)^2}\mathbf{M}_{n/2}^{1+\ln \rho/\ln 2},
\]
where
\[
\hat{\lambda}_{n/2}(\rho):=\bar{\lambda}_{n/2}(1+\ln \rho /\ln 2)+\ln \mathbf{M}_{n/2}/\ln 2.
\]

Fix $\rho=\hat{r}/2$, let $\hat{\lambda}_{n/2}:=\hat{\lambda}_{n/2}(\hat{r}/2)$, $\dot{\lambda}_{n/2}:=\bar{\lambda}_{n/2}/\ln 2$, \[\hat{\mathbf{M}}_{n/2}(V,u,x_0):=\mathbf{M}_{n/2}^{-(1+\ln (\hat{r}/2)/\ln 2)}\] and set for $0<r<\hat{r}/2$
\[
\mathfrak{N}_{n/2}(V,u,x_0):=\mathfrak{M}_{n/2}(V,u,x_0,\hat{r}/2)\hat{\mathbf{M}}_{n/2}(V,u,x_0),
\]
where $\mathfrak{M}_{n/2}$ is as in \eqref{a1.5}. Putting together \eqref{a1.5}, with $r$ replaced by $\hat{r}/2$, and the previous inequality, we obtain
\begin{equation}\label{vo2}
\mathfrak{N}_{n/2}(V,u,x_0)r^{\hat{\lambda}_{n/2}}e^{-\dot{\lambda}_{n/2}(\ln r)^2}\le \|u\|_{L^2(B(x_0,r))},\quad 0<r<\hat{r}/2.
\end{equation}

In light of \eqref{vo1} and \eqref{vo2}, we can state the following theorem, where we have extended the definition of $X_s$ and $\mathscr{V}_s$ to $s=n/2$ by defining $X_{n/2}:=W^{2,p}(\Omega)$ and $\mathscr{V}_{n/2}:=\mathscr{V}$.

\begin{theorem}\label{thmvo}
Let $s\in [n/2,\infty]$, $V\in \mathscr{V}_s$ and $u\in X_s$ satisfying $(-\Delta +V)u=0$ and $\|u\|_{L^2(\Omega)}=1$. Then
\[
\mathfrak{N}_s(V,u,x_0)r^{\hat{\lambda}_s}e^{-\dot{\lambda}_s(\ln r)^2}\le \|u\|_{L^2(B(x_0,r))},\quad 0<r<\hat{r}/2.
\]
\end{theorem}

A quantitative vanishing order was obtained in \cite[Theorem 1]{DZ1} for the equation $(-\Delta +W\cdot \nabla +V)u=0$ in a ball of $\mathbb{R}^n$ in the case $n\ge 3$, $V\in L^s(\Omega)$ with $s\in (3n/2-1,\infty]$ and $W\in L^m(\Omega,\mathbb{C}^n)$ with $m\in ((3n^2-2n)/(5n-2),\infty]$. The result of \cite[Theorem 1]{DZ1} is expressed in terms of the norm $L^\infty$. This reference also contains other variants of the quantitative vanishing order given in \cite[Theorem 1]{DZ1}. \cite[Theorem 1]{DZ1} was extended in \cite[Theorem 1]{Da} to the case $W=0$ and $V\in L^s(\Omega)$ with $s\in (n/2,\infty)$. A similar result to the case $n\ge 3$ (\cite[Theorem 1]{DZ1}) was proved in \cite[Theorem 1]{DZ2} in dimension two when $V\in L^s(\Omega)$ with $s\in (1,\infty]$ and $W\in L^m(\Omega,\mathbb{C}^n)$ with $m\in (2,\infty]$. As in our case, the constants in the above-mentioned results depend explicitly on the norms of $V$ and $W$ in $L^s(\Omega)$ and $L^m(\Omega)$, respectively, and their proofs rely on Carleman inequalities.

\section{Quantitative vanishing order  for continuous potentials}\label{S9}

For simplicity, all the functions we consider in this section are real-valued. Let $\rho>0$ be arbitrarily fixed and $B:=B_\rho=B(0,\rho)$. In the following, $0<r\le \rho$ and we use the notations $B_r:=B(0,r)$ and $S_r:=\partial B_r$.

Let $\mathcal{V}\in C(\bar{B})\setminus \{0\}$  and $v\in H^2(B)\setminus\{0\}$ satisfying $(-\Delta +\mathcal{V})v=0$. Define
\begin{align*}
&H(r):=\int_{S_r}v^2ds ,
\\
&D(r) :=\int_{B_r}\left\{|\nabla v|^2+\mathcal{V}v^2\right\}dx.
\end{align*}
From the calculations in \cite[Exercise 5.3, Chapter 5]{Ch24_3}, we have
\begin{align}
&H'(r)=(n-1)r^{-1}H(r)+2D(r), \label{Aa1}
\\
&D'(r)=(n-2)r^{-1}D(r)+\hat{D}(r)+2\bar{H}(r)+\hat{H}(r), \label{Aa2}
\end{align}
where
\begin{align*}
&\bar{H}(r):=\int_{S_r}(\partial _\nu v)^2ds,
\\
& \hat{H}(r):=\int_{S_r} \mathcal{V} v^2ds ,
\\
&\hat{D}(r):=-\int_{B_r} \mathcal{V}\left\{2 v(x\cdot \nabla v) +(n-2)r^{-1}v^2\right\} dx.
\end{align*}
 
Let us fix $\kappa\ge \|\mathcal{V}\|_{L^\infty(B)}$ and we set
\[
r_\kappa=r_\kappa(n,\kappa,\rho):=\min \left(\rho ,\sqrt{(n-1)\kappa^{-1}}\right). 
\]
Then, again from the calculations in \cite[Exercise 5.3, Chapter 5]{Ch24_3}, we have 
\begin{equation}\label{K}
K(r):=\int_{B_r}v^2dx \le rH(r),\quad 0<r\le r_\kappa.
\end{equation}
Define the frequency function $N$ as follows
\[
N(r)=N(v,\mathcal{V},r):=\frac{rD(r)}{H(r)}.
\]
We verify
\begin{align*}
\frac{N'(r)}{N(r)}&=\frac{1}{r}+\frac{D'(r)}{D(r)}-\frac{H'(r)}{H(r)}
\\
&=\frac{\hat{D}(r)}{D(r)}+\frac{\hat{H}(r)}{D(r)}+2\frac{\bar{H}(r)H(r)-D(r)^2}{D(r)H(r)} .
\end{align*}
Using the known inequality $\bar{H}(r)H(r)\ge D(r)^2$, we obtain
\begin{equation}\label{N}
\frac{N'(r)}{N(r)}\ge \frac{\hat{D}(r)}{D(r)}+\frac{\hat{H}(r)}{D(r)}.
\end{equation}

Set
\[
\mathcal{I}:=\{ r\in (0,r_\kappa);\; N(r)>\max (N(r_\kappa),1)\}.
\]
According to the uniqueness of continuation, $H(r)\ne 0$. Therefore, $N$ is continuous and so $\mathcal{I}$ is an open set. Thus, there exists $J\subset \mathbb{N}$ such that  
\[
\mathcal{I}=\bigcup_{j\in J}(r_j,s_j).
\]
In each $(r_j,s_j)$, since $N(r)>1$, we have $H(r)<rD(r)$. Then
\begin{equation}\label{H}
\left|\frac{\hat{H}(r)}{D(r)}\right|\le \kappa \frac{H(r)}{D(r)} \le \kappa.
\end{equation}
On the other hand, we have
\begin{align*}
\left| \hat{D}(r)\right| &\le \kappa\left( 2r\int_{B_r}|v||\nabla v|dx+(n-2)r^{-1}K(r)\right)
\\
& \le \kappa\left( r\int_{B_r}v^2dx+r\int_{B_r}|\nabla v|^2dx+(n-2)r^{-1}K(r)\right)
\\
&\le \kappa\left(rD(r) - r\int_{B_r}\mathcal{V} v^2dx+[r+(n-2)r^{-1}]K(r)\right)
\\
&\le \kappa\left(rD(r) +\left[(\kappa+1)r+ (n-2)r^{-1}\right] K(r)\right).
\end{align*}
This and \eqref{K} implies
\[
\left| \hat{D}(r)\right| \le \kappa\left(rD(r) +\left(r^2(\kappa+1)+ n-2\right) H(r)\right),\quad 0<r\le r_\kappa,
\]
which, combined with $H(r)<rD(r)$, gives
\begin{equation}\label{D}
\left|\frac{\hat{D}(r)}{D(r)}\right| \le \rho\kappa\left(\rho^2(\rho+\kappa)+ n-1\right).
\end{equation}
Let
\[
\bar{\kappa}=\bar{\kappa}(n,\kappa,\rho):=\kappa+\rho\kappa\left(\rho^2(1+\kappa)+ n-1\right).
\]
Putting together \eqref{N}, \eqref{H} and \eqref{D}, we get
\[
\frac{N'(r)}{N(r)}\ge -\bar{\kappa}.
\]
In consequence, we have
\[
N(r)\le e^{\bar{\kappa}s_j}N(s_j)\le e^{\bar{\kappa}r_\kappa}\max (N(r_\kappa),1), \quad r\in (r_j,s_j).
\]
That is we have, where $\tilde{\kappa}:=\bar{\kappa}r_\kappa$,
\[
N(r)\le e^{\tilde{\kappa}}\max (N(r _\kappa),1), \quad r\in \mathcal{I},
\]
Noting  that $N\le  \max (N(r_\kappa),1)$ in $(0,r_\kappa)\setminus \mathcal{I}$, we deduce
\[
N(r)\le e^{\tilde{\kappa}}\max (N(r _\kappa),1), \quad 0<r<r_\kappa.
\]

In light of \eqref{Aa1} and the preceding inequality, we obtain
\[
\left(\ln\left(\frac{H(r)}{r^{n-1}} \right)\right)'=2\frac{N(r)}{r}\le \frac{M}{r},\quad 0<r<r_\kappa,
\]
where we set $M=M(\mathcal{V},v):=e^{\tilde{\kappa}}\max (N(r _\kappa),1)$. Then, for $0<r_1<r_2<r_\kappa$, we obtain
\[
\int_{r_1}^{r_2}\left(\ln\left(\frac{H(r)}{r^{n-1}} \right)\right)' dr=\ln \left(\frac{H(r_2)r_1^{n-1}}{H(r_1)r_2^{n-1}}\right)\le  \ln \left(\frac{r_2}{r_1}\right)^M.
\]
Whence
\[
H(r_2)\le (r_2/r_1)^{M+n-1}H(r_1),
\]
and then
\[
K(r_2)=r_2\int_0^1H(sr_2)dr\le r_2(r_2/r_1)^{M+n-1}\int_0^1H(sr_1)dr=(r_2/r_1)^{\bar{M}}K(r_1),
\]
where $\bar{M}:=M+n$. We obtain in particular the following doubling inequality
\begin{equation}\label{DD}
\|u\|_{L^2(B_{2r})}\le 2^{\bar{M}}\|u\|_{L^2(B_r)},\quad 0<r<r_\kappa/2.
\end{equation}

Let $x_0\in \Omega$, $V\in C(\overline{\Omega})\setminus \{0\}$ and $u\in H^2(\Omega)$ satisfying $(-\Delta +V)u=0$ and $\|u\|_{L^2(\Omega)}=1$. Applying \eqref{DD} with $\rho=\mathrm{dist}(x_0,\partial \Omega)/2$, and $\mathcal{V}=V(x_0+\cdot)$, $v=u(x_0+\cdot)$ and $\kappa=\kappa_V:=\|V\|_{L^\infty(\Omega)}$, we obtain
\begin{equation}\label{DD1}
\|u\|_{L^2(B(x_0,2r))}\le 2^{\bar{M}}\|u\|_{L^2(B(x_0,r))},\quad 0<r<r_{\kappa_V}/2.
\end{equation}
Note that in the present case
\[
N(r_{\kappa_V})=\frac{\displaystyle r_{\kappa_V}\int_{B(x_0,r_{\kappa_V})}(|\nabla u|^2+ Vu^2)dx}{\|u\|_{L^2(\partial B(x_0,r_{\kappa_V}))}^2}.
\]

On the other hand, it follows from \eqref{qu17}
\begin{equation}\label{DD2}
\mathfrak{M}_\infty(V,u,x_0,r)\le \|u\|_{L^2(B(x_0,r))} ,\quad 0<r<r^\ast,
\end{equation}
where $r^\ast$ and $\mathfrak{M}_\infty(V,u,x_0,r)$ are  as in Lemma \ref{lemmagr}.

Fix
\begin{equation}\label{rbar}
0<\bar{r}<\min (r_{\kappa_V}/2,r^\ast)/2.
\end{equation}
For $0<r<\bar{r}$, let $k\ge 1$ be the integer satisfying $2^{k-1}r<\bar{r} <2^kr$. As we have done before, by iterating \eqref{DD1}, we obtain
\[
\|u\|_{L^2(B(x_0,\bar{r}))}\le 2^{k\bar{M}}\|u\|_{L^2(B(x_0,r))}
\]
which, combined with \eqref{DD2}, yields
\begin{equation}\label{DD3}
\mathfrak{M}_\infty(V,u,x_0,\bar{r})\le 2^{k\bar{M}}\|u\|_{L^2(B(x_0,r))}.
\end{equation}
But $k<1+\ln (\bar{r}/r)/\ln 2$. In this case
\[
2^{k\bar{M}}\le 2^{(1+\ln (\bar{r}/r)/\ln 2)\bar{M}}=2^{\bar{M}}(\bar{r}/r)^{\bar{M}}.
\]
This and \eqref{DD3} yield the following result.

\begin{theorem}\label{thmivo}
Let $x_0\in \Omega$, $V\in C(\overline{\Omega})\setminus \{0\}$ and $u\in H^2(\Omega)$ satisfying $(-\Delta +V)u=0$ and $\|u\|_{L^2(\Omega)}=1$. Then we have
\[
\left[(2\bar{r})^{-\bar{M}}\mathfrak{M}_\infty(V,u,x_0,\bar{r})\right]r^{\bar{M}}\le \|u\|_{L^2(B(x_0,r))},\quad 0<r<\bar{r},
\]
where $\bar{r}$ is as in \eqref{rbar}.
\end{theorem}

The method we used in this section is taken from \cite{GL}.

\section{Extensions}\label{S10}

We discuss in this section a partial extension of the previous results to an operator of the form $-\Delta +W\cdot \nabla +V$ with $V\in L^s(\Omega)$ and $W\in L^m(\Omega,\mathbb{C}^n)$ with $s\in (n/2,\infty]$ and $m\in (n,\infty]$. By establishing a three-ball inequality for this operator using a Carleman inequality directly, we see that an additional condition on $m$ and $s$ is necessary. This is what we show in this section. At the end of this section, we explain how we can lift this condition using an interpolation inequality. 

The notations and assumptions are those appearing after Theorem \ref{mthm1}. In the present section we will use the following consequence of \cite[Theorem 1.1]{DET}.

\begin{proposition}\label{proW1}
Set $\zeta:=\left(D, \omega_0,\omega, \varrho,\beta, \|\varphi\|_{C^3(\bar{D})}\right)$. There exist $\bar{c}=\bar{c}(\zeta)>0$ and $\tau_0=\tau_0(\zeta)\ge 1$ such that for all $\tau \ge \tau_0$, $f\in L^2(D)$, $g\in L^p(D)$ and $u\in W^{2,p}_{\omega}(D)$ satisfying $-\Delta u=f+g$ we have
\begin{equation}\label{W1}
\tau^{3/2}\|u\|_{L^2_\tau (D)}+\tau^{1/2}\|\nabla u\|_{L^2_\tau (D)}\le \bar{c}\left(\|f\|_{L_\tau ^2(D)}+\tau^{3/4-1/(2n)}\|g\|_{L_\tau ^p(D)}\right).
\end{equation}
Here $L^r_\tau (\Omega):=L^r (\Omega,e^{r\tau \varphi}dx)$, $r=p,2$.
\end{proposition}

In the following, $\zeta$ is as in Proposition \ref{proW1} and $\bar{c}=\bar{c}(\zeta)$ will denote a generic constant.

Let $V=V^1+V^2$ with $V^1\in L^{n/2}(D)$ and $V^2\in L^\infty (D)$, and $W=W^1+W^2$ with $W^1\in L^n(D,\mathbb{C}^n)$ and $W^2\in L^\infty(D,\mathbb{C}^n)$. Let $h\in L^2(D)$ and $u\in W^{2,p}_{\omega}(D)$ satisfying 
\begin{equation}\label{W0}
(-\Delta +W\cdot\nabla +V)u=h.
\end{equation}
Applying \eqref{W1} to $u$ with $f=h-W^2\cdot \nabla u-V^2u$ and $g=-W^1\cdot \nabla u-V^1u$, we obtain
\begin{align*}
&\tau^{3/2}\|u\|_{L^2_\tau (D)}+\tau^{1/2}\|\nabla u\|_{L^2_\tau (D)}\le \bar{c}\|h\|_{L_\tau^2(D)}
\\
&\hskip 2cm +\bar{c}\left(\|W^2\|_{L^\infty(D)}\|\nabla u\|_{L_\tau ^2(D)}+\|V^2\|_{L^\infty(D)}\| u\|_{L^2_\tau (D)}\right)
\\
&\hskip 2.7cm+\bar{c}\tau^{3/4-1/(2n)}\left(\|W^1\|_{L^n(D)}\|\nabla u\|_{L_\tau ^2(D)}+\|V^1\|_{L^{n/2}(D)}\|u\|_{L_\tau^{p'}(D)}\right).
\end{align*}
This and
\begin{align*}
\|u\|_{L_\tau^{p'}(D)}=\|e^{\tau \varphi} u\|_{L^{p'}(D)}&\le \sigma\|\nabla (e^{\tau \varphi}u)\|_{L^2(D)}
\\
&\le \sigma\|\tau e^{\tau \varphi}u\nabla \varphi +e^{\tau \varphi}\nabla u\|_{L^2(D)}
\\
&\le \bar{c}(\tau \|u\|_{L^2_\tau (D)}+\|\nabla u\|_{L^2_\tau (D)})
\end{align*}
imply
\begin{align}
&\tau^{3/2}\|u\|_{L^2_\tau (D)}+\tau^{1/2}\|\nabla u\|_{L^2_\tau (D)}\le \bar{c}\|h\|_{L_\tau^2(D)}\label{W2}
\\
&\hskip .2cm+ \bar{c}\left(\|W^2\|_{L^\infty(D)}+\tau^{3/4-1/(2n)}\left[\|W^1\|_{L^n(D)}+\|V^1\|_{L^{n/2}(D)}\right]\right)\|\nabla u\|_{L_\tau ^2(D)}\nonumber
\\
&\hskip 3cm+\bar{c}\left(\tau^{7/4-1/(2n)}\|V^1\|_{L^{n/2}(D)}+\|V^2\|_{L^\infty(D)}\right)\|u\|_{L_\tau ^2(D)}.\nonumber
\end{align}

Next, let $W\in L^m(D,\mathbb{C}^n)$ with $m>n$ and $V\in L^s(D)$ with $s>n/2$. Set
\[
\varkappa_W:=\|W\|_{L^m(D)},\quad \varkappa_V:=\|V\|_{L^s(D)}.
\]
Let $t>0$. We already know that $V$ is decomposed in the form $V=V_t^1+V_t^2$, where $V_t^1\in L^{n/2}(D)$, $V_t^2\in L^\infty (D)$ and satisfy
\[
\|V_t^1\|_{L^{n/2}(D)}\le t^{-(2s/n-1)}\varkappa_V^{2s/n},\quad \|V_t^2\|_{L^\infty(D)}\le t.
\]
We can proceed similarly to decompose $W$ in the form $W=W_t^1+W_t^2$ so that $W_t^1\in L^n(D,\mathbb{C}^n)$, $W_t^2\in L^\infty(D,\mathbb{C}^n)$ and 
\[
\|W_t^1\|_{L^n(D)}\le t^{-(m/n-1)}\varkappa_W^{m/n},\quad \|W_t^2\|_{L^\infty(D)}\le t.
\]

Let $u\in W^{2,p}_{\omega}(D)$ satisfying \eqref{W0}. Then \eqref{W2} with $W^j=W_t ^j$ and $V^j=V_t^j$ for $j=1,2$ yields
\begin{align}
&\tau^{3/2}\|u\|_{L^2_\tau (D)}+\tau^{1/2}\|\nabla u\|_{L^2_\tau (D)}\le \bar{c}\|h\|_{L_\tau^2(D)}\label{W3}
\\
&\hskip .5cm+\bar{c}\left(t +\tau^{3/4-1/(2n)}\left[t^{-(m/n-1)}\varkappa_W^{m/n}+t^{-(2s/n-1)}\varkappa_V^{2s/n}\right]\right)\|\nabla u\|_{L_\tau ^2(D)}\nonumber
\\
&\hskip 4cm+\bar{c}\left(\tau^{7/4-1/(2n)}t^{-(2s/n-1)}\varkappa_V^{2s/n}+t\right)\|u\|_{L_\tau ^2(D)}.\nonumber
\end{align}

Let $\ell:= \min (m,2s)/n-1$. Assume in the following that $\tau \ge \tau_0\ge 1$ and $t=\tau^\delta$, where $\delta >0$ is to be determined. In this case, \eqref{W3} implies
\begin{align*}
&\tau^{3/2}\|u\|_{L^2_\tau (D)}+\tau^{1/2}\|\nabla u\|_{L^2_\tau (D)}\le \bar{c}\|h\|_{L_\tau^2(D)}
\\
&\hskip 1cm+\bar{c}\left(\tau^\delta+\tau^{3/4-1/(2n)-\delta\ell}\left[\varkappa_W^{m/n}+\varkappa_V^{2s/n}\right]\right)\|\nabla u\|_{L_\tau ^2(D)}
\\
&\hskip 4.2cm+\bar{c}\left(\tau^{7/4-1/(2n)-\delta\ell}\varkappa_V^{2s/n}+\tau^\delta\right)\|u\|_{L_\tau ^2(D)}.
\end{align*}

In the following, we make the assumption $\min (m,2s)>3n/2-1$. By choosing 
\[
0<\delta=\delta(n,m,s):=(3n-2)/[4n(\ell+1)]  <1/2,
\]
we get $\delta = 3/4-1/(2n)-\delta\ell$ and therefore
\begin{align}
&\tau^{3/2}\|u\|_{L^2_\tau (D)}+\tau^{1/2}\|\nabla u\|_{L^2_\tau (D)}\le \bar{c}\|h\|_{L_\tau^2(D)}\label{W4}
\\
&\hskip 2cm+\bar{c}\tau^\delta\left(1+\varkappa_W^{m/n}+\varkappa_V^{2s/n}\right)\|\nabla u\|_{L_\tau ^2(D)}\nonumber
\\
&\hskip 4.2cm+\bar{c}\tau^{1+\delta} \left(1+\varkappa_V^{2s/n}\right)\|u\|_{L_\tau ^2(D)}.\nonumber
\end{align}

With the temporary notation 
\[
\mathfrak{z} :=1+\varkappa_W^{m/n}+\varkappa_V^{2s/n},
\]
\eqref{W4} yields
\begin{equation}\label{W5}
\tau(\tau^{1/2}-\bar{c}\mathfrak{z} \tau^\delta )\|u\|_{L^2_\tau (D)}+(\tau^{1/2}-\bar{c}\mathfrak{z} \tau^\delta)\|\nabla u\|_{L^2_\tau (D)}\le \bar{c}\|h\|_{L_\tau^2(D)}.
\end{equation}
Let $\beta=\beta(n,m,s):= 1/(1/2-\delta)\, (>2)$. Under the condition
\begin{equation}\label{W6}
\tau \ge (2\bar{c}\mathfrak{z})^\beta+\tau_0,
\end{equation}
we proceed as in the case $W=0$ to obtain the following inequality
\begin{equation}\label{W7}
\|u\|_{L^2_\tau (D)}\le \bar{c}e^{\bar{c}(\varkappa_W^{\ell_m}+\varkappa_V^{k_s})}\|(-\Delta +W\cdot \nabla +V)u\|_{L^2_\tau (D)},\quad \tau >0,
\end{equation}
where $\ell_m:=m\beta/n$ and $k_s:=2s\beta/n$.

Next, if $W\in L^\infty(D,\mathbb{C}^n)$, then \eqref{W2} with $W^1=0$ and $W^2=W$, $V^1=V_t^1$ and $V^2=V_t^2$ yields
\begin{align}
&\tau^{3/2}\|u\|_{L^2_\tau (D)}+\tau^{1/2}\|\nabla u\|_{L^2_\tau (D)}\le \bar{c}\|h\|_{L_\tau^2(D)}\label{W8}
\\
&\hskip .5cm+ \bar{c}\left(\|W\|_{L^\infty(D)}+\tau^{3/4-1/(2n)}\|V_t^1\|_{L^{n/2}(D)}\right)\|\nabla u\|_{L_\tau ^2(D)}\nonumber
\\
&\hskip 3cm+\bar{c}\left(\tau^{7/4-1/(2n)}\|V_t^1\|_{L^{n/2}(D)}+\|V_t^2\|_{L^\infty(D)}\right)\|u\|_{L_\tau ^2(D)}.\nonumber
\end{align}
Set $\varkappa_W:=\|W\|_{L^\infty(D)}$. Then, under this notation and the condition
\[
\tau>\tau_0 +(2\bar{c}\varkappa_W)^2,
\]
\eqref{W8} implies
\begin{align*}
&\tau^{3/2}\|u\|_{L^2_\tau (D)}+\tau^{1/2}\|\nabla u\|_{L^2_\tau (D)}
\\
&\hskip 2cm\le \bar{c}\|h\|_{L_\tau^2(D)}
 +\bar{c}\tau^{3/4-1/(2n)}\|V_t^1\|_{L^{n/2}(D)}\|\nabla u\|_{L_\tau ^2(D)}
\\
&\hskip 3cm+\bar{c}\left(\tau^{7/4-1/(2n)}\|V_t^1\|_{L^{n/2}(D)}+\|V_t^2\|_{L^\infty(D)}\right)\|u\|_{L_\tau ^2(D)}.
\end{align*}
In consequence, we have
\begin{align*}
&\tau^{3/2}\|u\|_{L^2_\tau (D)}+\tau^{1/2}\|\nabla u\|_{L^2_\tau (D)}
\\
&\hskip 2cm\le \bar{c}\|h\|_{L_\tau^2(D)}
 +\bar{c}\tau^{3/4-1/(2n)}t^{-(2s/n-1)}\varkappa_V^{2s/n}\|\nabla u\|_{L_\tau ^2(D)}
\\
&\hskip 3cm+\bar{c}\left(\tau^{7/4-1/(2n)}t^{-(2s/n-1)}\varkappa_V^{2s/n}+t\right)\|u\|_{L_\tau ^2(D)}.
\end{align*}
As before, we derive from this inequality the following one
\begin{equation}\label{W9}
\|u\|_{L^2_\tau (D)}\le \bar{c}e^{\bar{c}(\varkappa_W^2+\varkappa_V^{k_s})}\|(-\Delta +W\cdot \nabla +V)u\|_{L^2_\tau (D)},\quad \tau >0.
\end{equation}
Interchanging the roles of $W$ and $V$, we get similarly as the preceding cases, where $\varkappa_V:=\|V\|_{L^\infty(D)}$,
\begin{equation}\label{W10}
\|u\|_{L^2_\tau (D)}\le \bar{c}e^{\bar{c}(\varkappa_W^{\ell_m}+\varkappa_V^{2/3})}\|(-\Delta +W\cdot \nabla +V)u\|_{L^2_\tau (D)},\quad \tau >0.
\end{equation}

Finally, if $W\in L^\infty(D,\mathbb{C}^n)$ and $V\in L^\infty (D)$ then \eqref{W2} becomes
\begin{align*}
&\tau^{3/2}\|u\|_{L^2_\tau (D)}+\tau^{1/2}\|\nabla u\|_{L^2_\tau (D)}\le \bar{c}\|h\|_{L_\tau^2(D)}
\\
&\hskip 2cm+ \bar{c}\left(\|W\|_{L^\infty(D)}\|\nabla u\|_{L_\tau ^2(D)}+\|V\|_{L^\infty(D)}\|u\|_{L_\tau ^2(D)}\right),
\end{align*}
from which we obtain
\begin{equation}\label{W11}
\|u\|_{L^2_\tau (D)}\le \bar{c}e^{\bar{c}(\varkappa_W^2+\varkappa_V^{2/3})}\|(-\Delta +W\cdot \nabla +V)u\|_{L^2_\tau (D)},\quad \tau >0.
\end{equation}

Let $\beta=\beta (n,m,s)>1$ defined as above and let $\ell_m:=m\beta/n$ if $m\in (n,\infty)$, $\ell_m:=2$ if $m=\infty$, $k_s:=2s\beta/n$ if $s\in (n/2,\infty)$ and $k_s:=2/3$ if $s=\infty$. 

Combining \eqref{W7},  \eqref{W9}, \eqref{W10} and \eqref{W11}, we obtain the following Carleman inequality, where we set
\[
\mathscr{N}:=\{(m,s)\in (n,\infty]\times (n/2,\infty];\; \min (m,2s)>3n/2-1\}.
\]

\begin{theorem}\label{thmW1}
Let $(m,s)\in \mathscr{N}$, $\omega\Subset D$. For all $W\in L^m(D,\mathbb{C}^n)$, $V\in L^s(D)$, $\tau >0$ and  $u\in W^{2,p}_\omega(D)$ we have
\begin{equation}\label{W12}
\|u\|_{L^2_\tau (D)}\le \bar{c}e^{\bar{c}(\varkappa_W^{\ell_m}+\varkappa_V^{k_s})}\|(-\Delta +W\cdot \nabla +V)u\|_{L^2_\tau (D)}.
\end{equation}
\end{theorem} 

A Caccioppoli-type inequality will be needed to establish three-ball inequality for the actual case. To this end we set
\[
\aleph(V,W):=
\left\{
\begin{array}{lll}
1+\kappa_W^{m/(m-n)}+\kappa_V^{s/(2s-n)},\quad &(m,s)\in (m,\infty)\times (n/2,\infty),
\\
1+\kappa_W+\kappa_V^{s/(2s-n)}, &(m,s)\in \{\infty\}\times (n/2,\infty),
\\
1+\kappa_W^{m/(m-n)}+\kappa_V^{1/2}, &(m,s)\in (n/2,\infty)\times \{\infty\},
\\
1+\kappa_W+\kappa_V^{1/2}, &m=s=\infty,
\end{array}
\right.
\]
where
\[
\kappa_V:=\|V\|_{L^s(\Omega)},\quad \kappa_W:=\|W\|_{L^m(\Omega)}.
\]

\begin{proposition}\label{proW2}
Let $m\in (m,\infty]$, $s\in (n/2,\infty]$ and $\omega_0\Subset \omega_1\Subset \Omega$. There exists a constant $\mathfrak{c}=\mathfrak{c}(n,\Omega,m,s)>0$  such that for all $W\in L^m(\Omega,\mathbb{C}^n)$, $V\in L^s(D)$ and $u\in W^{2,p}(\Omega)$ we have
\begin{equation}\label{Wci}
\mathfrak{c}\|\nabla u\|_{L^2(\omega_0)}\le \|(-\Delta +W\cdot\nabla+V)u\|_{L^2(\Omega)}+d^{-1}\aleph(V,W)\|u\|_{L^2(\omega_1)},
\end{equation}
where $d:=\mathrm{dist}(\omega_0,\omega_1)$.
\end{proposition}

\begin{proof}
Let $(m,s)\in (n,\infty)\times (n/2,\infty)$. Pick $\chi \in C_0^\infty (\omega_1)$ satisfying $0\le \chi \le 1$, $\chi =1$ in a neighborhood of $\omega_0$ and $|\nabla \chi|\le \mathbf{k}d^{-1}$, where $\mathbf{k}$ is a generic universal constant. Let $W\in L^m(\Omega,\mathbb{C}^n)$, $V\in L^s(\Omega)$ and $u\in W^{2,p}(\Omega)$. Let $t>0$ and $\tau >0$. As we have done before we use the decompositions $V=V_t^1+V_t^2$ and $W=W_\tau^1+W_\tau ^2$ so that $V_t^1\in L^{n/2}(\Omega)$, $V_t^2\in L^\infty(D)$, $W_\tau ^1\in L^n(D,\mathbb{C}^n)$, $W_\tau ^2\in L^\infty(D,\mathbb{C}^n)$ and
\begin{align*}
&\|V_t^1\|_{L^{n/2}(\Omega)}\le t^{-(2s/n-1)}\kappa_V^{2s/n},\quad \|V_t^2\|_{L^\infty(\Omega)}\le t,
\\
&\|W_\tau^1\|_{L^m(\Omega)}\le \tau^{-(m/n-1)}\kappa_W^{m/n},\quad \|W_\tau^2\|_{L^\infty(\Omega)}\le \tau.
\end{align*}

Integrating by parts, we obtain
\begin{align*}
&\int_\Omega \chi^2 \bar{u}(-\Delta u+W\cdot \nabla u+Vu)dx=2\int_\Omega \chi \bar{u}\nabla \chi \cdot \nabla udx+\int_\Omega \chi^2|\nabla u|^2dx
\\
&\hskip 3cm +\int_\Omega \chi^2 \bar{u}W_\tau^1\cdot \nabla udx +\int_\Omega \chi^2 V_t^1|u|^2dx+\mathbf{J},
\end{align*}
where
\[
\mathbf{J}=\int_\Omega \chi^2 \bar{u}W_\tau^2\cdot \nabla udx +\int_\Omega \chi^2 V_t^2|u|^2dx .
\]
Thus,
\begin{align}
&\int_\Omega \chi^2 |\nabla u|^2dx \le  \|\chi u\|_{L^{p'}(\Omega)}\|(-\Delta +W\cdot \nabla +V)u\|_{L^p(\Omega)}+|\mathbf{J}|\label{W13}
\\
&\hskip 1cm +2\int_\Omega \chi |u||\nabla \chi \cdot \nabla u|dx+\int_\Omega \chi^2 |u||W_\tau^1\cdot \nabla u|dx +\int_\Omega \chi^2 |V_t^1||u|^2dx. \nonumber
\end{align}

Let $\epsilon >0$. Then
\begin{align*}
 &\|\chi u\|_{L^{p'}(\Omega)}\|(-\Delta +W\cdot \nabla +V)u\|_{L^p(\Omega)}\le \epsilon  \|\chi u\|_{L^{p'}(\Omega)}^2
 \\
 &\hskip 4cm +(4\epsilon)^{-1}\|(-\Delta +W\cdot \nabla +V)u\|_{L^p(\Omega)}^2.
 \end{align*}
 Since
 \begin{align*}
 \|\chi u\|_{L^{p'}(\Omega)}^2 &\le \sigma^2\|\nabla (\chi u)\|_{L^2(\Omega)}^2
 \\
 &\le 2\sigma^2 \left( \|\chi \nabla u\|_{L^2(\Omega)}^2+\|u \nabla \chi \|_{L^2(\Omega)}^2\right),
 \end{align*}
 we get
 \begin{align}
 &\|\chi u\|_{L^{p'}(\Omega)}\|(-\Delta +W\cdot \nabla +V)u\|_{L^p(\Omega)}\label{W14}
 \\
 &\hskip 2cm \le 2\epsilon \sigma^2 \left( \|\chi \nabla u\|_{L^2(\Omega)}^2+\|u \nabla \chi \|_{L^2(\Omega)}^2\right)\nonumber
 \\
 &\hskip 4cm +(4\epsilon)^{-1}\|(-\Delta +W\cdot \nabla +V)u\|_{L^p(\Omega)}^2.\nonumber
 \end{align}
 Using \eqref{W14} in \eqref{W13}, we get
 \begin{align}
&(1-2\epsilon \sigma^2)\int_\Omega \chi^2 |\nabla u|^2dx \le  (4\epsilon)^{-1}\|(-\Delta +W\cdot \nabla +V)u\|_{L^p(\Omega)}^2+|\mathbf{J}|\label{W15}
\\
&\hskip 3cm +2\epsilon \sigma^2\|u \nabla \chi \|_{L^2(\Omega)}^2 +2\int_\Omega \chi |u||\nabla \chi \cdot \nabla u|dx\nonumber
\\
&\hskip 4cm+\int_\Omega \chi^2 |u||W_\tau^1\cdot \nabla u|dx +\int_\Omega \chi^2 |V_t^1||u|^2dx. \nonumber
\end{align}

By combining \eqref{W15} and the following inequality
\begin{align*}
2\int_\Omega \chi |u||\nabla \chi \cdot \nabla u|dx &\le 2\int_\Omega [ |u||\nabla \chi|][ \chi |\nabla u|]dx
\\
&\le \epsilon \|\chi \nabla u\|_{L^2(\Omega)}^2+\epsilon^{-1}\|u \nabla \chi \|_{L^2(\Omega)}^2,
\end{align*}
we obtain
 \begin{align*}
&(1-\epsilon (2\sigma^2+1))\int_\Omega \chi^2 |\nabla u|^2dx \le  (4\epsilon)^{-1}\|(-\Delta +W\cdot \nabla +V)u\|_{L^p(\Omega)}^2+|\mathbf{J}|
\\
&\hskip 3cm +(2\epsilon \sigma^2+\epsilon^{-1})\|u \nabla \chi \|_{L^2(\Omega)}^2 
\\
&\hskip 4cm+\int_\Omega \chi^2 |u||W_\tau^1\cdot \nabla u|dx +\int_\Omega \chi^2 |V_t^1||u|^2dx. 
\end{align*}
Taking in this inequality $\epsilon=(4\sigma^2+2)^{-1}$, we find
 \begin{align}
&\int_\Omega \chi^2 |\nabla u|^2dx \le  \sigma_0\|(-\Delta +W\cdot \nabla +V)u\|_{L^p(\Omega)}^2+2|\mathbf{J}|\label{W16}
\\
&\hskip 1.5cm +\sigma_1\|u \nabla \chi \|_{L^2(\Omega)}^2 +2\int_\Omega \chi^2 |u||W_\tau^1\cdot \nabla u|dx +2\int_\Omega \chi^2 |V_t^1||u|^2dx, \nonumber
\end{align}
where $\sigma_0:=2\sigma^2+1$ and $\sigma_1:=2\sigma^2(2\sigma^2+1)^{-1}+8\sigma^2+4$.

We recall that we already demonstrated the following inequality
\[
2\int_\Omega \chi^2 |V_t^1||u|^2dx\le 4\sigma^2t^{-(2s/n-1)}\kappa_V^{2s/n}\left(\|u \nabla \chi \|_{L^2(\Omega)}^2+\|\chi \nabla u \|_{L^2(\Omega)}^2\right).
\]
In the sequel, we fix with $t=(4\sigma)^{2n/(2s-n)}\kappa_V^{2s/(2s-n)}$. In that case, we have
\begin{equation}\label{W17}
2\int_\Omega \chi^2 |V_t^1||u|^2dx\le 4^{-1}\left(\|u \nabla \chi \|_{L^2(\Omega)}^2+\|\chi \nabla u \|_{L^2(\Omega)}^2\right).
\end{equation}

On the other hand, we have 
\begin{align*}
&2\int_\Omega \chi^2 |u||W_\tau^1\cdot \nabla u|dx\le 2\|\chi u\|_{L^{p'}(\Omega)}\|\chi W_\tau^1\cdot \nabla u\|_{L^p(\Omega)} 
\\
&\hskip 1.5cm\le 2\sigma \|\nabla (\chi  u)\|_{L^2(\Omega)}\|W_\tau ^1\|_{L^n(\Omega)}\|\chi \nabla u\|_{L^2(\Omega)} 
\\
&\hskip 1.5cm\le \sigma^2 \|W_\tau^1\|_{L^n(\Omega)} \|\nabla (\chi u)\|_{L^2(\Omega)}^2+\|W_\tau^1\|_{L^n(\Omega)}\|\chi \nabla u\|_{L^2(\Omega)}^2 .
\end{align*}
As
\[
\|\nabla (\chi u)\|_{L^2(\Omega)}^2\le 2\|u \nabla \chi \|_{L^2(\Omega)}^2+2\|\chi \nabla u\|_{L^2(\Omega)}^2,
\]
we find
\begin{align*}
&2\int_\Omega \chi^2 |u||W_\tau^1\cdot \nabla u|dx
\\
&\hskip 1.5cm \le (2\sigma^2+1) \|W_\tau^1\|_{L^n(\Omega)}\|\chi \nabla u\|_{L^2(\Omega)}^2+2\sigma^2 \|W_\tau^1\|_{L^n(\Omega)}\|u \nabla \chi \|_{L^2(\Omega)}^2
\end{align*}
and then
\begin{align}
&2\int_\Omega \chi^2 |u||W_\tau^1\cdot \nabla u|dx\label{W18}
\\
&\hskip 2cm \le (2\sigma^2+1) \tau^{-(m/n-1)}\kappa_W^{m/n}\|\chi \nabla u\|_{L^2(\Omega)}^2\nonumber
\\
&\hskip 4cm+2\sigma^2 \tau^{-(m/n-1)}\kappa_W^{m/n}\|u \nabla \chi \|_{L^2(\Omega)}^2.\nonumber
\end{align}
In the following, we fix $\tau=[4(2\sigma^2+1)]^{n/(m-n)}\kappa_W^{m/(m-n)}$. In this case \eqref{W18} yields
\begin{equation}\label{W19}
2\int_\Omega \chi^2 |u||W_\tau^1\cdot \nabla u|dx\le 4^{-1}\|\chi \nabla u\|_{L^2(\Omega)}^2+2^{-1}\sigma^2(2\sigma^2+1)^{-1}\|u \nabla \chi \|_{L^2(\Omega)}^2.
\end{equation}
Using \eqref{W17} and \eqref{W19} in \eqref{W16}, we obtain
 \begin{equation}\label{W20}
\int_\Omega \chi^2 |\nabla u|^2dx \le  2\sigma_0\|(-\Delta +W\cdot \nabla +V)u\|_{L^p(\Omega)}^2+2|\mathbf{J}|
 +\sigma_2 \|u \nabla \chi \|_{L^2(\Omega)}^2 , 
\end{equation}
where $\sigma_2:=2\sigma_1+1/2+\sigma^2(2\sigma^2+1)^{-1}$.

Next, we have 
\begin{align*}
\int_\Omega \chi^2 |u||W_\tau^2\cdot \nabla u|dx&\le \|\chi u\|_{L^2(\Omega)}\|W_\tau^2\|_{L^\infty (\Omega)}\|\chi \nabla u\|_{L^2(\Omega)}
\\
&\le 4^{-1}\|\chi \nabla u\|_{L^2(\Omega)}^2+\|W_\tau^2\|_{L^\infty (D)}^2\|\chi u\|_{L^2(\Omega)}^2
\\
&\le 4^{-1}\|\chi \nabla u\|_{L^2(\Omega)}^2+\tau^2\|\chi u\|_{L^2(\Omega)}^2
\\
&\le 4^{-1}\|\chi \nabla u\|_{L^2(\Omega)}^2+[4(2\sigma^2+1)]^{2n/(m-n)}\kappa_W^{2m/(m-n)}]\|\chi u\|_{L^2(\Omega)}^2
\end{align*}
and
\[
\int_\Omega |V_t^2|\chi ^2|u|^2\le t\|\chi u\|_{L^2(\Omega)}^2=(4\sigma)^{2n/(2s-n)}\kappa_V^{2s/(2s-n)}\|\chi u\|_{L^2(\Omega)}^2.
\]
Whence 
\begin{equation}\label{W21}
2|\mathbf{J}|\le 2^{-1}\|\chi \nabla u\|_{L^2(\Omega)}^2+\mathfrak{c}\left(\kappa_W^{2m/(m-n)}+\kappa_V^{2s/(2s-n)}\right)\|\chi u\|_{L^2(\Omega)}^2.
\end{equation}
Here and henceforth, $\mathfrak{c}=\mathfrak{c}(n,\Omega,m,s)>0$ denotes a generic constant.

Putting together \eqref{W20} and \eqref{W21}, we obtain
\begin{align*}
&\mathfrak{c} \|\chi \nabla u\|_{L^2(\Omega)}^2\le \|(-\Delta +W\cdot \nabla +V)u\|_{L^p(\Omega)}^2
\\
&\hskip 3cm+\left(1+\kappa_W^{2m/(m-n)}+\kappa_V^{2s/(2s-n)}\right)\|\chi u\|_{L^2(\Omega)}^2,
\end{align*}
from which the expected inequality follows when $(m,s)\in (n,\infty)\times (n/2,\infty)$. The other cases are proved similarly.
\end{proof}

In the following, $\mathbf{c}=\mathbf{c}(n,m,s)\ge 1$ and $\mathbf{c}_1=\mathbf{c}_1(n,m,s,r_0)>0$ will denote generic constants and,  for $V\in L^s(\Omega)$ and $W\in L^m(\Omega,\mathbb{C}^n)$ with $(m,s)\in (n,\infty]\times (n/2,\infty]$, $\varphi_{m,s}(V,W)$ will denote a generic constant of the form
\[
\varphi_{m,s}(V,W):=e^{\mathbf{c}_1(\kappa_W^{\ell_m}+\kappa_V^{k_s})},
\]
where $\kappa_V:=\|V\|_{L^s(\Omega)}$ and $\kappa_W:=\|W\|_{L^m(\Omega)}$, and $\ell_m$ and $k_s$ are as in Theorem \ref{thmW1}.

In light of Proposition \ref{proW2} and Theorem \ref{thmW1}, we can repeat the proof of Theorem \ref{mthm1} to obtain the following three-ball inequality

\begin{theorem}\label{thmW2}
Let $(m,s)\in \mathscr{N}$ and $r_0>0$ chosen so that $\Omega^{r_0}$ is nonempty and $0<r<r_0/3$. For all $x_0\in \Omega^{3r}$, $W\in L^m(\Omega,\mathbb{C}^n)$, $V\in L^s(\Omega)$ and $u\in W^{2,p}(\Omega)$ satisfying $(-\Delta+W\cdot \nabla +V)u=0$ we have
\begin{equation}\label{W22}
\|u\|_{L^2(B(x_0,2r))}\le 
\mathbf{c}\varphi_{s,m}(V,W)\|u\|_{L^2(B(x_0,r))}^\alpha \|u\|_{L^2(B(x_0,3r))}^{1-\alpha}.
\end{equation}
\end{theorem}

Under the assumptions and the  notations of Section \ref{S5}, for $V\in L^s(\Omega)$ and $W\in L^m(\Omega,\mathbb{C}^n)$ with $(m,s)\in (n,\infty]\times (n/2,\infty]$, we get the following theorem by modifying slightly the proof of Theorem \ref{thma2}.

\begin{theorem}\label{thmW3}
Let $(m,s)\in \mathscr{N}$, $0<\mathfrak{t}<1/2$ and $\omega\Subset \Omega$. There exist $r_\ast=r_\ast(\Omega,\omega,\mathfrak{r})\le \mathfrak{r}/4$ such that for all $W\in L^m(\Omega,\mathbb{C}^n)$, $V\in L^s(\Omega)$, $0<r<r_\ast$ and $u\in W^{2,p}(\Omega)$ satisfying $(-\Delta +W\cdot \nabla +V)u=0$ we have
\begin{equation}\label{W23}
\|u\|_{L^2(\Omega)}\le \mathbf{c}\varphi_{m,s}(V,W)\left(e^{\bar{\mathbf{c}}\varrho(r)}\|u\|_{L^2(\omega)}+r^\mathfrak{t} \|u\|_{H^1(\Omega)}\right),
\end{equation}
where $\bar{\mathbf{c}}=\bar{\mathbf{c}}(n,\Omega,\mathfrak{t})>0$ is a constant.
\end{theorem}

Moreover, if we further assume that $\Omega$ is of class $C^{1,1}$, then we prove similarly to Theorem \ref{thmaa1} the following result.

\begin{theorem}\label{thmW4}
Assume that $\Omega$ is of class $C^{1,1}$. Let $\omega\Subset \Omega$, $(m,s)\in \mathscr{N}$, $0<\mathfrak{t} <1/2$. Let $\zeta:=(n,\Omega,\omega, \mathfrak{r},\mathfrak{t})$. There exist three constants $\upsilon =\upsilon(\zeta)>0$, $\varsigma=\varsigma(\zeta)>0$ and $\tilde{\mathbf{c}}=\tilde{\mathbf{c}}(\zeta,m,s)>0$ such that for all $W\in L^m(\Omega,\mathbb{C}^n)$, $V\in L^s(\Omega)$, $u\in W^{2,p}(\Omega)$ satisfying $(-\Delta +W\cdot \nabla+V)u=0$ and $0<r<1$ we have
\begin{equation}\label{W24}
\|u\|_{L^2(\Omega)}\le \tilde{\mathbf{c}}\varphi_{m,s}(V,W)\left(e^{\upsilon r^{-\varsigma}}\|u\|_{L^2(\omega)}+r^\mathfrak{t} \|u\|_{H^1(\Omega)}\right).
\end{equation}
\end{theorem}

In the remainder of this section, we show that the condition $\min(m,2s)>3n/2-1$ can be suppressed by using, instead of \cite[Theorem 1.1]{DET}, directly the interpolation inequality \cite[Theorem 1.1]{CET}. We emphasize that this interpolation inequality is obtained from Carleman's inequalities of \cite[Theorem 1.1]{DET}, combined with a rather technical argument based on a lemma due to Wolff \cite{Wo}. We only show how we prove the three-ball inequality in this case. The other results must be modified accordingly.

We use hereafter the following convention $h(\infty)=\lim_{\rho \rightarrow \infty}h(\rho)$. Consider the notations
\[
\gamma_s=\gamma(n,s):=
\left\{
\begin{array}{ll}
4s/[3(2s-n)+2],  &s\in [n,\infty],
\\
4ns/[(3n+2)(2s-n)],\quad &s\in (n/2,n],
\end{array}
\right.
\]
and 
\[
\delta_m=\delta_m(n,m):=2m/(m-n),\quad m\in (n,\infty].
\]

\begin{theorem}\label{W5}
Let $(m,s)\in (n,\infty]\times (n/2,\infty]$. For all $x_0\in \Omega^{3r}$, $0<r<r_0/3$, $V\in L^s(\Omega)$, $W\in L^m(\Omega ,\mathbb{C}^n)$ and $u\in W^{2,p}(\Omega)$ satisfying $(-\Delta+W\cdot \nabla +V)u=0$ we have
\[
\|u\|_{L^2(B(x_0,2r))}\le 
\mathbf{c}e^{\mathbf{c}_1(\kappa_W^{\delta_m}+\kappa_V^{\gamma_s})}\|u\|_{L^2(B(x_0,r))}^\alpha \|u\|_{L^2(B(x_0,3r))}^{1-\alpha}.
\]
\end{theorem}

\begin{proof}
We use in this proof the notation $B_\rho:=B(0,\rho)\subset \mathbb{R}^n$, $\rho >0$. 
For $(m,s)\in (n,\infty]\times (n/2,\infty]$, Let  $\mathcal{V}\in L^s(B_3)$, $\mathcal{W}\in L^m(B_3,\mathbb{C}^n)$ and $w\in W^{2,p}(B_3)$ satisfying $(-\Delta + \mathcal{W}\cdot \nabla + \mathcal{V})w=0$. From \cite[Theorem 1.1]{CET}, there exist $c_0=c_0(n)>0$ and $0<\alpha=\alpha(n)<1$ such that
\begin{equation}\label{W25}
\|w\|_{H^1(B_2)}\le c_0e^{c_0(\varkappa_{\mathcal{W}}^{\delta_m}+\varkappa_\mathcal{V}^{\gamma_s} )}\|w\|_{H^1(B_{1/2})}^\alpha \|w\|_{H^1(B_{5/2})}^{1-\alpha},
\end{equation}
where $\varkappa_\mathcal{W}:=\|\mathcal{V}\|_{L^m(B_3)}$ and $\varkappa_\mathcal{V}:=\|\mathcal{V}\|_{L^s(B_3)}$.

Combined with Caccioppoli's inequality of Proposition \ref{proW2}, \eqref{W25} yields
\[
\|w\|_{L^2(B_2)}\le \mathbf{c}\aleph(V,W)e^{c_0(\varkappa_{\mathcal{W}}^{\delta_m}+\varkappa_\mathcal{V}^{\gamma_s}) }\|w\|_{L^2(B_1)}^\alpha \|w\|_{L^2(B_3)}^{1-\alpha}
\]
and then
\begin{equation}\label{W26}
\|w\|_{L^2(B_2)}\le \mathbf{c}e^{\mathbf{c}(\varkappa_{\mathcal{W}}^{\delta_m}+\varkappa_\mathcal{V}^{\gamma_s}) }\|w\|_{L^2(B_1)}^\alpha \|w\|_{L^2(B_3)}^{1-\alpha}.
\end{equation}

Next, let $x_0\in \Omega^{3r}$, $V\in L^s(\Omega)$, $W\in L^m(\Omega)$ and $u\in W^{2,p}(\Omega)$ satisfying $(-\Delta +W\cdot \nabla+V)u=0$, and set
\[
w(y):=u(x_0+ry),\quad  \mathcal{V}(y):=r^2V(x_0+ry),\quad \mathcal{W}:=rW(x_0+ry)\quad y\in B_3.
\]
Then $(-\Delta +\mathcal{W}\cdot \nabla + \mathcal{V})w=0$ and
\begin{align*}
&\varkappa_\mathcal{V}=r^{2-n/s}\|V\|_{L^s(B(x_0,3r))}\le r^{2-n/s}\kappa_V,
\\
&\varkappa_\mathcal{W}=r^{1-n/m}\|W\|_{L^m(B(x_0,3r))}\le r^{1-m/s}\kappa_W,
\end{align*}
with the convention that $2-n/s=2$ if $s=\infty$ and $1-n/m=1$ if $m=\infty$. Here, we recall that $\kappa_V:=\|V\|_{L^s(\Omega)}$ and $\kappa_W:=\|W\|_{L^m(\Omega)}$. Let $k:=\min(\gamma_s(2-n/s),\gamma_m(1-m/s))$ Applying \eqref{W26}, we obtain
\[
\|u\|_{L^2(B(x_0,2r))}\le 
\mathbf{c}e^{\mathbf{c}r_0^k(\kappa_W^{\delta_m}+\kappa_V^{\gamma_s})}\|u\|_{L^2(B(x_0,r))}^\alpha \|u\|_{L^2(B(x_0,3r)}^{1-\alpha}.
\]
The expected inequality then follows.
\end{proof}

\section{Quantitative uniqueness of continuation from Cauchy data}\label{S11}

In this section $\Gamma:=\partial \Omega$. Let $\psi \in C^2(\bar{\Omega})$ chosen so that there exists $\varrho>0$ such that
\[
\psi \ge \varrho, \quad |\nabla \psi |\ge \varrho .
\]
Define $\phi_\lambda:=e^{\lambda \psi}$, $\lambda >0$ and $\zeta:=\left(\varrho ,\|\psi\|_{C^2(\bar{\Omega})}\right)$. It follows from \cite[Theorem 2.8]{ChLN} that there exist $\lambda=\lambda(\zeta)>0$, $\tau_0=\tau_0(\zeta)\ge 1$ and $\bar{c}=\bar{c}(\zeta)>0$ so that for all $\tau \ge \tau_0$ and $u\in H^2(\Omega)$ we have
\begin{align}
&\tau^{3/2}\|u\|_{L_\tau^2(\Omega)}+\tau^{1/2}\|\nabla u\|_{L_\tau^2(\Omega)}\label{11.1}
\\
&\hskip 2cm\le \bar{c}\left(\|\Delta u\|_{L_\tau^2(\Omega)}+\tau^{3/2}\|u\|_{L_\tau^2(\Gamma)}+\tau^{1/2}\|\nabla u\|_{L_\tau^2(\Gamma)}\right).\nonumber
\end{align}
Here $L_\tau^2(\Omega):=L^2(D, \Phi_\tau (x)dx)$ and $L_\tau^2(\Gamma):=L^2(\Gamma, \Phi_\tau (x)ds(x))$  with $\Phi_\tau:=e^{2\tau \phi_\lambda }$.

Let $s\in [n,\infty]$, $V\in L^s(\Omega)$, $W\in L^\infty(\Omega,\mathbb{C}^n)$ and $w\in H^1(\Omega)$, and set $\kappa_V:=\|V\|_{L^s(\Omega)}$ and $\kappa_W:=\|W\|_{L^\infty(\Omega)}$. Let $s':=2s/(s-2)$ with the convention that $s'=2$ when $s=\infty$. In the following $\mathfrak{c}=\mathfrak{c}(n,\Omega,s)>0$ is a generic constant. Since $L^{p'}(\Omega)$ is continuously embedded in $L^{s'}(\Omega)$ and $H^1(\Omega)$ is continuously embedded in $L^{p'}(\Omega)$, we get
\begin{align*}
\|Vw\|_{L^2(\Omega)}&\le \|V\|_{L^s(\Omega)}\|w\|_{L^{s'}(\Omega)}
\\
&\le \mathfrak{c}\|V\|_{L^s(\Omega)}\|w\|_{L^{p'}(\Omega)}
\\
&\le \mathfrak{c}\kappa_V\left(\|w\|_{L^2(\Omega)}+\|\nabla w\|_{L^2(\Omega)}\right)
\end{align*}
Taking in this inequality $w=\Phi_\tau u $ and using 
\[
\nabla (\Phi_\tau u)=2\tau \Phi_\tau u \nabla \phi_\lambda +\Phi_\tau\nabla u,
\]
we obtain from  \eqref{11.1}  
\begin{align}
&(\tau^{3/2}-\bar{c}\tau\kappa_V)\|u\|_{L_\tau^2(\Omega)}+(\tau^{1/2}-\bar{c}(\kappa_V+\kappa_W))\|\nabla u\|_{L_\tau^2(\Omega)}\label{11.2}
\\
&\hskip .2cm\le \bar{c}\left(\|(-\Delta+W\cdot \nabla +V) u\|_{L_\tau^2(\Omega)}+\tau^{3/2}\|u\|_{L_\tau^2(\Gamma)}+\tau^{1/2}\|\nabla u\|_{L_\tau^2(\Gamma)}\right).\nonumber
\end{align}
Here and henceforth $\bar{c}=\bar{c}(\zeta,n,\Omega,s)>0$ is a generic constant. 

As we have done before, we obtain from \eqref{11.2} the following inequality which is valid for all $\tau >0$

\begin{align}
&\|u\|_{L_\tau^2(\Omega)}\le \bar{c}e^{\bar{c}(\kappa_V^2+\kappa_W^2)}\|(-\Delta+W\cdot \nabla +V) u\|_{L_\tau^2(\Omega)}\label{11.3}
\\
&\hskip 4cm +\bar{c}e^{\bar{c}(\kappa_V^2+\kappa_W^2)}\left(\|u\|_{L_\tau^2(\Gamma)}+\|\nabla u\|_{L_\tau^2(\Gamma)}\right). \nonumber
\end{align}
In the case $s=\infty$ and $W=0$, we get in similar manner the following inequality which is valid for $\tau >0$

\begin{align}
&\tau^{3/2}\|u\|_{L_\tau^2(\Omega)}\label{11.4}
\\
&\hskip 1cm\le \bar{c}e^{\bar{c}\kappa_V^{2/3}}\left(\|(-\Delta+V) u\|_{L_\tau^2(\Omega)}+\|u\|_{L_\tau^2(\Gamma)}+\|\nabla u\|_{L_\tau^2(\Gamma)}\right).\nonumber
\end{align}

Next, let $S$ be a nonempty subset of $\Gamma$ and fix $\tilde{x}\in S$. Then let $r_0>0$ be such that $B(\tilde{x},r_0)\cap \Gamma\subset S$. As $\Omega$ is on one side of its boundary, we find $x_0\in \mathbb{R}^n\setminus \bar{\Omega}$ sufficiently close to $\tilde{x}$ in such a way that $\rho=\mathrm{dist}(x_0,\bar{\Omega})=\mathrm{dist}(x_0,\bar{B}(\tilde{x},r_0)\cap\Gamma)$. We fix $r>0$ such that $B(x_0,r+\rho)\cap \Gamma :=S_0\subset S$.

Let $D=\Omega \cap B(x_0,\rho+r)$ and define $\psi(x)=\ln [(2\rho+r)^2/|x-x_0|^2]$. We verify
\begin{align*}
&\psi (x)\ge \ln [(2\rho+r)^2/(\rho+r)^2]\; (>0),\quad x\in \bar{D},
\\
&|\nabla \psi(x)|\ge 2|x-x_0|^{-1}\ge 2(\rho+r)^{-1}, \quad x\in \bar{D}.
\end{align*}

Let $\chi\in C_0^\infty (D)$ satisfying $0\le \chi \le 1$, $\chi=1$ in a neighborhood of $\Omega\cap B(x_0,\rho+r/2)$. Pick $V\in L^s(\Omega)$ with $s\in [n,\infty]$, $W\in L^\infty(\Omega,\mathbb{C}^n)\setminus \{0\}$ and $u\in H^2(\Omega)$ satisfying $(-\Delta +W\cdot \nabla +V)u=0$. Using
\[
(-\Delta +W\cdot \nabla +V)(\chi u)=-u\Delta \chi -2\nabla \chi \cdot \nabla u+uW\cdot \nabla \chi,
\]
we get
\begin{equation}\label{11.5}
|(-\Delta +W\cdot \nabla +V)(\chi u)|^2\le \hat{\mathbf{c}}(1+\kappa_W)^2(|u|^2+|\nabla u|^2).
\end{equation}
Here $\hat{\mathbf{c}}=\hat{\mathbf{c}}(S)>0$ is a constant.

In the following, $\hat{\mathbf{c}}=\hat{\mathbf{c}}(n,\Omega,s, S)>0$ will denote a generic constant and
\begin{align*}
&D_0:=B(x_0,\rho+r/4)\cap \Omega, 
\\
&D_1=[B(x_0,\rho+r)\setminus \bar{B}(x_0,\rho+r/2)]\cap \Omega.
\end{align*}

Applying \eqref{11.3} with $u$ replaced by $\chi u$, \eqref{11.5} and using the fact that 
\[
\mathrm{supp}((-\Delta +W\cdot \nabla +V)(\chi u))\subset D_1,
\]
we obtain
\begin{align}
&\|u\|_{L_\tau^2(D_0)}\le\hat{\mathbf{c}}e^{\hat{\mathbf{c}}(\kappa_V^2+\kappa_W^2)}\left(\|u\|_{L_\tau^2(D_1)}+\|\nabla u\|_{L_\tau^2(D_1)}\right)\label{11.6}
\\
&\hskip 3cm+\hat{\mathbf{c}}e^{\hat{\mathbf{c}}(\kappa_V^2+\kappa_W^2)}\left(\|u\|_{L_\tau^2(S_0)}+\|\nabla u\|_{L_\tau^2(S_0)}\right).\nonumber
\end{align}
Define
\begin{align*}
&a:=\frac{(2\rho+r)^{2\lambda}}{(\rho+r/4)^{2\lambda}}-\frac{(2\rho+r)^{2\lambda}}{(\rho+r/2)^{2\lambda}},
\\
&b:=\frac{(2\rho+r)^{2\lambda}}{\rho^{2\lambda}}-\frac{(2\rho+r)^{2\lambda}}{(\rho+r/4)^{2\lambda}}.
\end{align*}
Using \eqref{11.6}, we obtain
\begin{align}
&\|u\|_{L^2(D_0)}\le\hat{\mathbf{c}}e^{\hat{\mathbf{c}}(\kappa_V^2+\kappa_W^2)}e^{-a \tau}\left(\|u\|_{L^2(D_1)}+\|\nabla u\|_{L^2(D_1)}\right)\label{11.7}
\\
&\hskip 3cm+\hat{\mathbf{c}}e^{\hat{\mathbf{c}}(\kappa_V^2+\kappa_W^2)}e^{b \tau}\left(\|u\|_{L^2(S_0)}+\|\nabla u\|_{L^2(S_0)}\right).\nonumber
\end{align}

Let $\bar{\omega}\Subset D_0$ be fixed arbitrarily. Then we obtain from \eqref{11.7}
\begin{equation}\label{11.8}
\|u\|_{L^2(\bar{\omega})}\le \hat{\mathbf{c}}e^{\hat{\mathbf{c}}(\kappa_V^2+\kappa_W^2)}\left(e^{-a \tau}\|u\|_{H^1(\Omega)}+e^{b \tau}\left(\|u\|_{L^2(S_0)}+\|\nabla u\|_{L^2(S_0)}\right)\right).
\end{equation}

We proceed similarly in the case $s=\infty$ and $W=0$ to derive the following inequality
\begin{equation}\label{11.8.1}
\|u\|_{L^2(\bar{\omega})}\le \hat{\mathbf{c}}e^{\hat{\mathbf{c}}\kappa_V^{2/3}}\left(e^{-a \tau}\|u\|_{H^1(\Omega)}+e^{b \tau}\left(\|u\|_{L^2(S_0)}+\|\nabla u\|_{L^2(S_0)}\right)\right).
\end{equation}
Here $\hat{\mathbf{c}}=\hat{\mathbf{c}}(n,\Omega,S)$.

In order to put together \eqref{11.8} and \eqref{11.8.1} we consider the notation
\[
\Upsilon (V,W):=
\left\{
\begin{array}{ll}
e^{\hat{\mathbf{c}}(\kappa_V^2+\kappa_W^2)},\quad &W\ne 0,
\\
e^{\hat{\mathbf{c}}\kappa_V^{2/3}}, &s=\infty ,\; W=0.
\end{array}
\right.
\]
With this notation, we have
\begin{equation}\label{11.8.2}
\|u\|_{L^2(\bar{\omega})}\le \hat{\mathbf{c}}\Upsilon (V,W)\left(e^{-a \tau}\|u\|_{H^1(\Omega)}+e^{b \tau}\left(\|u\|_{L^2(S_0)}+\|\nabla u\|_{L^2(S_0)}\right)\right).
\end{equation}

By taking $\epsilon=e^{-a \tau}$ in \eqref{11.8.2}, we can state the following result.

\begin{theorem}\label{thm11.1}
Let $S$ be a nonempty open subset of $\Gamma$, $s\in [n,\infty]$ and set $\zeta:=(n,\Omega, S)$. Then there exist two constants $c=c (\zeta)>0$, $\hat{\mathbf{c}}=\hat{\mathbf{c}}(\zeta,s)>0$ and $\bar{\omega}=\bar{\omega} (\Omega, S) \Subset \Omega$ such that for all $V\in L^s(\Omega)$, $W\in L^\infty(\Omega,\mathbb{C}^n)$ and $u\in H^2(\Omega)$ satisfying $(-\Delta+W\cdot \nabla +V)u=0$ and $0<\epsilon <1$ we have
\begin{equation}\label{11.9}
\|u\|_{L^2(\bar{\omega})}\le \hat{\mathbf{c}}\Upsilon (V,W)\left(\epsilon\|u\|_{H^1(\Omega)}+\epsilon^{-c}\left(\|u\|_{L^2(S)}+\|\nabla u\|_{L^2(S)}\right)\right).
\end{equation}
\end{theorem}

Under the assumptions of the preceding theorem, we apply Theorem \ref{thmW3} with $\omega=\bar{\omega}$ to get for $0<r<r_\ast\, (=r_\ast(\Omega,\bar{\omega},\mathfrak{r}))$
\begin{equation}\label{11.10}
\|u\|_{L^2(\Omega)}\le \mathbf{c}\varphi_{\infty,s}(V,W)\left(e^{\bar{\mathbf{c}}\varrho(r)}\|u\|_{L^2(\bar{\omega})}+r^\mathfrak{t} \|u\|_{H^1(\Omega)}\right).
\end{equation}
The constants are those appearing in \eqref{W23} with $\omega=\bar{\omega}$ and $m=\infty$.

In the following, we use the notation
\begin{equation}\label{11.11}
\mathfrak{F}(V,W):=\Upsilon(V,W)\varphi_{\infty,s}(V,W).
\end{equation}
Let $s$, $V$, $W$ and $u$ be as in Theorem \ref{thm11.1}. Then a combination of \eqref{11.9} and \eqref{11.10} gives for $0<\epsilon <1$ and $0<r<r_\ast$
\begin{align*}
&\|u\|_{L^2(\Omega)}\le \hat{\mathbf{c}}\mathfrak{F}(V,W)\left( (\epsilon e^{\bar{\mathbf{c}}e^{\bar{\mathfrak{h}}r^{-n}}} +r^\mathfrak{t})\|u\|_{H^1(\Omega)}\right.
\\
&\hskip 4cm \left. +e^{\bar{\mathbf{c}}e^{\bar{\mathfrak{h}}r^{-n}}} \epsilon^{-c}\left(\|u\|_{L^2(S)}+\|\nabla u\|_{L^2(S)}\right)\right).
\end{align*}

Reducing $r_\ast$ if necessary, we assume that $0<r^\mathfrak{t}e^{-\bar{\mathbf{c}}e^{\bar{\mathfrak{h}}r^{-n}}}<1$ for all $0<r<r_\ast$. In the last inequality, we take $\epsilon=r^\mathfrak{t}e^{-\bar{\mathbf{c}}e^{\bar{\mathfrak{h}}r^{-n}}}$. Proceeding in this way, we obtain the following result.

\begin{theorem}\label{thm11.2}
Let $S$ be a nonempty open subset of $\Gamma$, $0<\mathfrak{t}<1/2$, $s\in [n,\infty]$ and set $\zeta:=(n,\Omega,\mathfrak{t},S)$. Then there exist two constants $\bar{\mathbf{c}}=\bar{\mathbf{c}}(\zeta)>0$ and $\hat{\mathbf{c}}=\hat{\mathbf{c}}(\zeta,s)>0$ such that for all $V\in L^s(\Omega)$, $W\in L^\infty(\Omega,\mathbb{C}^n)$ and $u\in H^2(\Omega)$ satisfying $(-\Delta+W\cdot \nabla +V)u=0$ and $0<r <r_\ast$ we have
\begin{equation}\label{11.12}
\|u\|_{L^2(\Omega)}\le \hat{\mathbf{c}}\mathfrak{F}(V,W)\left(r^\mathfrak{t}\|u\|_{H^1(\Omega)}+e^{\bar{\mathbf{c}} e^{\hat{\mathfrak{h}}r^{-n}}} \left(\|u\|_{L^2(S)}+\|\nabla u\|_{L^2(S)}\right)\right),
\end{equation}
where $\mathfrak{F}$ is given by \eqref{11.11} and $\hat{\mathfrak{h}}:=\bar{\mathfrak{h}}+1$.
\end{theorem}

By proceeding as in the proof of Theorem \ref{thm11.2} and using Theorem \ref{thmW4} instead of Theorem \ref{thmW3}, we prove the following result.

\begin{theorem}\label{thm11.3}
Assume that $\Omega$ is of class $C^{1,1}$. Let $S$ be a nonempty open subset of $\partial \Omega$, $0<\mathfrak{t}<1/2$, $s\in [n,\infty]$ and set $\zeta:=(n,\Omega,\mathfrak{r}, \mathfrak{t},S)$. Then there exist three constants $\upsilon=\upsilon(\zeta)>0$, $\varsigma=\varsigma(\zeta)>0$ and $\tilde{\mathbf{c}}=\tilde{\mathbf{c}}(\zeta,s)>0$ such that for all $V\in L^s(\Omega)$, $W\in L^\infty(\Omega,\mathbb{C}^n)$ and $u\in H^2(\Omega)$ satisfying $(-\Delta+W\cdot \nabla +V)u=0$ and $0<r <1$ we have
\begin{equation}\label{11.14}
\|u\|_{L^2(\Omega)}\le \tilde{\mathbf{c}}\mathfrak{F}(V,W)\left(r^\mathfrak{t}\|u\|_{H^1(\Omega)}+e^{\upsilon r^{-\varsigma}} \left(\|u\|_{L^2(S)}+\|\nabla u\|_{L^2(S)}\right)\right).
\end{equation}
\end{theorem}

\subsection*{Acknowledgement}
We would like to thank Antonius Frederik Maria ter Elst for the fruitful discussion we had on the fact that $\Omega^r$ is connected for small $r>0$ when $\Omega$ is Lipschitz, and for pointing us to his work \cite{ER} where this result is proven.
 
This work was supported by JSPS KAKENHI Grant Number JP25K17280 and JP23KK0049.

\appendix

\section{Covering lemma}\label{appA}

\begin{lemma}\label{lemA1}
Let $K$ be a compact subset of $\mathbb{R}^n$ satisfying $\mathrm{diam}(K)\le \mathbf{d}$, for some $\mathbf{d}>0$, and $\epsilon_0>0$. There exists a constant $\hat{c}=\hat{c}(n,\mathbf{d},\epsilon_0)>0$ such that for all $0<\epsilon \le \epsilon_0$, $K$ can be covered by $\ell$ balls of radius $\epsilon$ with $\ell \le \hat{c} \epsilon^{-n}$
\end{lemma}
\begin{proof}
Let $x_1\in K$ be arbitrarily chosen. If $K\setminus\bar{B}(x_1,\epsilon)\ne \emptyset$, we fix $x_2\in K\setminus\bar{B}(x_1,\epsilon)$. If $K\setminus[\bar{B}(x_1,\epsilon)\cup \bar{B}(x_2,\epsilon)]\ne \emptyset$, we fix $x_3\in K\setminus[\bar{B}(x_1,\epsilon)\cup \bar{B}(x_2,\epsilon)]\ne \emptyset$ and so on. By induction in $j$, we construct a sequence $(x_j)$ in $K$ satisfying $|x_j-x_k|\ge \epsilon$ if $j\ne k$. By compactness, the sequence $(x_j)$ is finite because otherwise it would not have a convergent subsequence. Let $\ell$ be the number of element of $(x_j)$ so that $K\subset \bigcup_{j=1}^\ell \bar{B}(x_j,\epsilon)$. By construction,
\[
\bar{B}(x_j,\epsilon/2)\cap \bar{B}(x_k,\epsilon/2)=\emptyset \quad j\ne k.
\]
This and the fact that
\[
\bigcup_{j=1}^\ell \bar{B}(x_j,\epsilon/2)\subset \tilde{K}:=\{x\in \mathbb{R}^n;\; \mathrm{dist}(x,K)\le \epsilon_0/2\}
\]
yield $\ell (\epsilon/2)^n\le |\tilde{K}|$. The expected result follows by using that $\tilde{K}$ is contained in a cube $Q$ with $|Q|\le (\mathbf{d}+\epsilon_0)^n$.
\end{proof}

\section{Geometric Lemma}\label{appB}

\begin{lemma}\label{lemGeo}
Let $Q=\bigcup_{j \in J} Q_j$ be the connected union of $\ell \, (=|J|)$ closed cubes with edges of length $r/\sqrt{n}$, for some $r>0$. All $x,y\in Q$ can be connected by a broken line of length less than or equal to $\ell r$.
\end{lemma}

\begin{proof}
Let $V$ denotes the union of the vertices of $\{Q_j;\; j\in J\}$. Let $x,y\in Q$ be with $x\in Q_j$ and $y\in Q_k$, for some $j,k\in J$. Let $x_v\in V\cap Q_j$ and $y_v\in V\cap Q_k$ be fixed arbitrarily. There exists in $Q$ a broken line connecting $x_v$ to $y_v$ consisting of $m$ segments $[v_j,v_{j+1}]$ such that $m\le \ell-2$, $v_0=x_v$ and $v_m=y_v$ (note that $]v_j,v_{j+1}[$ is contained inside a cube, on a face or an edge). Since $|[x,x_v]|\le r$, $|[y,y_v]|\le r$ and $|[v_j,v_{j+1}]|\le r$, $j=0,\ldots ,m-1$, we conclude that any two points of $Q$ can be connected by a broken line of length less than or equal to $\ell r$.
\end{proof}


\begin{thebibliography}{99}

\bibitem{Ad} Adams R. A. : Sobolev spaces. Pure Appl. Math., Vol. 65. Academic Press [Harcourt Brace Jovanovich, Publishers], New York-London, 1975, xviii+268 pp.

\bibitem{Ba} Barb S. : Topics in geometric analysis with applications to partial differential equations. PHD dissertation, University of Missouri-Columbia, 2009.

\bibitem{CET} Caro, P. ; Ervedoza S. ; Thabouti, L. :  Quantitative unique continuation for non-regular perturbations of the Laplacian. arXiv:2411.19021. 

\bibitem{DET}  Dehman, B. ;  Ervedoza, S. ;  Thabouti, L. :  Carleman estimates for elliptic boundary value problems and applications to the quantification of unique continuation. Annales Henri Lebesgue 7 (2024), 1603-1668.

\bibitem{Ch19} Choulli, M. :  Inverse problems for Schr\"odinger equations with unbounded potentials. Notes of the course given during AIP 2019 summer school. arXiv:1909.11133.

\bibitem{Ch19} Choulli,  M. : Comments on the determination of the conductivity by boundary measurements.
J. Math. Anal Appl.  517 (2) (2023), 126638, 31 pp.  

\bibitem{Ch20}  Choulli, M. : New global logarithmic stability result for the Cauchy problem for elliptic equations, Bull. Aust. Math. Soc. 101 (1) (2020) 141-145. 


\bibitem{Ch24_2} Choulli, M. : Uniqueness of continuation for semilinear elliptic equations. Partial Differ. Equ. Appl. 5, 22 (2024).


\bibitem{Ch24_3} Choulli, M. : 
Applied functional analysis. (Analyse fonctionnelle appliqu\'ee.) 
Enseignement SUP-Maths. Les Ulis: EDP Sciences, 2024,  viii+314 pp.



\bibitem{ChLN} M. Choulli :  An introduction to the uniqueness of continuation of second order partial differential equations. Lecture notes.

\bibitem{Da} Davey, B. : Quantitative unique continuation for Schr\"odinger operators. J. Funct. Anal. 279 (2020), no. 4, 108566, 23 pp.

\bibitem{DZ1} Davey, B. ; Zhu, J. : Quantitative uniqueness of solutions to second order elliptic equations with singular potentials in two dimensions. Calc. Var. Partial Differential Equations 57 (2018), no. 3, Paper No. 92, 27 pp.

 \bibitem{DZ2} Davey, B. ; Zhu, J. : Quantitative uniqueness of solutions to second-order elliptic equations with singular lower order terms. Comm. Partial Differential Equations 44 (2019), no. 11, 1217–1251.

\bibitem{ER} ter Elst, A. F. M. ; Ruddell, K.  :  An explicit bound for the Poincaré constant on a Lipschitz domain. Banach Center Publ., 112. Polish Academy of Sciences, Institute of Mathematics, Warsaw, 2017, 87-97.

\bibitem{GL} Garofalo, N. ; Lin, F.-H. : Monotonicity properties of variational integrals, $A^p$ weights and unique continuation. Indiana Univ. Math. J. 35 (1986), no. 2, 245–268.

\bibitem{Gr} Grisvard, P. : Elliptic problems in nonsmooth domains. Monogr. Stud. Math., 24. Pitman (Advanced Publishing Program), Boston, MA, 1985, xiv+410 pp.
 
\bibitem{JK} Jerison, D. ; Kenig, C. E. : Unique continuation and absence of positive eigenvalues for Schrödinger operators. Ann. of Math. (2) 121 (1985), no. 3, 463-494.

\bibitem{Le} Leoni, G. : A first course in Sobolev spaces. Grad. Stud. Math., 181. American Mathematical Society, Providence, RI, 2017, xxii+734 pp.

\bibitem{Li}  Lin, F.-H. :  A uniqueness theorem for parabolic equations. Comm. Pure Appl. Math. 43 (1990), no. 1, 127-136.

\bibitem{Wo} Wolff, T. H. :  A property of measures in $\mathbb{R}^n$ and an application to unique continuation. Geom. Funct. Anal.  2 (2), (1992), 225-284.
 


\end{thebibliography}
\end{document}